\documentclass[11pt]{amsart}
\usepackage{tikz,amsthm,amsmath,amstext,amssymb,amscd,epsfig,euscript, mathrsfs, dsfont,pspicture,multicol,graphpap,graphics,graphicx,times,enumerate,subfig,sidecap,wrapfig,color}
\usepackage{amsmath}
\usepackage{amssymb}
\usepackage{pdfsync}

\newcommand{\M}{{\mathcal M}}       %
\newcommand{\R}{{\mathbb R}}       % Field of real numbers
\newcommand{\Z}{{\mathbb Z}}       % Ring of integer numbers
\newcommand{\DD}{{\mathcal D}}

\newcommand{\HH}{{\mathcal H}}

\newcommand{\LL}{{\mathcal L}}

\newcommand{\QQ}{{\mathcal Q}}

\newcommand{\AZ}{{\mathcal A}}

\newcommand{\GZ}{{\mathcal G}}

\newcommand{\RR}{{\mathcal R}}

\newcommand{\SSS}{{\mathcal S}}

\newcommand{\diam}{\mathop{\rm diam}}
\newcommand{\dist}{{\rm dist}}
\newcommand{\ds}{\displaystyle }

\newcommand{\interior}[1]{{\stackrel{\mbox{\scriptsize$\circ$}}{#1}}}

\newcommand{\rf}[1]{{(\ref{#1})}}

\newcommand{\supp}{\operatorname{supp}}

\newcommand{\vphi}{{\varphi}}
\newcommand{\ve}{{\varepsilon}}
\newcommand{\vu}{{\vspace{1mm}}}
\newcommand{\vv}{{\vspace{2mm}}}
\newcommand{\vvv}{\vspace{4mm}}
\newcommand{\wt}[1]{{\widetilde{#1}}}
\newcommand{\wh}[1]{{\widehat{#1}}}

\newcommand{{\pv}}{{\rm{pv}}}

\newcommand{\LD}{{\mathsf{LD}}}
\newcommand{\bad}{{\mathsf{Bad}}}

\newcommand{\sss}{{\mathsf{Stop}}}

\def\Xint#1{\mathchoice
{\XXint\displaystyle\textstyle{#1}}%
{\XXint\textstyle\scriptstyle{#1}}%
{\XXint\scriptstyle\scriptscriptstyle{#1}}%
{\XXint\scriptscriptstyle\scriptscriptstyle{#1}}%
\!\int}
\def\XXint#1#2#3{{\setbox0=\hbox{$#1{#2#3}{\int}$ }
\vcenter{\hbox{$#2#3$ }}\kern-.58\wd0}}

\def\avint{\Xint-}

\newtheorem{theorem}{Theorem}[section]
\newtheorem{lemma}[theorem]{Lemma}

\newtheorem{mlemma}[theorem]{Main Lemma}
\newtheorem{keylemma}[theorem]{Key Lemma}

\newtheorem*{lemma*}{Lemma}
\newtheorem*{theorem*}{Theorem}

\theoremstyle{definition}

\theoremstyle{remark}
\newtheorem{rem}[theorem]{\bf Remark}

\numberwithin{equation}{section}

\newcommand{\brem}{\begin{rem}}
\newcommand{\erem}{\end{rem}}

\begin{document}

\begin{abstract}
In this paper we show that if $\mu$ is a Borel measure in $\R^{n+1}$ with growth of order $n$, so that the $n$-dimensional Riesz transform $\RR_\mu$ is bounded in $L^2(\mu)$, and $B\subset\R^{n+1}$ is a ball with $\mu(B)\approx r(B)^n$ such that:
\begin{itemize}
\item[(a)] there is some $n$-plane $L$ passing through the center of $B$ such that for some $\delta>0$ small enough, it holds 
$$\int_B \frac{\dist(x,L)}{r(B)}\,d\mu(x)\leq \delta\,\mu(B),$$

\item[(b)] for some constant $\ve>0$ small enough,
$$\int_{B} |\RR_\mu1(x) - m_{\mu,B}(\RR_\mu1)|^2\,d\mu(x) \leq \ve \,\mu(B),$$
where $m_{\mu,B}(\RR_\mu1)$ stands for the mean of $\RR_\mu1$ on $B$ with respect to $\mu$,
\end{itemize}
then there exists a uniformly $n$-rectifiable set $\Gamma$, with $\mu(\Gamma\cap B)\gtrsim \mu(B)$,  and so that $\mu|_\Gamma$ is absolutely
continuous with respect to $\HH^n|_\Gamma$. This result is an essential tool to solve an old question on a two phase problem for harmonic measure in subsequent papers by Azzam, Mourgoglou, Tolsa, and Volberg.
\end {abstract}

\title[The Riesz transform and rectifiability for general Radon measures]{The Riesz transform and quantitative rectifiability for general Radon measures}

\author{Daniel Girela-Sarri\'on}
\address{Departament de Matem\`atiques\\ Universitat Aut\`onoma de Barcelona\\ 
08193 Bellaterra (Barcelona, Catalonia) }
\email{ddgirela@gmail.com}

\author{Xavier Tolsa}
\address{ICREA,  Universitat Aut\`onoma de Barcelona and BGSMath \newline  Departament de Matem\`atiques\\
08193 Bellaterra (Barcelona, Catalonia) }
\email{xtolsa@mat.uab.cat}

%\keywords{Riesz trasform, rectifiability,Harmonic measure}
\thanks{The authors were supported by the ERC grant 320501 of the European Research Council (FP7/2007-2013), and also partially supported by 2014-SGR-75 (Catalonia),
 MTM2013-44304-P, 
 MTM-2016-77635-P,  MDM-2014-044 (MICINN, Spain),  and by Marie Curie ITN MAnET (FP7-607647).}

\maketitle

\tableofcontents
% ***************************************************************************

\section{Introduction}

In the work \cite{NToV} it was shown that, given and $n$-AD-regular measure in $\R^{n+1}$,
the $L^2(\mu)$ boundedness of the $n$-dimensional Riesz transform implies the uniform
$n$-rectifiability of $\mu$. In the codimension $1$ case, this result solved a long standing problem raised by David and Semmes
\cite{DS1}. In the present paper we obtain a related quantitative
result valid for general Radon measures in $\R^{n+1}$ with growth of order $n$. Our result turns out to be
an essential tool for the solution of an old question on harmonic measure obtained in 
recent works by Azzam, Mourgoglou, Tolsa \cite{AMT} and Azzam, Mourgoglou, Tolsa, Volberg \cite{AMTV}.

To state our main theorem in detail we need to introduce some notation and terminology.
Let $\mu$ be a Radon measure in $\R^{n+1}$. We say that $\mu$ has growth of order $n$ (with constant $C_0$) if
$$\mu(B(x,r))\leq C_0\,r^n\quad\mbox{ for all $x\in\R^{n+1}$ and all $r>0$.}$$
A measure $\mu$ is called $n$-AD-regular (or just AD-regular or Ahlfors-David regular) if there exists some
constant $C>0$ such that
$$C^{-1}r^n\leq \mu(B(x,r))\leq C\,r^n\quad \mbox{ for all $x\in
\supp(\mu)$ and $0<r\leq \diam(\supp(\mu))$.}$$

Given a signed Radon measure $\nu$ in $\R^{n+1}$ we consider the $n$-dimensional Riesz
transform
$$\RR\nu(x) = \int \frac{x-y}{|x-y|^{n+1}}\,d\nu(y),$$
whenever the integral makes sense. For $\ve>0$, its $\ve$-truncated version is given by 
$$\RR_\ve \nu(x) = \int_{|x-y|>\ve} \frac{x-y}{|x-y|^{n+1}}\,d\nu(y).$$
For a positive Radon measure $\mu$ and a function $f\in L^1_{loc}(\mu)$, we denote
$\RR_\mu f\equiv \RR(f\mu)$ and $\RR_{\mu,\ve} f\equiv \RR_\ve(f\mu)$.
We say that the Riesz transform $\RR_\mu$ is bounded in $L^2(\mu)$ if the truncated operators 
$\RR_{\mu,\ve}:L^2(\mu)\to L^2(\mu)$ are bounded uniformly on $\ve>0$.

If  $\mu$ is a measure with no point masses such that $\RR_\mu$ is bounded in $L^2(\mu)$ and 
$\nu$ is a finite Radon measure, the principal
value
$$\pv \RR\nu(x) = \lim_{\ve\to0}\RR_\ve\nu(x)$$
exists $\mu$-a.e. This follows easily from the results of \cite{NToV-pubmat}, arguing as in 
\cite[Chapter 8]{Tolsa-llibre} with the Cauchy transform replaced by the Riesz transform. See Section \ref{secbmo} for the more details.
Abusing notation, we also write $\RR\nu(x)$ instead of $\pv\RR\nu(x)$.

For $f\in L^1_{loc}(\mu)$ and $A\subset \R^{n+1}$ with $\mu(A)>0$, we consider the $\mu$-mean 
$$m_{\mu,A}(f) = \,\avint_A f\,d\mu = \frac1{\mu(A)}\int_A f\,d\mu.$$
Also, given a ball $B\subset\R^{n+1}$ and an $n$-plane $L$ in $\R^{n+1}$, we denote
$$\beta_{\mu,1}^L(B) = \frac1{r(B)^n}\int_B \frac{\dist(x,L)}{r(B)}\,d\mu(x),$$
where $r(B)$ stands for the radius of $B$.
In a sense, this coefficient measures how close the points from $\supp\mu$ are to the $n$-plane $L$ in the ball $B$.

We also set
$$\Theta_\mu(B) = \frac{\mu(B)}{r(B)^n},\qquad P_\mu(B) = \sum_{j\geq0} 2^{-j}\,\Theta_\mu(2^jB).$$
So $\Theta_\mu(B)$ is the $n$-dimensional density of $\mu$ on $B$ and $P_\mu(B)$ is some kind of smoothened version of this density.

\vv
Our main theorem is the following.

\begin{theorem}\label{teo0}
Let $\mu$ be a Radon measure on $\R^{n+1}$ and $B\subset \R^{n+1}$ a ball so that the following conditions
hold:
\begin{itemize}

\item[(a)] For some constant $C_0>0$, $C_0^{-1}r(B)^n\leq \mu(B)\leq C_0\,r(B)^n$.

\item[(b)]  $P_\mu(B) \leq C_0$, and $\mu(B(x,r))\leq C_0\,r^n$ for all $x\in B$ and $0<r\leq r(B)$.

\item[(c)] There is some $n$-plane $L$ passing through the center of $B$ such that for some $0<\delta\ll 1$, it holds $\beta_{\mu,1}^L(B)\leq \delta$.

\item[(d)] $\RR_{\mu|_B}$ is bounded in $L^2(\mu|_{B})$ with $\|\RR_{\mu|_B}\|_{L^2(\mu|_{B})\to L^2(\mu|_{B})}\leq C_1$, and $\RR (\chi_{2B}\mu)\in L^2(\mu|_B)$.

\item[(e)] For some constant $0<\ve\ll1$,
$$\int_{B} |\RR\mu(x) - m_{\mu,B}(\RR\mu)|^2\,d\mu(x) \leq \ve \,\mu(B).$$
\end{itemize}
Then there exists some constant $\tau>0$ such that if $\delta,\ve$ are small enough (depending on $C_0$ and $C_1$),
then there is a uniformly $n$-rectifiable set $\Gamma\subset\R^{n+1}$ such that
$$\mu(B\cap \Gamma)\geq \tau\,\mu(B).$$
The constant $\tau$ and the UR constants of $\Gamma$ depend on all the constants above.
\end{theorem}

\vv

Some remarks are in order. First, we mention that the notion of uniform $n$-rectifiability 
will be introduced in Section \ref{secprelim}. For the moment, for the reader's convenience,
let us say that this a quantitative (and stronger) version of the notion of $n$-rectifiability.
The UR constants are just the constants involved in the definition of  uniform $n$-rectifiability.
We also remark that it is immediate to check that the condition (b) above holds, for example, if
$\mu$ has growth of order $n$ (with constant $\frac12C_0$). The statement in (b) which involves $P_\mu(B)$ is somewhat more general 
and it is more convenient for applications. Finally, we warn the reader that in the case that
$\mu$ is not a finite measure, the statement (e) should be understood in a BMO sense. The fact that $P_\mu(B)<\infty$ by the assumption (b) guarantees that $\RR\mu(x) - m_{\mu,B}(\RR\mu)$
is correctly defined. See again Section \ref{secbmo} for more details.

Note that, in particular, the theorem above ensures the existence of some piece of positive $\mu$-measure of $B$ where $\mu$ and the Hausdorff measure $\HH^n$ are mutually absolutely continuous on some subset of $\Gamma$. This fact, which at first sight may appear rather surprising, is one of the main difficulties for the proof of this result.

It is worth comparing Theorem \ref{teo0} to L\'eger's theorem on Menger curvature.
Given three  points $x,y,z\in\R^2$, their Menger curvature is
$$c(x,y,z) = \frac1{R(x,y,z)},$$
where $R(x,y,z)$ is the radius of the circumference passing through $x,y,z$ if they are pairwise different, and $c(x,y,z) =0$ otherwise. The curvature of $\mu$ is defined by
$$c^2(\mu) = \iiint c(x,y,z)^2\,d\mu(x)\,d\mu(y)\,d\mu(z).$$
This notion was first introduced by Melnikov \cite{Melnikov} when studying analytic capacity and, modulo an ``error term'', is comparable to the squared $L^2(\mu)$ norm of the Cauchy transform
of $\mu$ (see \cite{MV}). One of the main ingredients of the proof of Vitushkin's conjecture 
for removable singularities for bounded analytic functions by
Guy David \cite{David-vitus} is L\'eger's theorem \cite{Leger} (sometimes called also David-L\'eger theorem).
The quantitative version of this theorem asserts that if $\mu$ is a Radon measure in $\R^2$ with growth of order $1$ and $B$ a ball such
that $\mu(B)\approx r(B)$, and further $c^2(\mu|_B)\leq \ve\,\mu(B)$ for some 
$\ve>0$ small enough, then there exists some (possibly rotated) Lipschitz graph $\Gamma\subset\R^2$ 
such that $\mu(B\cap\Gamma)\geq \frac9{10}\mu(B)$. In particular, as in Theorem \ref{teo0}, it follows that a big piece of
$\mu|_B$ is mutually absolutely continuous with respect to $\HH^1$ on some subset of $\Gamma$.
In a sense, one can think that Theorem \ref{teo0} is an analogue
for Riesz transforms of the quantitative L\'eger theorem for Menger curvature.
Indeed, the role of the assumption (e) in Theorem \ref{teo0} is played by the condition $c^2(\mu|_B)\leq \ve\,\mu(B)$. Further,  it is not difficult to check that this condition implies 
that there exists some line $L$ such that $\beta_{\mu,1}^L(B)\leq \delta\,\mu(B)$, with $\delta=\delta(\ve)\to 0$ as $\ve\to0$, analogously to the assumption (c) of Theorem \ref{teo0}. 

On the other hand, from the theorem of L\'eger described above, it follows easily that if $\HH^1(E)<\infty$ and $c^2(\HH^1|_E)<\infty$, then $E$ is $1$-rectifiable. The analogous result 
in the codimension $1$ case in $\R^{n+1}$ (proved in \cite{NToV-pubmat}) asserts if $E\subset \R^{n+1}$, $\HH^n(E)<\infty$, $\HH^n|_E$ has growth of order $n$,
and
$\|\RR(\HH^n|_E)\|_{L^2(\HH^n|_E)}<\infty$, then $E$ is $n$-rectifiable. However, as far as we know, this cannot be proved easily using Theorem \ref{teo0}.

The proof of Theorem \ref{teo0} is substantially different from the one of L\'eger's theorem. When estimating
the $L^2(\mu)$ norm of $\RR\mu$ 
we are dealing with a singular integral and there may be cancellations among different scales. So the situation is more delicate than in the case of the curvature $c^2(\mu)$, which is defined by a non-negative integrand (namely, the squared Menger 
curvature of three points).

To prove Theorem \ref{teo0} we will apply some of the techniques developed in \cite{ENV} and
\cite{NToV}. In particular, by using a variational argument 
we will estimate from below the $L^2(\mu)$ norm of the Riesz transform of a
 suitable periodization of a smoothened version of the measure $\mu$ restricted to some appropriate 
cube $Q_0$. The assumption that 
$\beta_{\mu,1}^L(B)\leq \delta$ in (c) is necessary for technical reasons, and we do not know if
the theorem holds without this condition. For a more detailed scheme of the proof of the theorem we refer the reader 
to Section \ref{secscheme}.

\vv
Finally we are going to announce  the aforementioned result on harmonic measure %from \cite{AMT} and \cite {AMTV} 
whose proof uses
Theorem \ref{teo0} as an essential tool. 
%For simplicity, we will only announce this for domains $\Omega_1,\Omega_2\subset\R^{n+1}$ satisfying the condition
%\begin{equation}\label{eqqqu1}
%\HH^s_\infty((\R^{n+1}\setminus\Omega_i) \cap B(x,r))\approx r^s\quad \mbox{ for all $x\in\partial\Omega_i$ and $0<r\leq r_0$,}
%\end{equation}
%for some fixed $s\in (n-1,n+1]$, $r_0>0$, where $\HH^s_\infty$ stands for the $s$-dimensional Hausdorff content. For example, the so-called NTA domains introduced by Jerison and Kenig 
%\cite{JK} satisfy this condition, and also the simply connected domains in the plane.

\begin{theorem}\label{t:main}
For $n\geq 2$, let $\Omega_1,\Omega_2\subset \R^{n+1}$ be two domains and denote by 
$\omega_1$ and $\omega_2$ their respective harmonic measures. Let $E\subset \partial \Omega_1\cap
\partial \Omega_2$ 
 be a Borel set such that $\omega_{1}|_E\ll\omega_{2}|_E\ll\omega_{1}|_E$. Then $E$ contains an $n$-rectifiable subset $F$ with $\omega_1(E\setminus F)=0$ such that
$\omega_1|_F$ and $\omega_2|_F$ are mutually absolutely continuous with respect to the Hausdorff measure $\HH^{n}|_F$. 
\end{theorem}

This theorem has been proved first by
Azzam, Mourgoglou, and Tolsa in \cite{AMT} under the additional assumptions that $\partial\Omega^1=\partial\Omega^2$ and that both $\Omega_1$ and $\Omega_2$ satisfy the so-called capacity density condition. The final version stated above is due to 
Azzam, Mourgoglou, Tolsa and Volberg and has appeared in \cite{AMTV}. Note that Theorem \ref{t:main}  does not assume
that either the boundaries of $\Omega_1$, $\Omega_2$, or the set $E$ itself have locally finite Hausdorff measure 
$\HH^n$. This is the main difficulty in the proof of the theorem, which is solved by using the connection between harmonic measure and Riesz transforms and applying Theorem \ref{teo0}.

Up to now the result stated in Theorem \ref{t:main} was known only in the case when $\Omega_1$, $\Omega_2$ are planar
domains, by results of Bishop, Carleson, Garnett and Jones \cite{BCGJ} and Bishop \cite{Bishop}, and it was an open problem
to extend this to higher dimensions (see Conjecture 8 in \cite{Bishop-conjectures}). 
For a partial result for NTA domains in the higher dimensional case, see the nice work \cite{KPT} by 
Kenig, Preiss, and Toro.

\vvv

% ***************************************************************************

\section{Preliminaries}\label{secprelim}

\subsection{Notation}

In this paper  we will use the letters $c,C$ to denote
constants (quite often absolute constants, perhaps depending on $n$) which may change their values at different
occurrences. On the other hand, constants with subscripts, such as $C_1$, do not change their values
at different occurrences.

We will write $a\lesssim b$ if there is $C>0$ so that $a\leq Cb$ and $a\lesssim_{t} b$ if the constant $C$ depends on the parameter $t$. We write $a\approx b$ to mean $a\lesssim b\lesssim a$ and define $a\approx_{t}b$ similarly. 

%For sets $A,B\subset \R^{n+1}$, we let \[\dist(A,B)=\inf\{|x-y|:x\in A,y\in B\}, \;\; \dist(x,A)=\dist(\{x\},A),\]
We denote the open ball of radius $r$ centered at $x$ by $B(x,r)$. For a ball $B=B(x,r)$ and $a>0$ we write $r(B)$ for its radius and $a B=B(x,a r)$. The notation $A(x,r_1,r_2)$ stands
for an open annulus centered at $x$ with inner radius $r_1$ and outer radius $r_2$.

\subsection{Cubes and densities}

Given a cube $Q$, we denote by $\ell(Q)$ its side length.  Unless otherwise stated, we assume that its sides are parallel to
the coordinate axes of $\R^{n+1}$. By $a Q$ we denote a cube concentric with $Q$
with side length $a\ell(Q)$. We write
$$\Theta_\mu(Q) = \frac{\mu(Q)}{\ell(Q)^n},\qquad P_\mu(Q) = \sum_{j\geq0} 2^{-j}\,\Theta_\mu(2^jQ).$$

We also consider the pointwise lower and upper $n$-dimensional densities
$$\Theta_*^n(x,\mu) = \liminf_{r\to0+}\frac{\mu(B(x,r))}{(2r)^n},\qquad
\Theta^{n,*}(x,\mu) = \limsup_{r\to0+}\frac{\mu(B(x,r))}{(2r)^n}.$$

\subsection{Rectifiability and uniform rectifiability}

A set $E\subset \R^d$ is called $n$-rectifiable if there are Lipschitz maps
$f_i:\R^n\to\R^d$, $i=1,2,\ldots$, such that 
\begin{equation}\label{eq001}
\HH^n\biggl(E\setminus\bigcup_i f_i(\R^n)\biggr) = 0,
\end{equation}
where $\HH^n$ stands for the $n$-dimensional Hausdorff measure. 
Also, one says that 
a Radon measure $\mu$ on $\R^d$ is $n$-rectifiable if $\mu$ vanishes out of an $n$-rectifiable
set $E\subset\R^d$ and moreover $\mu$ is absolutely continuous with respect to $\HH^n|_E$.

A measure $\mu$ in $\R^d$ is  uniformly  $n$-rectifiable if it is 
$n$-AD-regular and
there exist $\theta, M >0$ such that for all $x \in \supp(\mu)$ and all $r>0$ 
there is a Lipschitz mapping $g$ from the ball $B_n(0,r)$ in $\R^{n}$ to $\R^d$ with $\text{Lip}(g) \leq M$ such that$$
\mu (B(x,r)\cap g(B_{n}(0,r)))\geq \theta r^{n}.$$
In the case $n=1$, it is known that $\mu$ is uniformly $1$-rectifiable if and only if $\supp(\mu)$ is contained in a rectifiable curve $\Gamma$ in $\R^d$ such that the arc length measure on $\Gamma$ is $1$-AD-regular.

A set $E\subset\R^d$ is called $n$-AD-regular if $\HH^n|_E$ is $n$-AD-regular, and it is called
uniformly $n$-rectifiable if $\HH^n|_E$ is uniformly  $n$-rectifiable.

\subsection{Riesz transforms}\label{secbmo}

We denote by $K$ the kernel of the $n$-dimensional Riesz transform. That is,
$$K(x) = \frac{x}{|x|^{n+1}}.$$

As mentioned in the Introduction, if $\mu$ is $n$-AD-regular and $\RR_\mu$ is bounded in $L^2(\mu)$, then
it has been shown in \cite{NToV} that $\mu$ must be uniformly $n$-rectifiable.
If we do not assume $\mu$ to be $n$-AD-regular and instead we just suppose that $\mu$ has no point masses, then the $L^2(\mu)$ boundedness of $\RR_\mu$ implies that
$\mu$ satisfies the growth condition
\begin{equation*}%\label{eqcreix901}
\mu(B(x,r))\leq C_0\,r^n\quad\mbox{ for all $x\in\supp\mu$, $r>0$.}
\end{equation*}
See \cite[p.\ 56]{David-wavelets}, for example.
Also, in this case, it has been proved in \cite{NToV-pubmat} that %the $L^2(\mu)$ boundedness of $\RR_\mu$ implies that 
$\mu$ is of the form
$$\mu = \mu_r + \mu_0,$$
where $\mu_r$ is $n$-rectifiable and $\mu_0$ has vanishing $n$-dimensional upper density $\mu_0$-a.e.\ (i.e., 
$\Theta^{n,*}(x,\mu_0)=0$ for $\mu_0$-a.e.\ $x\in\R^{n+1}$). This implies that,
for any finite Radon measure $\nu$,  the principal values $\RR\nu(x)\equiv \pv\RR\nu(x)$ exists 
for $\mu$-a.e.\ $x\in\R^{n+1}$. See for example the detailed arguments in
\cite[Chapter 8]{Tolsa-llibre} for the case of the Cauchy transform. 
 
\vv
Now let $\mu$ and $B$ be as in Theorem \ref{teo0}.
The fact that $P_\mu(B)<\infty$ in Theorem \ref{teo0} guaranties that $\RR\mu(x)- \RR\mu(y)$ is defined for $\mu$-a.e.\
$x,y\in B$ in a ``BMO sense''.  This means the following. We write, by definition,
\begin{equation}\label{eqbmo*1}
\RR\mu(x) -\RR\mu(y):= \RR(\chi_{2B}\mu)(x) - \RR(\chi_{2B}\mu)(y)  + 
\int_{\R^{n+1}\setminus 2B} \bigl[K(x-z) -K(y-z)\bigr]\,d\mu(z)
.
\end{equation}
The discussion in the previous paragraph ensures the existence of the principal values $\RR(\chi_{2B}\mu)(x)$, $\RR(\chi_{2B}\mu)(y)$ for $\mu$-a.e.\ $x,y\in B$,
because of the $L^2(\mu|_B)$ boundedness of $\RR_{\mu|_B}$. On the other hand,
by standard estimates, it is immediate to check that
\begin{equation}\label{eqbmo*1.5}
\int_{\R^{n+1}\setminus 2B} \bigl|K(x-z) -K(y-z)\bigr|\,d\mu(z)\lesssim P_\mu(B)<\infty,
\end{equation}
and thus the integral on the right hand side of \rf{eqbmo*1}
makes sense too. 

Given $x\in B$ and a subset $F\subset B$ with $\mu(F)>0$,  we write
\begin{equation}\label{eqbmo*2}
\RR\mu(x) - m_{\mu,F}(\RR\mu) := \frac1{\mu(F)}\int_F \bigl(\RR\mu(x) - \RR\mu(y)\bigl)\,d\mu(y).
\end{equation}
From the definition \rf{eqbmo*1} and the estimate \rf{eqbmo*1.5}, we deduce that 
$$\frac1{\mu(F)}\int_F \bigl|\RR\mu(x) - \RR\mu(y)\bigl|\,d\mu(y)
\leq |\RR(\chi_{2B}\mu)(x)| + \frac1{\mu(F)}\int_F \bigl|\RR(\chi_{2B}\mu)(y)\bigl|\,
d\mu(y) + C\,P_\mu(B).$$
Recall that $\RR (\chi_{2B}\mu)\in L^2(\mu|_B)$ by (e) in Theorem \ref{teo0}, and thus 
$\int_F \bigl|\RR(\chi_{2B}\mu)(y)\bigl|\,
d\mu(y)<\infty$. Hence, the integral on the right hand side of \rf{eqbmo*2} makes sense for all
$x$ for which the principal value $\RR (\chi_{2B}\mu)(x)$ exists (in particular, for $\mu$-a.e.\ 
$x\in B$). Further, the preceding estimate also shows that 
$$\RR\mu(x) - m_{\mu,F}(\RR\mu)\in L^2(\mu|_B)$$
for any set $F$ with $\mu(F)>0$, and thus the integral
$$\int_{B} |\RR\mu(x) - m_{\mu,B}(\RR\mu)|^2\,d\mu(x)$$
in the assumption (e) of Theorem \rf{teo0} is finite.
\vv

Finally we remark that the hypothesis that $\RR (\chi_{2B}\mu)\in L^2(\mu|_B)$ in (e) of Theorem \ref{teo0} is superfluous. Indeed, the assumption that $\RR_{\mu|_{B}}$ is bounded in $L^2(\mu|_{B})$
ensures that $\RR(\chi_B\mu)\in L^2(\mu|_{B})$. Also,
from the fact that $\mu(B(x,r))\leq C_0\,r^n$ for all $x\in B$, one can deduce
that 
$\RR_{\mu|_{2B\setminus B}}: L^2(\mu|_{2B\setminus B})\to L^2(\mu|_{B})$ is bounded
(see \cite[Lemma 3]{Verdera-Arkiv} for a related argument, for example), and thus 
$\RR(\chi_{2B\setminus B}\mu)\in L^2(\mu|_{B})$ too. However, to avoid technicalities we have preferred
to write the theorem as above and so we skip the detailed arguments.

\subsection{Scheme of the proof of Theorem \ref{teo0}}\label{secscheme}

In the first step of the proof we find a suitable cube $Q_0$ deep inside the ball $B$ which contains some significant 
portion 
of the measure $\mu$ and so that the $\mu$ is very flat in a big neighborhood of $Q_0$, with flatness  measured in terms of some coefficients $\alpha$ involving some variant of the Wasserstein distance $W_1$. 
%Further, it holds that $$\int_{Q_0} |\RR\mu(x) - m_{\mu,Q_0}(\RR\mu)|^2\,d\mu(x) \leq \ve \,\mu(Q_0).$$
This task is carried out in Section~\ref{sec3}.

In Section \ref{sec4} we obtain a localized version of the BMO type estimate in (e) of Theorem \ref{teo0} which is more handy
and will be necessary later. Namely, we show that
$\int_{Q_0}|\RR_\mu \chi_{AQ_0}|^2\,d\mu$ is very small if $A$ is big enough and $\mu$ is flat enough in $3AQ_0$.

In Section \ref{secdm} we recall the properties of the dyadic cells of the David-Mattila lattice, which we will associate to the measure $\mu|_{Q_0}$. The David-Mattila lattice is very appropriate for the study of Calder\'on-Zygmund theory and quantitative rectifiability for non-doubling measures. This was first introduced in \cite{David-Mattila} in connection with Vitushkin's conjecture 
for Lipschitz harmonic functions.

The main point of the proof of Theorem \ref{teo0} consists in showing that there exists a subset $F\subset Q_0$ with
a significant proportion of the measure $\mu$ of $Q_0$ (and thus of $B$) so that
\begin{equation}\label{eqAD*}
\mu(B(x,r))\approx r^n\quad\mbox{ for all $x\in F$, $0<r\leq \ell(Q_0)$.}
\end{equation}
In particular, this tells us that $\mu|_F = h\,\HH^n|_F$ for some function $h\approx1$. The existence of such set $F$ is ensured 
by the Key Lemma \ref{keylemma}, which is proved in Sections \ref{sec6}-\ref{sec9}.

We prove the Key Lemma by contradiction.
 So we assume  that there is a family of low density cells $Q\in \LD$ from the David-Mattila lattice (with $\mu(Q)\ll\ell(Q)^n$)
which cover all of $Q_0$ with the possible exception of a remaining set of measure $\mu$ smaller than $\ve_0\,\mu(Q_0)$.
There is a long and technical part of the arguments which, roughly speaking, consists in replacing the measure $\mu$ by a suitable periodized and smoothened measure $\eta$ which approximates $\mu$ on $Q_0$ and is absolutely continuous with respect to Lebesgue measure, and so that $\int_{Q_0}|\RR \eta|^2\,d\eta$ is very small. This is done in several steps in Section \ref{sec6}-\ref{sec8}.

The proof of the Key Lemma is completed in Section \ref{sec9} by estimating $\int_{Q_0}|\RR \eta|^2\,d\eta$ from below by means of a variational argument. Some related variational arguments have appeared in \cite{ENV} and \cite{NToV}. The 
argument in \cite{ENV} is applied to a compactly supported measure and does not require any periodization like the one in our paper. On the other hand, the argument in \cite{NToV} uses a reflection trick instead of  periodization. The reflection trick is appropriate when working with the components of $\RR$ which are parallel to the approximating plane $L$, but this does not work when dealing with all the components of $\RR$ as in our situation, as far as we know.

In the final Section \ref{sec10} we show that the set $F$ mentioned above has a big piece which is contained in a uniformly $n$-rectifiable set $\Gamma$. To this end, using \rf{eqAD*} and a suitable covering argument, we 
construct an $n$-AD-regular measure $\zeta$ which coincides with $\mu$ on a big piece $\wt F\subset F$, and so that
moreover $\RR_\zeta$ is bounded in $L^2(\zeta)$. Then, by \cite{NToV} it turns out that $\Gamma:=\supp\zeta$ is
uniformly $n$-rectifiable and the proof of Theorem \ref{teo0} is concluded.

\vvv

% ***************************************************************************

\section{The Main Lemma}\label{sec3}

In this section we show that, under the assumptions of Theorem \ref{teo0},  there is 
 a suitable cube $Q_0$ deep inside the ball $B$ with  $\mu(Q_0)\approx\mu(B)$
 and so that the $\mu$ is very flat in  $3AQ_0$, for some big $A\gg1$. We measure flatness in terms of some coefficients $\alpha$ involving some variant of the Wasserstein distance $W_1$. 
%Further, it holds that $$\int_{Q_0} |\RR\mu(x) - m_{\mu,Q_0}(\RR\mu)|^2\,d\mu(x) \leq \ve \,\mu(Q_0).$$
This allows us to reduce Theorem \ref{teo0} to the proof of the Main Lemma \ref{mainlemma} below.

\subsection{Preliminaries and statement of the Main Lemma}

Given two real Radon measure $\mu$ and $\sigma$ in $\R^{n+1}$ and a cube $Q\subset\R^{n+1}$, we set
\begin{equation}\label{eqalpha0}
d_Q(\mu,\sigma) = \sup_f \int f\,d(\mu-\sigma),
\end{equation}
where the supremum is taken over all the $1$-Lipschitz functions supported on $Q$. 
Given an $n$-plane $L$ in $\R^{n+1}$, then we denote 
$$\alpha_\mu^L(Q) = \frac1{\ell(Q)^{n+1}}\inf_{c\in\R} d_Q(\mu,c\,\HH^n|_L).$$
It is immediate to check that if $\mu$ is a positive measure, we can assume $c\geq0$ in the infimum
above.

Given $t>0$, we say that $Q$ has $t$-thin boundary if
$$\mu\left(\left\{x\in 2Q:\dist(x,\partial Q)\leq\lambda\,\ell(Q)\right\}\right) \leq t\,\lambda\,\mu(2Q)\quad
\mbox{ for all $\lambda>0$.}$$
It is well known that for any given cube $Q_0\subset\R^{n+1}$ and $a>1$, there exists 
another cube $Q$ with $t$-thin boundary such that $Q_0\subset Q\subset aQ_0$, with $t$ depending
just on $n$ and $a$.
For the proof, we refer the reader to Lemma 9.43 of \cite{Tolsa-llibre}, for example.

\vv

\begin{mlemma} \label{mainlemma}
Let $n\geq1$ and let $C_0,C_1>0$ be some arbitrary constants. There exists $A=A(C_0,C_1,n)>10$ big enough and $\ve=\ve(C_0,C_1,n)>0$ small enough such that if $\delta=\delta(A,C_0,C_1,n)>0$ is small enough, then the following holds.
Let $\mu$ be a Radon measure in $\R^{n+1}$ and $Q_0\subset \R^{n+1}$ a cube centered at the origin satisfying the following 
properties:\vu

\begin{itemize}
\item[(a)] $\mu(Q_0)=\ell(Q_0)^n$.\vu

\item[(b)] $P_\mu(AQ_0) \leq C_0$.\vu

\item[(c)] For all $x\in AQ_0$ and $0<r\leq A\ell(Q_0)$, $\Theta_\mu(B(x,r))\leq C_0$.\vu
 
 \item[(d)] $Q_0$ has $C_0$-thin boundary. \vu

\item[(e)] $\alpha_\mu^H(3AQ_0)\leq \delta$, where $H=\{x\in\R^{n+1}:x_{n+1}=0\}$.\vu

\item[(f)] $\RR_{\mu|_{2Q_0}}$ is bounded in $L^2(\mu|_{2Q_0})$ with $\|\RR_{\mu|_{2Q_0}}\|_{L^2(\mu|_{2Q_0})\to L^2(\mu|_{2Q_0})}\leq C_1$.\vu

\item[(g)] We have
$$\int_{Q_0} |\RR\mu(x) - m_{\mu,Q_0}(\RR\mu)|^2\,d\mu(x) \leq \ve \,\mu(Q_0).$$\vu

\end{itemize}
Then there exists some constant $\tau>0$ and a uniformly $n$-rectifiable set $\Gamma\subset\R^{n+1}$ such that
$$\mu(Q_0\cap \Gamma)\geq \tau\,\mu(Q_0).$$
The constant $\tau$ and the UR constants of $\Gamma$ depend on all the constants above.
\end{mlemma}

Note that the conditions (a) and (c) in the Main Lemma imply that $\mu(2Q_0)\lesssim C_0\,\mu(Q_0)$.

\vv

\subsection{Reduction of Theorem \ref{teo0} to the Main Lemma}

Assume that the Main Lemma is proved. Then in order to prove Theorem \ref{teo0} it is enough to show the following.

\begin{lemma}\label{lemreduc}
Let $\mu$ and $B\subset\R^{n+1}$ satisfy the assumptions of Theorem \ref{teo0} with constants $C_0$, $C_1$, $\delta$, and $\ve$. For all $A'\ge10$ and all
$\delta',\ve'>0$, if $\delta$ and $\ve$ are small enough, there exists a cube $Q_0$ satisfying:\vu

\begin{itemize}
\item[(a)] $3A'Q_0\subset B$ and $\dist(A'Q_0,\partial B) \geq C_0'^{-1}\,r(B)$, with $C_0'$ depending only on $C_0$ and $n$.\vu

\item[(b)] For some constant $\gamma=\gamma(\delta')>0$, $\gamma\,r(B)\leq \ell(Q_0)\leq A'^{-1}\,r(B)$.\vu

\item[(c)] $\mu(Q_0)\geq C_0'^{-1}\ell(Q_0)^n$.\vu

\item[(d)] $Q_0$ has $C_0'$-thin boundary. \vu

\item[(e)] $\alpha_\mu^L(3A'Q_0)\leq\delta'$, where $L$ is some $n$-plane that passes through the center of $Q_0$ and is parallel to one of the faces.\vu

\item[(f)] $\ds\int_{Q_0} |\RR\mu(x) - m_{\mu,Q_0}(\RR\mu)|^2\,d\mu(x) \leq \ve' \,\mu(Q_0)$.\vu

\end{itemize}
\end{lemma}
\vv

Before proving this lemma, we show how this is used to reduce Theorem \ref{teo0} to the Main Lemma.\vv

\begin{proof}[\bf Proof of Theorem \ref{teo0} using Lemma \ref{lemreduc} and the Main Lemma \ref{mainlemma}]
Let $B\subset\R^{n+1}$ be some ball satisfying the assumptions of Theorem \ref{teo0} with constants $C_0$, $C_1$, $\delta$, and $\ve$. Let $Q_0$ be the cube given by Lemma \ref{lemreduc}, for some constants $A'\ge10$ and 
$\delta',\ve'>0$ to be fixed below.

We just have to check that the assumptions (a)-(g) of the Main Lemma are satisfied by the measure
$$\wt\mu = \frac{\ell(Q_0)^n}{\mu(Q_0)}\,\mu$$
if $A'$ is big enough and $\delta',\ve'$ are small enough.
Obviously, the assumption (a) from the Main Lemma is satisfied by $\wt\mu$.

To show that (b) in the Main Lemma holds (with a constant different from $C_0$), note first that
\begin{equation}\label{eqobv27}
\frac{\ell(Q_0)^n}{\mu(Q_0)}\approx 1,
\end{equation}
with the implicit constant depending on $C_0$ and $C_0'$. Indeed, from the assumption (c) in Theorem \ref{teo0},
$\mu(Q_0)\lesssim C_0 \ell(Q_0)^n$, and by (c) in Lemma \ref{lemreduc}, $\mu(Q_0)\geq C_0'^{-1}\ell(Q_0)^n$.
Then we have
$$P_{\wt\mu}(A'Q_0)\lesssim P_\mu(A'Q_0)= \sum_{j\geq0: 2^jA'Q_0\subset B }2^{-j}\Theta_\mu(2^jA'Q_0) + 
\sum_{j\geq0: 2^jA'Q_0\not\subset B }2^{-j}\Theta_\mu(2^jA'Q_0).$$
The first sum on the right hand side does not exceed $C\,C_0'$ because $\Theta_\mu(2^jA'Q_0)\lesssim C_0$ for all cubes
$2^jA'Q_0$ contained in $B$. Also, one can check that the last sum is bounded by $C\,P_\mu(B)$ because $\ell(2A'^jQ_0)\gtrsim
r(B)$ for all the $j$'s in this sum, taking into account that $\dist(A'Q_0,\partial B) \geq C_0'^{-1}\,r(B)$.

The condition (c) in the Main Lemma is a consequence of the facts that $3A'Q_0\subset B$, and
$\wt\mu(B(x,r))\leq C_0\,r^n$ for all $x\in B$ and $0<r\leq r(B)$ (by (b) in Theorem \ref{teo0} and \rf{eqobv27}).
So it suffices to take $A'$ big enough. 

The assumptions (d), (e), (g) in the Main Lemma are obviously satisfied too because of \rf{eqobv27} and the respective conditions (d), (e), (f) in Lemma \ref{lemreduc},
with somewhat different constants $C_0'',\delta'',\ve''$ replacing $C_0,\delta,\ve$. Finally, (f) in the Main Lemma is an immediate consequence of (d) in Theorem \ref{teo0} and the fact that 
$2Q_0\subset  3A'Q_0\subset B$.
\end{proof}

\vv

\subsection{The proof of Lemma \ref{lemreduc}}

Below we identify $\R^n$ with the horizontal $n$-plane $H=\{x\in\R^{n+1}:x_{n+1}=0\}$. Then, given a Radon measure $\sigma$
in $\R^n$ and a cube $Q\subset\R^n$, we denote
\begin{equation}\label{eqsigma99}
\alpha_\sigma^{\R^n}(Q) = \frac1{\ell(Q)^{n+1}}\inf_{c\geq0} d_Q(\sigma,c\,\HH^n|_{\R^n}),
\end{equation}
where $d_Q$ is defined as in \rf{eqalpha0} (although now $Q\subset\R^n$ instead of 
$Q\subset\R^{n+1}$) and 
the infimum runs over all constants $c>0$. Note that
$$\alpha_\sigma^{\R^n}(Q)\approx \alpha_\sigma^H(\wh Q),$$
where $\wh Q= Q\times[-\ell(Q)/2,\ell(Q)/2]$. This follows easily from the fact that any $1$-Lipschitz function supported in $Q$ can be extended to a $C$-Lipschitz function supported in $\wh Q$, with $C\lesssim1$.

We need a couple of auxiliary results. The first one is the following.

\begin{lemma} \label{lemf1}
Suppose that $\sigma$ is some finite Radon measure supported in $\R^n$ such that $d\sigma(x) = \rho(x)\,d\HH^n|_{\R^n}$,
with $\|\rho\|_\infty<\infty$.
Then, for every $R\in\DD(\R^n)$ we have
$$\sum_{Q\in\DD(\R^n):Q\subset R} \alpha^{\R^n}_\sigma(3Q)^2\ell(Q)^n \lesssim \|\rho\|_\infty^2 \ell(R)^n.$$
\end{lemma}

In this lemma, $\DD(\R^n)$ stands for the family of the usual dyadic cubes in $\R^n$.

We will derive this lemma from a related result for $n$-AD-regular measures supported on Lipschitz graphs
from \cite{Tolsa-plms},
although the lemma may be proved by more straightforward arguments.

\begin{proof}
Note first that, by replacing $\sigma$ by $\|\rho\|_\infty^{-1}\,\sigma$, we may assume $\|\rho\|_\infty=1$
in the lemma.

Consider now the auxiliary measure   measure 
$\wt\sigma = 2\,\HH^n|_{\R^n} + \sigma$. The advantage of the new measure $\wt\sigma$ is that 
this is $n$-AD-regular.
For any cube $Q\subset\R^n$ and any $c\in\R$,
$$d_Q(\sigma,c\,\HH^n|_{\R^n}) = d_Q(\sigma+2\HH^n|_{\R^n} ,\,(c+2)\HH^n|_{\R^n})
= d_Q\bigl(\wt\sigma,\,(c+2)\HH^n|_{\R^n}\bigr).$$
Thus, taking the infimum over $c\in\R$ we deduce that
$$\alpha^{\R^n}_\sigma(Q) = \alpha^{\R^n}_{\wt\sigma}(Q)$$ for any cube $Q\subset\R^n$.
Now Theorem 1.1 from \cite{Tolsa-plms} applied to the graph of the trivial function $x\mapsto A(x)=0$ (identifying
$\R^n$ with $\R^n\times\{0\}$, say), gives
$$\sum_{Q\in\DD(\R^n):Q\subset R} \alpha^{\R^n}_{\wt\sigma}(3Q)^2\ell(Q)^n \lesssim \ell(R)^n,$$
which proves the lemma.
\end{proof}

\vv
The second auxiliary result we need is the next one.

\begin{lemma} \label{lemf2}
Let $\sigma$ be some finite Radon measure in $\R^n$ and $R\in\DD(\R^n)$ such that
$$\sigma(Q)\leq C_2\ell(Q)^n$$
for all the cubes $Q\in\DD(\R^n)$ with $\ell(Q)\geq \ell_0$.
Then, for every $R\in\DD(\R^n)$ we have
$$\sum_{\substack{Q\in\DD(\R^n):Q\subset R\\ \ell(Q)\geq\ell_0}} \alpha^{\R^n}_\sigma(3Q)^2\ell(Q)^n \lesssim C_2^2\ell(R)^n.$$
\end{lemma}

\begin{proof}
Let $\vphi(x) = m(B(0,\ell_0))^{-1}\,\chi_{B(0,\ell_0)}(x)$. Consider the function $\rho = \vphi*\sigma$ and the measure $\nu=\rho\,dx$.
We have $\|\rho\|_\infty\lesssim C_2$, since for all $x\in\R^n$
$$\rho(x) = \frac1{m(B(0,\ell_0))} \int \vphi(x-y)\,d\sigma(y) = \frac{\sigma(B(x,\ell_0))}{m(B(x,\ell_0))} \lesssim C_2.$$

Let us see that
\begin{equation} \label{see1}
\dist_{3Q}(\nu,\sigma)\lesssim C_2\ell_0\ell(Q)^n\quad \mbox{ for any cube $Q$ with $\ell(Q)\geq\ell_0$.}
\end{equation}
For any function $1$-Lipschitz function $f$ supported on $3Q$, we have
$$\biggl| \int f\,d\nu - \int f\,d\sigma \biggr| = \biggl| \int f\,(\vphi*\sigma)\,dx - \int f\,d\sigma \biggr| = \biggl| \int f*\vphi\,d\sigma - 
\int f\,d\sigma \biggr|.
$$
Since $f$ is $1$-Lipschitz we have
$$|f(x) - f* \vphi(x)| = \biggl| \int_{y\in B(x,\ell_0)} \bigl(f(x) - f(y)\bigr) \vphi(x-y)\,dy\biggr| \leq 
\int \ell_0\, \vphi(x-y)\,dy = \ell_0.$$
Thus,
$$\biggl| \int f\,d\nu - \int f\,d\sigma \biggr|\lesssim \ell_0\sigma(6Q)\lesssim C_2\ell_0\ell(Q)^n,$$
since $\supp(f)\cup\supp(f* \vphi)\subset 6Q$, 
and so \rf{see1} holds.

From \rf{see1} we infer that
$$\alpha^{\R^n}_\sigma(3Q)\lesssim \alpha_\nu^{\R^n}(3Q) + C_2\,\frac{\ell_0}{\ell(Q)},$$
and thus
\begin{align*}
\sum_{\substack{Q\in\DD(\R^n):Q\subset R\\ \ell(Q)\geq\ell_0}} \alpha_\sigma^{\R^n}(3Q)^2\ell(Q)^n & \lesssim 
\sum_{\substack{Q\in\DD(\R^n):Q\subset R\\ \ell(Q)\geq\ell_0}} \Bigl(\alpha^{\R^n}_\nu(3Q)^2 + C_2^2\,\frac{\ell_0^2}{\ell(Q)^2}\Bigr) \ell(Q)^n .
\end{align*}
By Lemma \ref{lemf1} we have
$$\sum_{\substack{Q\in\DD(\R^n):Q\subset R\\ \ell(Q)\geq\ell_0}} \alpha^{\R^n}_\nu(3Q)^2 \ell(Q)^n  \lesssim C_2^2\,\ell(R)^n.$$
On the other hand,
\begin{align*}
\sum_{\substack{Q\in\DD(\R^n):Q\subset R\\ \ell(Q)\geq\ell_0}}  \frac{\ell_0^2}{\ell(Q)^2}\,\ell(Q)^n
& = \sum_{j\geq0:2^{-j}\ell(R)\geq \ell_0} \,
\sum_{\substack{Q\in\DD(\R^n):Q\subset R\\ \ell(Q)=2^{-j}\ell(R)}}  \frac{\ell_0^2}{(2^{-j}\ell(R))^2}\,\ell(Q)^n\\
& = \ell(R)^n\sum_{j\geq0:2^{-j}\ell(R)\geq \ell_0}  \frac{\ell_0^2}{(2^{-j}\ell(R))^2}\approx \ell(R)^n,
\end{align*}
and so the lemma follows.
\end{proof}

\vv

\begin{proof}[\bf Proof of Lemma \ref{lemreduc}]
Let $\mu$ and $B$ be as in Theorem \ref{teo0}.
By a suitable translation and rotation we may assume that the $n$-plane $L$ from Theorem \ref{teo0} 
coincides with 
the horizontal $n$-plane $H=\{x\in\R^{n+1}:x_{n+1}=0\}$ and that $B=B(0,r_0)$.

Our first objective consists in finding an auxiliary cube $R_0$ contained in $B$, 
centered in $H$, and
far from $\partial B$, so that $\mu(R_0)\approx\mu(B)$. 
 The cube $Q_0$, to be chosen later, will be an appropriate cube contained  in $R_0$.  
To find $R_0$, for some constant $0<d<1/10$
to be fixed below, we consider a grid $\QQ$ of $n$-dimensional cubes with side
length $2d\,r_0$ in $H$, so that they cover $H$ and  have disjoint interiors.
We also consider the family of $(n+1)$-dimensional cubes 
$$\wh\QQ = \{Q\times [-d \,r_0,d\,r_0]: Q\in\QQ\},$$
so that the union of the cubes from $\wh\QQ$ equals the strip
$$V=\{x\in\R^{n+1}:\dist(x,H)\leq d\,r_0\}.$$
For any constant $0<a<1$ we have
$$\mu\bigl(B\setminus (aB\cap V)\bigr) \leq \sum_{\substack{P\in\wh\QQ:\\P\cap (B\setminus  aB)\neq \varnothing} } \mu(P\cap B)+
\mu\biggl(B\setminus \bigcup_{P\in\wh\QQ} P\biggr) =: S_1 + S_2.$$
To bound $S_1$ we use the growth condition of order $n$ of $\mu|_B$:
$$S_1\lesssim_{C_0}\!\!\!
\sum_{\substack{P\in\wh\QQ:\\P\cap (B\setminus aB)\neq \varnothing} } \!\!\ell(P)^n 
\lesssim \HH^n\bigl(H\cap A\bigl(0,(a-n^{1/2}2d)r_0, (1+n^{1/2}2d)r_0\bigr)\bigr)\lesssim _{C_0}
(d+1-a)r_0^n$$
(the constant $4n^{1/2}$ mutipliying $d$ on the right hand side was absorbed by the symbol ``$\lesssim$'',
taking into account that $1-a>0$).
To estimate $S_2$ we use the fact that the points $x\in B\setminus \bigcup_{P\in\wh\QQ} P$
are at a distance from $H$ larger that $dr_0$ and apply Chebyshev:
$$S_2 \leq \int_B \frac{\dist(x,H)}{d\,r_0}\,d\mu(x) = \frac1d\,\beta_{\mu,1}^H(B)\,r_0^n.$$

So we obtain 
$$\mu\bigl(B\setminus (aB\cap V)\bigr) \leq C(C_0) \biggl( (d+1-a) +\frac1d\,\beta_{\mu,1}^H(B)\biggr)\,\mu(B).$$ 
We take $d$ and $a$ so that
$$10 (n+1)^{1/2}d = (1-a) = \frac1{10C(C_0)},$$
and we assume
$$\beta_{\mu,1}^H(B)\leq \delta\leq\frac d{10C(C_0)} = \frac1{10 (n+1)^{1/2}(10C(C_0))^2},$$
so that $\mu\bigl(B\setminus (aB\cap V)\bigr)\leq \frac3{10}\,\mu(B)$. Now we choose $R_0$ to be a cube from $\wh\QQ$
which intersects $aB\cap V$ and has maximal $\mu$-measure. Obviously, 
$$\mu(R_0)\approx_{C_0} \mu(aB\cap V)\approx_{C_0} \mu(B),$$
and since $\diam(R_0)= 2(n+1)^{1/2}d\,r_0 = \frac{1-a}{5}\,r_0$, it follows that
\begin{equation}\label{eq4r0}
\dist(4R_0,\partial B)\approx_{C_0} (1-a)r_0\approx_{C_0} r_0.
\end{equation}

The cube $Q_0$ we are looking for will be an appropriate cube contained in $R_0$. To find this, 
first we consider the thin strip
$$V_\delta = \bigl\{x\in\R^{n+1}:\dist(x,H)\leq \delta^{1/2}r_0\bigr\}.$$
Observe that
\begin{equation}\label{eqbvd1}
\mu(B\setminus V_\delta) \leq \int_B \frac{\dist(x,H)}{\delta^{1/2}r_0}\,d\mu(x)
=\frac{\beta_{\mu,1}^H(B)}{\delta^{1/2}}\,r_0^n\lesssim_{C_0} \delta^{1/2}\,\mu(B).
\end{equation}
Denote by $\Pi$ the orthogonal projection on $H$ and consider the measure
$\sigma=\Pi_\#(\mu|_{V_\delta})$. Since $V_\delta$ has width $2\delta^{1/2}$, from the growth
condition (b) in Theorem \ref{teo0}, it follows that
$\sigma(Q)\lesssim_{C_0} \ell(Q)^n$ for any cube $Q$ centered on $H$ with $\ell(Q)\geq
\delta^{1/2}\,r_0$. Assume without loss of generality that $R_0$ is a dyadic cube. Then, by Lemma \ref{lemf2},
$$\sum_{\substack{Q\in\DD(\R^n,R_0)\\ \ell(Q)\geq\delta^{1/2}r_0}} \alpha^{\R^n}_\sigma(3Q)^2\ell(Q)^n \lesssim_{C_0}\ell(R_0)^n,$$
where $\DD(\R^n,R_0)$ stands for the family of dyadic cubes in $\R^n$ contained in $R_0$. 
From this inequality, it easily follows that, for any constant $A'>10$, 
$$\sum_{\substack{Q\in\DD(\R^n,R_0):\ell(A'Q)\leq \ell(R_0)\\ \ell(Q)\geq\delta^{1/(4n+1)}r_0}} \alpha^{\R^n}_\sigma(4A'Q)^2\ell(Q)^n \lesssim_{C_0,A'}\ell(R_0)^n.$$
Note that we have used the fact that $\delta^{1/(4n+1)}>\delta^{1/2}$.
Since the number of dyadic generations between the largest cubes $Q\in\DD(\R^n)$ with $\ell(A'Q)\leq\ell(R_0)$ and the smallest ones with side length $\ell(Q)\geq\delta^{1/(4n+1)}r_0$ is
comparable to 
$$\log_2\frac{C(A')\ell(R_0)}{\delta^{1/(4n+1)}r_0}\approx \log_2\frac{C(A',C_0)}{\delta^{1/(4n+1)}},$$
we deduce that there exists some intermediate generation $j$ such that
$$\sum_{\substack{Q\in\DD_j(\R^n,R_0):\ell(A'Q)\leq \ell(R_0)}} \alpha^{\R^n}_\sigma(4A'Q)^2\ell(Q)^n \lesssim_{C_0,A'} \frac1{\log_2\frac{C(A',C_0)}{\delta^{1/(4n+1)}}} \ell(R_0)^n.$$
Thus, for any $\delta'>0$, if $\delta$ is small enough, we derive
\begin{align*}
\sum_{\substack{Q\in\DD_j(\R^n,R_0):\\ \ell(A'Q)\leq \ell(R_0)}} \alpha^{\R^n}_\sigma(4A'Q)^2\,\sigma(Q) &\leq C(C_0) 
\sum_{\substack{Q\in\DD_j(\R^n,R_0):\\
\ell(A'Q)\leq \ell(R_0)}} \alpha^{\R^n}_\sigma(4A'Q)^2\ell(Q)^n 
\leq \frac{\delta'^2}{50}\sigma(R_0).
\end{align*}
Denote by $\GZ$  the subfamily of  cubes from $\DD_j(\R^n,R_0)$ such that
$\Theta_\sigma(Q)\geq \frac12\,\Theta_\sigma(R_0)$. Observe that
$$\sum_{Q\in\DD_j(\R^n,R_0)\setminus \GZ}\sigma(Q)\leq \frac12\,\Theta_\sigma(R_0)
\sum_{Q\in\DD_j(\R^n,R_0)\setminus \GZ}\ell(Q)^n \leq
\frac12\,\Theta_\sigma(R_0)\,\ell(R_0)^n = 
\frac12\,\sigma(R_0).$$
Hence, $\sum_{Q\in \GZ}\sigma(Q)\geq \frac12\,\sigma(R_0)$, and so
$$\sum_{\substack{Q\in\GZ}} \alpha^{\R^n}_\sigma(4A'Q)^2\,\sigma(Q)\leq 
\frac{\delta'^2}{50}\sigma(R_0)\leq \frac{\delta'^2}{25}\,\sigma\biggl(\,\bigcup_{Q\in \GZ}Q\biggr).$$
Then we deduce that there exists some cube $Q\in\GZ$ such that 
\begin{equation}\label{eqq*231}
 \alpha^{\R^n}_\sigma(4A'Q)\leq
\frac{\delta'}{5}.
\end{equation}

Denote $\wh Q= Q\times[-\ell(Q)/2,\ell(Q)/2]$. We wish now to bound 
 $\alpha^H_\mu(4A'\wh Q)$ in terms of $ \alpha^{\R^n}_\sigma(4A'Q)$. Let $c_H$ be the constant
 that minimizes the infimum in the definition of $\alpha^{\R^n}_\sigma(4A'Q)$ in 
 \rf{eqsigma99}. 
 Given any $1$-Lipschitz function $f$ supported on $4A'\wh Q$ we have
 \begin{align*}
 \left|\int f\,d(\mu-c_H\,\HH^n|_H)\right| &\leq 
 \int_{4A'\wh Q\setminus V_\delta}|f|\,d\mu + \left|\int f\,d(\mu|_{V_\delta}- \sigma)\right| +
 \left|\int\! f\,d(\sigma-c_H\,\HH^n|_H)\right| \\ &=: I_1 + I_2+I_3.
 \end{align*}
By \rf{eqbvd1}, using also that $\ell(Q)\geq \delta^{1/(4n+4)}r_0$,
we have
$$I_1\leq \|f\|_\infty\,\mu(B\setminus V_\delta)\lesssim_{C_0} \delta^{1/2}\,\ell(4A'Q)\,\mu(B)
\lesssim_{C_0,A'} \delta^{1/2}\,\ell(Q)\,r_0^n\lesssim_{C_0,A'} \delta^{1/4}\,\ell(Q)^{n+1}.$$
Now we deal with $I_2$. By the definition of $\sigma$ and the Lipschitz condition on $f$, we get:
\begin{align*}
I_2 & = \left|\int_{4A'\wh Q} f(x)-f(\Pi(x))\,d\mu|_{V_\delta}(x) \right|\\
&\leq \int_{4A'\wh Q} \dist(x,H)\,d\mu|_{V_\delta}(x)\leq \beta_{\mu,1}(B)\,r_0^{n+1}\leq \delta \,r_0^{n+1}
\leq \delta^{3/4}\,\ell(Q)^{n+1}.
\end{align*}
Finally, concerning $I_3$, by \rf{eqq*231} we have
$$I_3 \leq \alpha^{\R^n}_\sigma(4A'Q) \,\ell(4A'Q)^{n+1}\leq
\frac{\delta'}{5}\, \,\ell(4A'Q)^{n+1}.$$
Gathering the estimates obtained for $I_1$, $I_2$, $I_3$ and choosing $\delta$ small enough we 
obtain
$$\left|\int f\,d(\mu-c_H\,\HH^n|_H)\right|\leq \frac{\delta'}2\, \,\ell(4A'Q)^{n+1},$$
and thus  $\alpha^H_\mu(4A'\wh Q)\leq \frac{\delta'}2$.

Finally, we choose $Q_0$ to be a cube with thin boundary such that
$\wh Q\subset Q_0\subset 1.1\wh Q$ (the existence of cubes with thin boundaries such as $Q_0$ has already been
mentioned just before the Main Lemma).
 Since $3A'Q_0\subset 4A'\wh Q\subset 4R_0$ and $\ell(3A'Q_0)
\approx \ell(4A'\wh Q)$, we deduce that $3A'Q_0\subset B$ (by \rf{eq4r0}) and 
that $\alpha^H_\mu(3A' Q_0)\lesssim \alpha^H_\mu(4A'\wh Q)\lesssim \delta'$. Then it is easy to check that $Q_0$ satisfies all the properties (a)-(e)
by construction. Regarding (f), we have
\begin{align*}
\int_{Q_0} |\RR\mu(x) - m_{\mu,Q_0}(\RR\mu)|^2\,d\mu(x) & \leq
2\int_{Q_0} |\RR\mu(x) - m_{\mu,B}(\RR\mu)|^2\,d\mu(x) \\
& \leq 2\,\ve\,\mu(B)\approx_{C_0,\delta} \ve\,\mu(Q_0).
\end{align*}
Thus if $\ve$ is small enough, (f) holds.
\end{proof}

\vvv

% ***************************************************************************

\section{The Localization Lemma}\label{sec4}

This and the remaining sections of this paper are devoted to the proof of the Main Lemma 
\ref{mainlemma}. 
We assume that the hypotheses of the Main Lemma \ref{mainlemma} hold. From now on, we allow all the constants denoted by $C$ and all the implicit constants in the relations ``$\lesssim$'' and ``$\approx$''
to depend on the constants $C_0$ and $C_1$ in the Main Lemma (but not on $A$, $\delta$ or $\ve$).

The main objective of this section is to obtain a local version of the BMO type estimate
in (f) of the Main Lemma \ref{mainlemma}, which will be useful in the next sections of the paper
and is more handy than the statement (f).

Recall that $H$ stands for the horizontal hyperplane $\{x\in\R^{n+1}:x_{n+1}=0\}$. Also we let $c_H$ be some constant that minimizes the infimum in the definition
of $\alpha^H(3AQ_0)$ and we denote $\LL_H=c_H\,\HH^n|_H$.

\vv
\begin{lemma}\label{lemtontu}
If $\delta$ is small enough  (depending on $A$), then we have $c_{H}\approx1$ and $\mu(AQ_0)\lesssim A^n\mu(Q_0)$.
\end{lemma}

\begin{proof}
Let $\vphi$ be a non-negative $C^1$ function supported on $2Q_0$ which equals $1$ on $Q_0$ and satisfies $\|\nabla\vphi\|_\infty\lesssim
1/\ell(Q_0)$.
Then we have
\begin{equation}\label{eqdl9}
\left|\int \vphi\,d(\mu-\LL_{H}) \right|\leq \|\nabla\vphi\|_\infty\,\ell(3AQ_0)^{n+1}\,\alpha^H_\mu(3AQ_0)
\lesssim A^{n+1}\,\delta\,\ell(Q_0)^n.
\end{equation}
Note that the left hand side above equals
$$
\left|\int \vphi\,d\mu-c_{H}\int_{H} \vphi\,d\HH^n \right| = \left|c_1 - c_{H}\right|\int_{H} \vphi \,d\HH^n
,$$
with $$c_1 = \frac{\int \vphi\,d\mu}{\int_{H} \vphi \,d\HH^n} \approx 1,$$
taking into account that 
$$\mu(Q_0)\leq\int \vphi\,d\mu\leq \mu(2Q_0)\lesssim C_0\,\ell(2Q_0)^n\lesssim C_0\,\mu(Q_0).$$
So from \rf{eqdl9} we deduce that
$$\left|c_1 - c_{H}\right| \lesssim  A^{n+1}\,\delta\,\frac{\ell(Q_0)^n}{\int_{H} \vphi \,d\HH^n}\lesssim A^{n+1}\,\delta.$$
The right hand side is $\ll 1\approx c_1$ if $\delta $ is small enough (depending on $A$), and so we infer that
$$c_{H}\approx c_1\approx 1.$$

The estimate $\mu(AQ_0)\lesssim  A^n\,\ell(Q_0)^n$ is an immediate consequence  of the
assumptions either (b) or (c) in the Main Lemma \ref{mainlemma}.
%  we take another auxiliary non-negative $C^1$ function $\wt\vphi$ supported on $3AQ_0$ which equals $1$ on $AQ_0$ and satisfies $\|\nabla\wt\vphi\|_\infty\lesssim
%1/\ell(AQ_0)$. Then we have
%\begin{align*}
%\mu(AQ_0)& \leq \int\wt\vphi\,d\mu \\
%& \leq \left|\int \wt\vphi\,d(\mu-\LL_{H}) \right| + \int \wt\vphi\,d\LL_{H} \\
%&\lesssim \|\nabla\wt\vphi\|_\infty \,\ell(3AQ_0)^{n+1}\,\alpha^H_\mu(3AQ_0) + c_{H}\ell(3AQ_0)^n\\
%&\lesssim A^n\,\delta\,\ell(Q_0)^n + \ell(3AQ_0)^n \lesssim  A^n\,\ell(Q_0)^n.
%\end{align*}
\end{proof}
\vv

\begin{lemma}[Localization Lemma] \label{lem1}
If $\delta$ is small enough  (depending on $A$), then we have
$$\int_{Q_0}|\RR_\mu \chi_{AQ_0}|^2\,d\mu\lesssim \left(\ve + \frac1{A^2} + A^{4n+2} \delta^{1/(4n+4)}\right)\,\mu(Q_0).$$
\end{lemma}

\begin{proof}
Note first that, by standard estimates, for $x,y\in Q_0$,
\begin{align*}
\bigl|\RR_\mu\chi_{(AQ_0)^c} (x) - \RR_\mu\chi_{(AQ_0)^c} (y)\bigr|&\lesssim 
\int_{(AQ_0)^c} \frac{|x-y|}{|x-z|^{n+1}}\,d\mu(z)\\
&\lesssim
\frac{|x-y|}{\ell(AQ_0)}\,P_\mu(AQ_0)
\lesssim \frac1A\,P_\mu(AQ_0)\lesssim \frac1A,
\end{align*}
taking into account the assumption (b) of the Main Lemma for the last inequality.
As a consequence,
$$\bigl|\RR_\mu\chi_{(AQ_0)^c} (x) - m_{\mu,Q_0}(\RR_\mu\chi_{(AQ_0)^c})\bigr|\lesssim\frac1A,$$
and so
$$\int_{Q_0}\bigl|\RR_\mu\chi_{(AQ_0)^c} (x) - m_{\mu,Q_0}(\RR_\mu\chi_{(AQ_0)^c})\bigr|^2d\mu(x)\lesssim \frac1{A^2}\,\mu(Q_0).$$
Together with the assumption (g) in the Main Lemma this gives
\begin{align*}
\int_{Q_0}|\RR_\mu \chi_{AQ_0}- m_{\mu,Q_0}(\RR_\mu \chi_{AQ_0})|^2\,d\mu & \leq 
2\int_{Q_0}\bigl|\RR\mu  - m_{\mu,Q_0}(\RR\mu)\bigr|^2d\mu\\
&\quad +
2\int_{Q_0}\bigl|\RR_\mu\chi_{(AQ_0)^c} - m_{\mu,Q_0}(\RR_\mu\chi_{(AQ_0)^c})\bigr|^2d\mu
\\
&\lesssim
\ve\,\mu(Q_0)+\frac1{A^2}\,\mu(Q_0).
\end{align*}
Hence, to conclude the proof of the lemma it suffices to show that 
$$\bigl|m_{\mu,Q_0}(\RR_\mu \chi_{AQ_0})\bigr|\lesssim A^{2n+1} \delta^{1/(8n+8)}\,\mu(Q_0)
.$$ 
By the antisymmetry of the Riesz kernel we have $m_{\mu,Q_0}(\RR_\mu \chi_{Q_0})=0$, and so the preceding estimate is equivalent
to
\begin{equation}\label{eq10}
\bigl|m_{\mu,Q_0}(\RR_\mu \chi_{AQ_0\setminus Q_0})\bigr|\leq A^{2n+1} \delta^{1/(8n+8)}\,\mu(Q_0).
\end{equation}

To prove \rf{eq10} first we take some small constant
 $0<\kappa<1/10$ to be fixed below. We let $\vphi$ be some $C^1$ function
which equals $1$ on $(1-\kappa)AQ_0\setminus (1+\kappa)Q_0$ and vanishes out of $AQ_0\setminus (1+\frac\kappa2)Q_0$, so that $\vphi$ is even and further
$\|\nabla\vphi\|_\infty\lesssim (\kappa\,\ell(Q_0))^{-1}$.
Then we split
\begin{equation}\label{eq10.5}
\left|\int_{Q_0} \RR_\mu \chi_{AQ_0\setminus Q_0}\,d\mu \right|
\leq \int_{Q_0} \bigl|\RR_\mu (\chi_{AQ_0\setminus Q_0}-\vphi)\bigr|\,d\mu 
+ \left|\int_{Q_0} \RR_\mu \vphi\,d\mu \right|.
\end{equation}
To bound the first integral on the right hand side note that
$\chi_{AQ_0\setminus Q_0}-\vphi=\psi_1 + \psi_2$, with
$$|\psi_1|\leq \chi_{AQ_0\setminus (1-\kappa)AQ_0} \quad\mbox{ and }\quad
|\psi_2|\leq \chi_{(1+\kappa)Q_0\setminus Q_0}.$$
Then we have
\begin{align*}
\int_{Q_0} \bigl|\RR_\mu (\chi_{AQ_0\setminus Q_0}-\vphi)\bigr|\,d\mu & \leq
\int_{Q_0} \bigl|\RR_\mu \psi_1\bigr|\,d\mu + \int_{Q_0} \bigl|\RR_\mu \psi_2\bigr|\,d\mu\\
&\leq \|\RR_\mu \psi_1\|_{L^\infty(\mu|_{Q_0})}\,\mu(Q_0) + 
\|\RR_\mu \psi_2\|_{L^2(\mu|_{Q_0})}\,\mu(Q_0)^{1/2}.
\end{align*}
Since $\dist(\supp\psi_1,Q_0)\approx A\ell(Q_0)$, we have
$$\|\RR_\mu \psi_1\|_{L^\infty(\mu|_{Q_0})} \lesssim \frac1{(A\ell(Q_0))^n} \,\|\psi_1\|_{L^1(\mu)}
\leq \frac1{(A\ell(Q_0))^n} \,\mu(AQ_0\setminus (1-\kappa)AQ_0).$$
On the other hand, since $\RR_\mu$ is bounded in $L^2(\mu|_{(1+\kappa)Q_0})$ and by the thin boundary property of $Q_0$ (in combination with the fact
that $\mu(2Q_0)\approx\mu(Q_0)$), we get
$$\|\RR_\mu \psi_2\|_{L^2(\mu|_{Q_0})}\leq C_1\|\psi_2\|_{L^2(\mu)}\leq 
C_1\,\mu((1+\kappa)Q_0\setminus Q_0)^{1/2}\leq C(C_0,C_1)\,\kappa^{1/2}\,\mu(Q_0)^{1/2}.$$
Therefore,
\begin{equation}\label{eq11}
\int_{Q_0} \bigl|\RR_\mu (\chi_{AQ_0\setminus Q_0}-\vphi)\bigr|\,d\mu\lesssim
\frac1{A^n} \,\mu(AQ_0\setminus (1-\kappa)AQ_0) + \kappa^{1/2}\mu(Q_0).
\end{equation}

To estimate $\mu(AQ_0\setminus (1-\kappa)AQ_0)$ we will use the fact that $\alpha^H_\mu(3AQ_0)\leq \delta$.
To this end, first we consider a function $\wt\vphi$ supported on $A(1+\kappa)Q_0\setminus (1-2\kappa)AQ_0$
which equals $1$ on $AQ_0\setminus (1-\kappa)AQ_0$, with $\|\nabla\wt\vphi\|_\infty\lesssim 1/(A\kappa\ell(Q_0))$.
Then we have
\begin{align}\label{eq29}
\mu(AQ_0\setminus (1-\kappa)AQ_0) &\leq \int \wt\vphi\,d\mu\\
& \leq \left|\int \wt\vphi\,d(\mu-\LL_{H})\right|
+ \int \wt\vphi\,d\LL_{H}\nonumber \\
& \leq \|\nabla\wt\vphi\|_\infty\,\ell(3AQ_0)^{n+1}\,\alpha^H_\mu(3AQ_0) 
+ \LL_{H}\bigl((1+\kappa)AQ_0\setminus (1-2\kappa)AQ_0\bigr)\nonumber\\
& \lesssim \left(\frac{A^{n}}\kappa\,\delta\, + \kappa\,A^{n}\right)\,\ell(Q_0)^n,\nonumber
\end{align}
where we used the estimate for $c_H$ in Lemma \ref{lemtontu} for the last inequality.
Hence, plugging this estimate into \rf{eq11} we obtain
\begin{equation}\label{eq14}
\int_{Q_0} \bigl|\RR_\mu (\chi_{AQ_0\setminus Q_0}-\vphi)\bigr|\,d\mu
\lesssim \left(
\frac 1\kappa \,\delta + \kappa + \kappa^{1/2}\right)\mu(Q_0) \lesssim
\left(
\frac \delta\kappa  + \kappa^{1/2}\right)\mu(Q_0)
.
\end{equation}

It remains to estimate the last summand in the inequality  \rf{eq10.5}. To this end we write
\begin{align}\label{eqdk3}
\left|\int_{Q_0} \RR_\mu \vphi\,d\mu \right| & \leq 
\left|\int_{Q_0} \RR_\mu \vphi\,d(\mu-\LL_{H}) \right|+
\left|\int_{Q_0} \RR (\vphi\mu - \vphi\LL_{H})\,d\LL_{H} \right|\\
&\quad  + 
\left|\int_{Q_0} \RR (\vphi\LL_{H})\,d\LL_{H} \right| \nonumber\\
& = T_1 + T_2 + T_3.\nonumber
\end{align}
Since $\vphi$ is even, by the antisymmetry of the Riesz kernel it follows easily that $T_3=0$.

To deal with $T_1$, consider another auxiliary function $\wh\vphi$ supported on $Q_0$ which equals
$1$ on $(1-\wh\kappa)Q_0$, for some small constant $0<\wh\kappa<\kappa$, so that $\|\nabla\wh\vphi\|_\infty\lesssim1/(\wh\kappa\ell(Q_0))$.
Then we write
$$T_1 \leq 
\left|\int \wh\vphi\,\RR_\mu \vphi\,d(\mu-\LL_{H})  \right| +
\left|\int (\chi_{Q_0}-\wh\vphi)\,\RR_\mu \vphi\, d(\mu-\LL_{H})\right| = T_{1,a} + T_{1,b}.$$
To estimate $T_{1,a}$ we set
$$T_{1,a} \leq \|\nabla(\wh\vphi\,\RR_\mu \vphi)\|_\infty \ell(3AQ_0)^{n+1} \,\alpha^H(3AQ_0),$$
by the definition of $\alpha^H(3AQ_0)$ and considering the $1$-Lipschitz function $\|\nabla(\wh\vphi\,\RR_\mu \vphi)\|_\infty^{-1}\,\wh\vphi\,\RR_\mu \vphi$.
We have
$$\|\nabla(\wh\vphi\,\RR_\mu \vphi)\|_\infty \leq 
\|\nabla(\RR_\mu \vphi)\|_{\infty,Q_0} + \|\nabla\wh\vphi\|_\infty\,\|\RR_\mu \vphi)\|_{\infty,Q_0}.
$$
Since $\dist(\supp\vphi,Q_0)\geq \frac\kappa2\ell(Q_0)$ and $\mu(AQ_0)\lesssim \ell(AQ_0)^n$
(by Lemma \ref{lemtontu}), we have
$$\|\RR_\mu \vphi\|_{\infty,Q_0}\lesssim \frac{\mu(AQ_0)}{(\kappa\ell(Q_0))^n} \lesssim \frac{A^n}{\kappa^n},$$
and, anagously,
$$\|\nabla(\RR_\mu \vphi)\|_{\infty,Q_0}\lesssim 
\frac{\mu(AQ_0)}{(\kappa\ell(Q_0))^{n+1}} \lesssim \frac{A^n}{\kappa^{n+1}\ell(Q_0)}.$$
Hence,
$$\|\nabla(\wh\vphi\,\RR_\mu \vphi)\|_\infty \lesssim \frac{A^n}{\kappa^{n+1}\ell(Q_0)} + 
\frac{A^n}{\wh\kappa\kappa^{n}\ell(Q_0)} \lesssim \frac{A^n}{\wh\kappa\,\kappa^{n}\ell(Q_0)}.$$
So we have
$$T_{1,a} \lesssim \frac{A^{2n+1}}{\wh \kappa\,\kappa^{n}}\,\delta\,\mu(Q_0).
$$
Now we consider the term $T_{1,b}$. We write
$$T_{1,b} \leq \|\chi_{Q_0}-\wh\vphi\|_{L^1(\mu + \LL_{H})} \,\|\RR_\mu \vphi\|_{\infty,Q_0}.$$
Recall that $\|\RR_\mu \vphi\|_{\infty,Q_0}\lesssim \frac{A^n}{\kappa^n}$. Also, by the construction of $\wh\vphi$ and the thin boundary of $Q_0$,
$$\|\chi_{Q_0}-\wh\vphi\|_{L^1(\mu + \LL_{H})}\lesssim \mu(Q_0\setminus (1-\wh\kappa)Q_0) 
+ \LL_H(Q_0\setminus (1-\wh\kappa)Q_0)\lesssim
\wh\kappa \,\mu(Q_0).$$
So we obtain
$$T_{1,b} \leq \frac{A^n}{\kappa^n}\,\wh\kappa\,\mu(Q_0).$$
Thus,
$$T_1\lesssim \left(\frac{A^{2n+1}}{\kappa^{n}\,\wh \kappa}\,\delta + \frac{A^n}{\kappa^n}\,\wh\kappa\right)
\mu(Q_0).$$
Choosing $\wh\kappa=\delta^{1/2}$, say, we get
$$T_1\lesssim \left(\frac{A^{2n+1}}{\kappa^n} + \frac{A^n}{\kappa^n}\right)\delta^{1/2}\,\mu(Q_0) \leq \frac{A^{2n+1}}{\kappa^n}\,\delta^{1/2}\,\mu(Q_0).$$

Finally, we turn our attention to $T_2$. We denote by $\RR_i$ the $i$-th component of $\RR$. By the antisymmetry of each $\RR_i$
we obtain
\begin{align*}
T_2=\left(\sum_{i=1}^{n+1} 
\left|\int_{Q_0} \RR_i (\vphi\mu - \vphi\LL_{H})\,d\LL_{H} \right|^2\right)^{1/2} &=\left(
\sum_{i=1}^{n+1} \left|\int \RR_i (\chi_{Q_0}\LL_{H})\,\vphi\,d(\mu - \LL_{H})\right|^2\right)^{1/2} \\
&\lesssim \|\nabla\bigl( \RR (\chi_{Q_0}\LL_{H})\,\vphi\bigr)\|_\infty\,\ell(3AQ_0)^{n+1} \,\alpha^H(3AQ_0).
\end{align*}
Observe that
$$\|\nabla\bigl( \RR (\chi_{Q_0}\LL_{H})\,\vphi\bigr)\|_\infty
\leq \|\nabla\bigl( \RR (\chi_{Q_0}\LL_{H})\|_{\infty,\supp\vphi} + 
\|\RR (\chi_{Q_0}\LL_{H})\bigr)\|_{\infty,\supp\vphi} \|\nabla\vphi\|_\infty.$$
Using that $\dist(\supp\vphi,Q_0)\geq \frac\kappa2\ell(Q_0)$ and that
$\LL_{H}(Q_0)= c_H\,\HH^n(Q_0\cap H)\approx \ell(Q_0)^n$
, we derive
$$\|\RR (\chi_{Q_0}\LL_{H})\bigr)\|_{\infty,\supp\vphi}\lesssim \frac{\LL_{H}(Q_0)}{(\kappa\ell(Q_0))^n} \lesssim \frac1{\kappa^n}$$
and
$$\|\nabla\bigl( \RR (\chi_{Q_0}\LL_{H})\|_{\infty,\supp\vphi} \lesssim
\frac{\LL_{H}(Q_0)}{(\kappa\ell(Q_0))^{n+1}} \lesssim \frac1{\kappa^{n+1}\ell(Q_0)}.$$
So we obtain
$$T_2\lesssim \frac{A^{n+1}}{\kappa^{n+1}}\,\delta	\,\mu(Q_0).$$

Gathering the estimates for $T_1$ and $T_2$, by \rf{eqdk3} we deduce
$$\left|\int_{Q_0} \RR_\mu \vphi\,d\mu \right| \lesssim \frac{A^{2n+1}}{\kappa^n}\,\delta^{1/2}\,\mu(Q_0) +
\frac{A^{n+1}}{\kappa^{n+1}}\,\delta	\,\mu(Q_0) \lesssim \frac{A^{2n+1}}{\kappa^{n+1}}\,\delta^{1/2}\,\mu(Q_0).$$
Plugging this estimate and \rf{eq14} into \rf{eq10.5}, we obtain
\begin{align*}
\left|\int_{Q_0} \RR_\mu \chi_{AQ_0\setminus Q_0}\,d\mu \right| & \lesssim \left(
\frac \delta\kappa  + \kappa^{1/2}\right)\mu(Q_0)
 +
\frac{A^{2n+1}}{\kappa^{n+1}}\,\delta^{1/2}\,\mu(Q_0)	\\
&\lesssim \left(\frac{A^{2n+1}}{\kappa^{n+1}}\,\delta^{1/2}	+ \kappa^{1/2}\right)\,\mu(Q_0).
\end{align*}
So if we choose $\kappa = \delta^{1/(4n+4)}$, we get
$$\left|\int_{Q_0} \RR_\mu \chi_{AQ_0\setminus Q_0}\,d\mu \right|
\lesssim \left(A^{2n+1}\,\delta^{1/4}	+ \delta^{1/(8n+8)}\right)\,\mu(Q_0)\lesssim 
	A^{2n+1} \delta^{1/(8n+8)}\,\mu(Q_0),$$
which yields \rf{eq10} and finishes the proof of the lemma.
\end{proof}

\vv
From now on we will assume that $\delta$ is small enough, depending on $A$, so that the conclusion in the preceding lemma
holds.

\vvv
% ***************************************************************************

\section{The dyadic lattice of David and Mattila}\label{secdm}

We will use the dyadic lattice of cells
with small boundaries of David-Mattila associated with a Radon measure $\sigma$ \cite[Theorem 3.2]{David-Mattila}. This lattice is very appropriate for the stopping time arguments in the next  section.
The properties of this lattice are summarized in the next lemma.

\begin{lemma}[David, Mattila]
\label{lemcubs}
Let $\sigma$ be a compactly supported Radon measure in $\R^{n+1}$.
Consider two constants $K_0>1$ and $A_0>5000\,K_0$ and denote $S_\sigma=\supp\sigma$. 
Then there exists a sequence of partitions
of $S_\sigma$ into Borel subsets $Q$, $Q\in \DD_{\sigma,k}$,
 which we will refer to as cells, with the following properties:
\begin{itemize}
\item For each integer $k\geq0$, $S_\sigma$ is the disjoint union of the cells $Q$, $Q\in\DD_{\sigma,k}$, and
if $k<l$, $Q\in\DD_{\sigma,l}$, and $R\in\DD_{\sigma,k}$, then either $Q\cap R=\varnothing$ or else $Q\subset R$.
\vv

\item The general position of the cells $Q$ can be described as follows. For each $k\geq0$ and each cell $Q\in\DD_{\sigma,k}$, there is a ball $B(Q)=B(z_Q,r(Q))$ such that
$$z_Q\in S_\sigma, \qquad A_0^{-k}\leq r(Q)\leq K_0\,A_0^{-k},$$
$$S_\sigma\cap B(Q)\subset Q\subset S_\sigma\cap 28\,B(Q)=S_\sigma \cap B(z_Q,28r(Q)),$$
and
$$\mbox{the balls\, $5B(Q)$, $Q\in\DD_{\sigma,k}$, are disjoint.}$$

\vv
\item The cells $Q\in\DD_{\sigma,k}$ have small boundaries. That is, for each $Q\in\DD_{\sigma,k}$ and each
integer $l\geq0$, set
$$N_l^{ext}(Q)= \{x\in S_\sigma\setminus Q:\,\dist(x,Q)< A_0^{-k-l}\},$$
$$N_l^{int}(Q)= \{x\in Q:\,\dist(x,S_\sigma\setminus Q)< A_0^{-k-l}\},$$
and
$$N_l(Q)= N_l^{ext}(Q) \cup N_l^{int}(Q).$$
Then
\begin{equation}\label{eqsmb2}
\sigma(N_l(Q))\leq (C^{-1}K_0^{-3(n+1)-1}A_0)^{-l}\,\sigma(90B(Q)).
\end{equation}
\vv

\item Denote by $\DD_{\sigma,k}^{db}$ the family of cells $Q\in\DD_{\sigma,k}$ for which
\begin{equation}\label{eqdob22}
\sigma(100B(Q))\leq K_0\,\sigma(B(Q)).
\end{equation}
We have that $r(Q)=A_0^{-k}$ when $Q\in\DD_{\sigma,k}\setminus \DD_{\sigma,k}^{db}$
and
\begin{equation}\label{eqdob23}
\sigma(100B(Q))\leq K_0^{-l}\,\sigma(100^{l+1}B(Q))\quad
\mbox{for all $l\geq1$ with $100^l\leq K_0$ and $Q\in\DD_{\sigma,k}\setminus \DD_{\sigma,k}^{db}$.}
\end{equation}
\end{itemize}
\end{lemma}

\vv

We use the notation $\DD_\sigma=\bigcup_{k\geq0}\DD_{\sigma,k}$. Observe that the families $\DD_{\sigma,k}$ are only defined for $k\geq0$. So the diameters of the cells from $\DD$ are uniformly
bounded from above. 

\vv
\begin{rem}\label{remfac1}
Any two disjoint cells $Q,Q'\in\DD_\sigma$ satisfy $\frac12B(Q)\cap \frac12B(Q')=\varnothing$. This holds with $\frac12$
replaced by $5$ by the statements in the lemma above in the case that $Q,Q'$ are of the
same generation $\DD_{\sigma,k}$. If $Q\in\DD_{\sigma,j}$ and $Q'\in\DD_{\sigma,k}$ with $j\neq k$, this follows easily too. Indeed, assume $j<k$, and suppose that $\frac12B(Q)\cap \frac12B(Q')\neq\varnothing$. Since $r(Q')\ll r(Q)$ (by choosing $A_0$ big enough), this implies that $B(Q')\subset B(Q)$, and so
$$B(Q')\cap S_\sigma\subset B(Q)\cap S_\sigma\subset Q,$$
which implies that $Q'\cap Q\neq\varnothing$ and gives a contradiction.
\end{rem}
\vv

%For $Q\in\DD$, we set $\DD(Q) =\{P\in\DD:P\subset Q\}$.
Given $Q\in\DD_{\sigma,k}$, we denote $J(Q)=k$. 
We set
$\ell(Q)= 56\,K_0\,A_0^{-k}$ and we call it the side length of $Q$. Note that 
$$\frac1{28}\,K_0^{-1}\ell(Q)\leq \diam(28B(Q))\leq\ell(Q).$$
Observe that $r(Q)\approx\diam(Q)\approx\ell(Q)$.
Also we call $z_Q$ the center of $Q$, and the cell $Q'\in \DD_{\sigma,k-1}$ such that $Q'\supset Q$ the parent of $Q$.
 We set
$B_Q=28 B(Q)=B(z_Q,28\,r(Q))$, so that 
$$S_\sigma\cap \tfrac1{28}B_Q\subset Q\subset B_Q.$$

We assume $A_0$ to be big enough so that the constant $C^{-1}K_0^{-3(n+1)-1}A_0$ in 
\rf{eqsmb2} satisfies 
$$C^{-1}K_0^{-3(n+1)-1}A_0>A_0^{1/2}>10.$$
Then we deduce that, for all $0<\lambda\leq1$,
\begin{align}\label{eqfk490}\nonumber
\sigma\bigl(\{x\in Q:\dist(x,S_\sigma\setminus Q)\leq \lambda\,\ell(Q)\}\bigr) + 
\sigma\bigl(\bigl\{x\in 3.5B_Q:\dist&(x,Q)\leq \lambda\,\ell(Q)\}\bigr)\\
&\leq
c\,\lambda^{1/2}\,\sigma(3.5B_Q).
\end{align}

We denote
$\DD_\sigma^{db}=\bigcup_{k\geq0}\DD_{\sigma,k}^{db}$.
Note that, in particular, from \rf{eqdob22} it follows that
\begin{equation}\label{eqdob*}
\sigma(3.5B_{Q})\leq \sigma(100B(Q))\leq K_0\,\sigma(Q)\qquad\mbox{if $Q\in\DD_\sigma^{db}.$}
\end{equation}
For this reason we will call the cells from $\DD_\sigma^{db}$ doubling. 
Given $Q\in\DD_\sigma$, we denote by $\DD_\sigma(Q)$
the family of cells from $\DD_\sigma$ which are contained in $Q$. Analogously,
we write $\DD_\sigma^{db}(Q) = \DD^{db}_\sigma\cap\DD(Q)$.

As shown in \cite[Lemma 5.28]{David-Mattila}, every cell $R\in\DD_\sigma$ can be covered $\sigma$-a.e.\
by a family of doubling cells:
\vv

\begin{lemma}\label{lemcobdob}
Let $R\in\DD_\sigma$. Suppose that the constants $A_0$ and $K_0$ in Lemma \ref{lemcubs} are
chosen suitably. Then there exists a family of
doubling cells $\{Q_i\}_{i\in I}\subset \DD_\sigma^{db}$, with
$Q_i\subset R$ for all $i$, such that their union covers $\sigma$-almost all $R$.
\end{lemma}

The following result is proved in \cite[Lemma 5.31]{David-Mattila}.
\vv

\begin{lemma}\label{lemcad22}
Let $R\in\DD_\sigma$ and let $Q\subset R$ be a cell such that all the intermediate cells $S$,
$Q\subsetneq S\subsetneq R$ are non-doubling (i.e.\ belong to $\DD_\sigma\setminus \DD_\sigma^{db}$).
Then
\begin{equation}\label{eqdk88}
\sigma(100B(Q))\leq A_0^{-10n(J(Q)-J(R)-1)}\sigma(100B(R)).
\end{equation}
\end{lemma}

%Let us remark that the constant $10$ in \rf{eqdk88} can be replaced by any other positive 
%constant if $A_0$ and $K_0$ are chosen suitably in Lemma \ref{lemcubs}, as shown in (5.30) of
%\cite{David-Mattila}.

%Given a ball (or an arbitrary set) $B\subset \R^{n+1}$, we consider its $n$-dimensional density:
%$$\Theta_\sigma(B)= \frac{\sigma(B)}{\diam(B)^n}.$$
%We will also write $\Theta_\sigma^p(x,r)$ instead of $\Theta_\sigma^p(B(x,r))$.

From the preceding lemma we deduce:

\vv
\begin{lemma}\label{lemcad23}
Let $Q,R\in\DD_\sigma$ be as in Lemma \ref{lemcad22}.
Then
$$\Theta_\sigma(100B(Q))\leq K_0\,A_0^{-9n(J(Q)-J(R)-1)}\,\Theta_\sigma(100B(R))$$
and
$$\sum_{S\in\DD_\sigma:Q\subset S\subset R}\Theta_\sigma(100B(S))\leq C\,\Theta_\sigma(100B(R)),$$
with $C$ depending on $K_0$ and $A_0$.
\end{lemma}

For the easy proof, see
 \cite[Lemma 4.4]{Tolsa-memo}, for example.

\vvv

% ***************************************************************************

\section{The low density cells and the stopping cells}\label{sec6}

We consider the measure $$\sigma = \mu|_{Q_0}$$ and the associated dyadic lattice $\DD_\sigma$ introduced in Section \ref{secdm}, re-scaled appropriately, so that we can assume that $Q_0$ is a cell from $\DD_\sigma$. 
Below we allow all the constants denoted by $C$ and all the implicit constants in the relations ``$\lesssim$'' and ``$\approx$''
to depend on the constants $A_0$ and $K_0$ from the construction of the lattice $\DD_\sigma$.
 
Let $0<\theta_0\ll1$ be a very small constant to be fixed later. We denote by $\LD$ the family of those cells from 
$\DD_\sigma$ such that $\Theta_\sigma(3.5B_Q)\leq \theta_0$ and have maximal side length.

The main difficulty for the proof of the Main Lemma \ref{mainlemma} consists in showing that the following holds.

\begin{keylemma}\label{keylemma}
There exists some constant $\ve_0$ such that if $A$ is big enough and $\theta_0,\delta,\ve$ are small enough (with $\delta$ possibly depending on $A$), 
then
$$\mu\Biggl(\bigcup_{Q\in \LD} Q\Biggr)\leq (1-\ve_0)\,\mu(Q_0).$$
\end{keylemma}

\vv

Consider the set
\begin{equation*}
F = Q_0\cap\supp\mu\setminus \bigcup_{Q\in \LD} Q.
\end{equation*}
From the definition of the family $\LD$ and the properties of the lattice $\DD_\sigma$, it follows that
\begin{equation}\label{eqAD**}
\mu(B(x,r))\approx r^n\quad\mbox{ for all $x\in F$, $0<r\leq \ell(Q_0)$.}
\end{equation}
Indeed, given any ball $B(x,r)$, with $x\in F$, $0<r\leq \ell(Q_0)$, 
the upper growth condition $\mu(B(x,r))\lesssim r^n$ is a consequence of the assumption (c)
of the Main Lemma \ref{mainlemma}. For the converse estimate,
consider the largest cell
$Q\in\DD_\sigma$ containing $x$ and such that $3.5B_Q\subset B(x,r)$, so that $r(B_Q)\approx r$. Since $Q\not\in\LD$ and $Q$ is not contained in any
other cell from $\LD$, then 
$$\Theta_\mu(B(x,r))\geq\Theta_\sigma(B(x,r))\gtrsim \Theta_\sigma(3.5B_Q)\geq\theta_0.$$

The estimate \rf{eqAD**} implies that $\mu|_F=h\,\HH^n|_F$, for some function $h\approx1$. So the Key Lemma ensures that
there is a significant portion of the measure $\mu$ of $Q_0$ which is absolutely continuous with respect to
$\HH^n$, which is one of the main points in the proof of the Main Lemma \ref{mainlemma}.

\vv
The proof of the Key Lemma will be carried out along the next sections of this paper. In what follows
{\em we will assume that 
\begin{equation}\label{assu1}
\mu\Biggl(\bigcup_{Q\in \LD} Q\Biggr)> (1-\ve_0)\,\mu(Q_0)
\end{equation}
and we will get a contradiction for $\ve_0$ small enough}.
To this end, first we need to construct another family of stopping cells which we will denote by $\sss$. This is
defined as follows. For each $Q\in\LD$ we consider the family of maximal cells contained in $Q$ 
from $\DD_\sigma^{db}$ (so they are doubling)
with side length at most $t\,\ell(Q)$, where $0<t<1$ is some small parameter which will be fixed below. We denote this family by 
$\sss(Q)$.
Then we define
$$\sss = \bigcup_{Q\in\LD} \sss(Q).$$

 Note that, by Lemma \ref{lemcobdob}, it is immediate that, for each $Q\in\LD$, the cells from $\sss(Q)$ cover $\mu$-almost all $Q$.
 So the assumption \rf{assu1} is equivalent to 
$$\mu\Biggl(\bigcup_{Q\in \sss} Q\Biggr)> (1-\ve_0)\,\mu(Q_0)
$$

In the remainder of this section we will prove some auxiliary results involving mainly the stopping cells, and another auxiliary measure $\mu_0$ that we will introduce below.

\begin{lemma}\label{lempoisson}
If we choose $t=\theta_0^{1/(n+1)}$, then we have:
$$\Theta_\mu(2B_Q)\leq P_\mu(2B_Q)\lesssim\theta_0^{1/(n+1)}\quad \mbox{ for all $Q\in\sss$.}$$
\end{lemma}

\begin{proof}
Let $Q\in\sss$ and $R\in\LD$ such that $Q\subset R$. The first inequality in the lemma is trivial and so we only have to prove the second one.
Let $R'\in\DD_\sigma$ the maximal cell such that $Q\subset R'
\subset R$ with $\ell(R')\leq t\,\ell(R)$, so that $\ell(R')\approx t\,\ell(R)$.
Then we write
\begin{align*}
P_\mu(2B_Q) & \lesssim \sum_{P\in\DD_\sigma:Q\subset P\subset R'} \Theta_\mu(2B_P)\,\frac{\ell(Q)}{\ell(P)} +
\sum_{P\in\DD_\sigma:R'\subset P\subset R} \Theta_\mu(2B_P)\,\frac{\ell(Q)}{\ell(P)} \\
& +
\sum_{P\in\DD_\sigma:R\subset P\subset Q_0} \Theta_\mu(2B_P)\,\frac{\ell(Q)}{\ell(P)}+
\sum_{k\geq1} \Theta_\mu(2^k Q_0)\,\frac{\ell(Q)}{\ell(2^kQ_0)}\\
& = S_1 + S_2+ S_3+S_4.
\end{align*}
To deal with the sums $S_1$ and $S_2$, note that for all $P\subset R$, since $2B_P\subset 2B_R$
(assuming $A_0$ to be big enough), we have
$$\Theta_\mu(2B_P) = \frac{\mu(2B_P)}{r(2B_P)^n}\leq \frac{\mu(2B_R)}{r(2B_P)^n} = 
\Theta_\mu(2B_R)\,\frac{r(B_R)^n}{r(B_P)^n} \approx\Theta_\mu(2B_R)\,\frac{\ell(R)^n}{\ell(P)^n}.$$
Therefore, since $\Theta_\mu(2B_R)\lesssim\theta_0$ and all the cells $P$ appearing in $S_2$
satisfy $\ell(P)\geq t\,\ell(R)$, we deduce that all such cells satisfy 
$\Theta_\mu(2B_P)\lesssim \frac1{t^n}\,\Theta_\mu(2B_R)\lesssim\frac{\theta_0}{t^n}$, and thus
$$S_2\lesssim \frac{\theta_0}{t^n}
\sum_{P\in\DD_\sigma:R'\subset P\subset R} \frac{\ell(Q)}{\ell(P)}
\lesssim
\frac{\theta_0}{t^n}.$$
Also, since there are no $\mu$-doubling cells between $R'$ and $Q$, from Lemma \ref{lemcad23}
 we deduce that the cells $P$ in the sum $S_1$ satisfy
 $$\Theta_\mu(2B_P)\lesssim \Theta_\mu(2B_{R'})\lesssim\frac1{t^n}\,\Theta_\mu(2B_R)\lesssim\frac{\theta_0}{t^n},$$
 and therefore we also get
$$S_1\lesssim \frac{\theta_0}{t^n}\sum_{P\in\DD_\sigma:Q\subset P\subset R'}\frac{\ell(Q)}{\ell(P)}\lesssim \frac{\theta_0}{t^n}
.$$ 
For the cells $P$ in the sum $S_3$ we just take into account that $\Theta_\mu(2B_P)\lesssim 1$,
and thus
$$S_3\lesssim \sum_{P\in\DD_\sigma:R\subset P\subset Q_0} \frac{\ell(Q)}{\ell(P)}\lesssim
\frac{\ell(Q)}{\ell(R)}\lesssim t.$$
Regarding the sum $S_4$, note that
$$S_4  = \sum_{1\leq k<\log_2A} \Theta_\mu(2^k Q_0)\,\frac{\ell(Q)}{\ell(2^kQ_0)}
+ \sum_{k:2^k\geq A} \Theta_\mu(2^k Q_0)\,\frac{\ell(Q)}{\ell(2^kQ_0)}= S_{4,a} + S_{4,b}.$$
For the indices $k$ in $S_{4,a}$ we use the fact that $\Theta_\mu(2^k Q_0)\lesssim C_0$, and so we get 
$$S_{4,a}\lesssim \sum_{1\leq k<\log_2A}\frac{\ell(Q)}{\ell(2^kQ_0)} \lesssim \frac{\ell(Q)}{\ell(Q_0)}\leq t.$$
For $S_{4,b}$ we write
$$S_{4,b} \lesssim \frac{\ell(Q)}{\ell(AQ_0)}\sum_{j\geq0} \Theta_\mu(2^j AQ_0)\,\frac{\ell(AQ_0)}{\ell(2^jAQ_0)}\lesssim
\frac{\ell(Q)}{\ell(AQ_0)}\,P_\mu(AQ_0)\lesssim \frac{\ell(Q)}{\ell(AQ_0)}\leq t.
$$
Hence,
$$P_\mu(2B_Q)\lesssim \frac{\theta_0}{t^n} + t \approx \theta_0^{1/(n+1)},$$
recalling that $t=\theta_0^{1/(n+1)}$.
\end{proof}
\vv

From now on we assume that we have chosen $t=\theta_0^{1/(n+1)}$, so that the conclusion of the preceding lemma holds. 

\vv
The family $\sss$ may consist of an infinite number of cells.
For technical reasons, it is convenient to consider a finite subfamily of $\sss$ which contains a very big proportion of the $\mu$ measure of $\sss$. So we let $\sss_0$ be a {\em finite} subfamily of $\sss$
such that
\begin{equation}\label{stop00}
\mu\Biggl(\bigcup_{Q\in \sss_0} Q\Biggr)> (1-2\ve_0)\,\mu(Q_0).
\end{equation}

We denote by $\bad$ the family of the cells $P\in\sss$ such that $1.1B_P\cap\partial Q_0\neq \varnothing$.

\begin{lemma}\label{lempocbad}
We have
$$\mu\Biggl(\bigcup_{Q\in \bad} Q\Biggr)\lesssim \theta_0^{1/(n+1)}\,\mu(Q_0).$$
\end{lemma}

\begin{proof}
Let $I\subset \bad$ an arbitrary finite family of bad cells. 
We apply the covering theorem of triple balls of Vitali to the family $\{1.15B_Q\}_{Q\in I}$, so that
we get a subfamily $J\subset I$ satisfying
\begin{itemize}
\item $1.15B_P\cap 1.15B_Q=\varnothing$ for different cells $P,Q\in J$, and
\item $\bigcup_{P\in I}1.15 B_P\subset \bigcup_{Q\in J}3.45 B_Q$.
\end{itemize}
Then, using that
$$\mu(3.45B_Q)\leq \mu(3.5B_Q)\lesssim \mu(B_Q)\lesssim \theta_0^{1/(n+1)} \,r(B_Q)^n
\quad\mbox{ for all $Q\in J$,}$$
 we get
$$\mu\biggl(\bigcup_{P\in I} B_P\biggr) \leq \sum_{Q\in J} \mu(3.45B_Q)\lesssim \theta_0^{1/(n+1)}\,\sum_{Q\in J}r(B_Q)^n.
$$
For each $Q\in J$ we have $1.1B_Q\cap \partial Q_0\neq\varnothing$ and so we deduce that
$$\HH^n(1.15B_Q\cap\partial Q_0)\gtrsim r(B_Q)^n.$$
Thus, using also that the balls $1.15B_Q$, $Q\in J$, are pairwise disjoint,
$$\mu\biggl(\bigcup_{P\in I} P\biggr) \lesssim \theta_0^{1/(n+1)}\sum_{Q\in J} \HH^n(1.15B_Q\cap\partial Q_0) 
\leq \theta_0^{1/(n+1)}\, \HH^n(\partial Q_0) \approx \theta_0^{1/(n+1)}\,\mu(Q_0),$$
and the lemma follows.
\end{proof}

\vv

We will now define an auxiliary measure $\mu_0$. First, given a small constant $0<\kappa_0\ll1$ (to be fixed below) and $Q\in\DD_\sigma$, we denote
\begin{equation}\label{eqik00}
I_{\kappa_0}(Q) = \{x\in Q:\dist(x,\supp\sigma\setminus Q)\geq \kappa_0\ell(Q)\}.
\end{equation}
So $I_{\kappa_0}(Q)$ is some kind of inner subset of $Q$. We set
\begin{equation}\label{eqdefmu0}
\mu_0 = \mu|_{Q_0^c} + \sum_{Q\in\sss_0\setminus\bad} \mu|_{I_{\kappa_0}(Q)}.
\end{equation}
Observe that, by the doubling and small boundary condition \rf{eqfk490} of $Q\in\sss_0$, we have
$$\mu(Q\setminus I_{\kappa_0}(Q))\lesssim \kappa_0^{1/2}\,\mu(3.5B_Q) \lesssim \kappa_0^{1/2}\,\mu(Q).$$
Combining this estimate with \rf{stop00} and Lemma \ref{lempocbad}, and taking into account the definition of $\mu_0$ in 
\rf{eqdefmu0}, we get the following estimate for the total variation of 
$\mu-\mu_0$:
\begin{align}\label{eqfac99}
\|\mu-\mu_0\| & =
\mu(Q_0) - \mu_0(Q_0)\\ & = \mu(Q_0) - \sum_{Q\in \sss_0\setminus \bad}\mu(I_{\kappa_0}(Q))\nonumber \\
& =
\mu\Biggl(Q_0\setminus \bigcup_{Q\in \sss_0} Q\Biggr) + \sum_{Q\in \bad}\mu(Q)+ \sum_{Q\in \sss_0\setminus \bad}\mu(Q\setminus I_{\kappa_0}(Q)) \nonumber\\
& \leq 2\ve_0\,
\mu(Q_0) + C\theta_0^{1/(n+1)}\,\mu(Q_0) + C\kappa_0^{1/2}\,\mu(Q_0).\nonumber
\end{align}
Together with Lemma \ref{lem1} this yields
the following.

\begin{lemma}\label{lemfac93}
If $\delta$ is small enough (depending on $A$), then we have
$$\int_{Q_0}|\RR(\chi_{AQ_0}\mu_0)|^2\,d\mu_0\lesssim \left(\ve + \frac1{A^2} + 
\delta^{1/(8n+8)}
+\ve_0+\theta_0^{1/(n+1)} + \kappa_0^{1/2}
\right)\,\mu(Q_0).$$
\end{lemma}

\begin{proof}
We have
\begin{align*}
\int_{Q_0}|\RR(\chi_{AQ_0}\mu_0)|^2\,d\mu_0 & \leq 
2\int_{Q_0}|\RR(\chi_{AQ_0}\mu)|^2\,d\mu + 2\int_{Q_0}|\RR(\chi_{AQ_0}(\mu-\mu_0))|^2\,d\mu\\
&\lesssim
\left(\ve + \frac1{A^2}+ A^{4n+2} \delta^{1/(4n+4)}
+\ve_0+\theta_0^{1/(n+1)} + \kappa_0^{1/2}
\right)\,\mu(Q_0),
\end{align*}
by Lemma \ref{lem1}, the $L^2(\mu|_{Q_0})$ boundedness of $\RR_{\mu|_{Q_0}}$, and \rf{eqfac99}.
\end{proof}

\vvv

% ***************************************************************************

\section{The periodic measure $\wt\mu$}\label{secmutilde}

To prove the Key Lemma \ref{keylemma}, in Section \ref{sec9} we will apply a variational argument to derive a contradiction.
The application of this variational argument requires to replace the measure $\mu$ by a suitable periodic
version of $\mu$, which we will denote by $\wt\mu$.
In this section we introduce $\wt\mu$,  we show that this is very flat in $3AQ_0$ (i.e., $\alpha_{\wt\mu}^H(3AQ_0)\ll1$), and we estimate $\int_{Q_0}|\RR(\chi_{AQ_0}\wt\mu)|^2\,d\wt\mu$. We also
prove other technical results involving $\wt\mu$.

Let $\M$ be the lattice of cubes in $\R^{n+1}$ obtained by translating $Q_0$ in directions parallel to $H$, so that $H$ coincides with the union of the $n$-dimensional cubes from the family $\{P\cap H\}_{P\in \M}$ and the cubes have disjoint interiors.
For each $P\in\M$, denote by $z_P$ the center of $P$ and consider the translation $T_P:x\to x+z_P$,
so that $P=T_P(Q_0)$. Note that $\{z_P:P\in\M\}$ coincides with the set $(\ell(Q_0)\Z^n)\times\{0\}$.
We define
$$\wt \mu = \sum_{P\in\M} (T_P)_\#(\mu_0|_{Q_0}).$$
That is,
$$
\wt\mu(E) = \sum_{P\in\M} \mu_0(Q_0\cap T_P^{-1}(E)) = \sum_{P\in\M} \mu_0(Q_0\cap (E-z_P)).
$$
It is easy to check that:
\begin{enumerate}[(i)]
\item $\wt \mu$ is periodic with respect to $\M$, that is, for all $P\in\M$ and all $E\subset\R^{n+1}$, $\wt\mu(E+z_P)=\wt\mu(E)$.
\item $\chi_{Q_0}\wt\mu = \mu_0$.
\end{enumerate}
The latter property holds because $\mu_0(\partial Q_0)=0$.

For simplicity, from now on we will assume that $A$ is a big enough odd natural number.

\begin{lemma}\label{lemalfa}
We have
$$\alpha^H_{\wt\mu}(3AQ_0)\leq C_3\,A^{n+1}\,\biggl(\ve_0 + \theta_0^{1/(n+1)} + \kappa_0^{1/2} + \delta^{1/2}\biggr).$$
In fact, 
$$\dist_{3AQ_0}(\wt\mu,\LL_{H})\leq C_3
A^{n+1} \,
\biggl(\ve_0 + \theta_0^{1/(n+1)} + \kappa_0^{1/2} + \delta^{1/2}\biggr)
\,\ell(3AQ_0)^{n+1},$$
where $\LL_{H}$ is the same minimizing measure as the one for $\alpha^H_{\mu}(3AQ_0)$.
\end{lemma}

\begin{proof}
Let $f$ be a Lipschitz function supported on $3AQ_0$ with Lipschitz constant at most $1$. Denote by $\M_0$
the family of cubes from $\M$ which are contained in $3AQ_0$. Let $\kappa>0$ be some small parameter to be fixed below.
Consider a $C^1$ function $\vphi$ supported on $Q_0$ which equals $1$ on $(1-\kappa)Q_0$, with
$\|\nabla\vphi\|_\infty\lesssim 1/(\kappa\ell(Q_0))$ and denote $\vphi_P(x)=\vphi(x-z_P)$.
Then we write
\begin{align}\label{eqa91}
\left|\int f\,d(\wt\mu-\LL_{H})\right| & \leq \sum_{P\in \M_0} \left|\int_P f\,d(\wt\mu-\LL_{H})\right|\\
& \leq \sum_{P\in \M_0} \left|\int \vphi_P\,f\,d(\wt\mu-\LL_{H})\right|+
\sum_{P\in \M_0} \int \bigl|(\chi_P-\vphi_P)\,f\bigr|\,d(\wt\mu+\LL_{H}).\nonumber
\end{align}

Let us estimate the first sum on the right hand side. Since $\wt\mu|_P = (T_P)_\#\mu_0|_{Q_0}$
and $\LL_{H} = (T_P)_\#\LL_{H}$, we have
\begin{align*}
\left|\int \vphi_P\,f\,d(\wt\mu-\LL_{H})\right| & = 
\left|\int \vphi(x)\,f(x+z_P)\,d(\mu_0-\LL_{H})\right| \\
& \leq \left|\int \vphi(x)\,f(x+z_P)\,d(\mu_0-\mu)\right|
+ \left|\int \vphi(x)\,f(x+z_P)\,d(\mu-\LL_{H})\right|\\
& = I_1+ I_2.
\end{align*}
To estimate $I_1$ we use \rf{eqfac99} and the fact that, by the mean value theorem, $\|\vphi\,f(\cdot+z_P)\|_\infty\lesssim \ell(3AQ_0)$:
$$I_1\leq \left|\int \vphi(x)\,f(x+z_P)\,d(\mu_0-\mu)\right|
\lesssim \bigl(\ve_0 + \theta_0^{1/(n+1)} + \kappa_0^{1/2}\bigr)\,\ell(3AQ_0)^{n+1}.$$

Concerning $I_2$, we write
$$I_2\lesssim \|\nabla(\vphi f(\cdot+z_P))\|_\infty
\ell(3AQ_0)^{n+1}\,\alpha^H_\mu(3A_Q).$$
Note that
$$\|\nabla(\vphi_P f)\|_\infty \leq \|\nabla  f\|_\infty + \|f\|_\infty\|\nabla\vphi_P)\|_\infty
\lesssim 1 + C\,A\,\ell(Q_0) \,\frac1{\kappa\,\ell(Q_0)} \lesssim \frac A\kappa.$$
Thus,
$$I_2\lesssim A \,\frac{\delta}{\kappa}\,\ell(3AQ_0)^{n+1}.$$
Hence,
$$\left|\int \vphi_P\,f\,d(\wt\mu-\LL_{H})\right|\lesssim A\,\biggl(\ve_0 + \theta_0^{1/(n+1)} + \kappa_0^{1/2} + \frac{\delta}{\kappa}\biggr)
\,\ell(3AQ_0)^{n+1}.$$

To deal with the last sum on the right hand side of \rf{eqa91} we write
\begin{align*}
\int \bigl|(\chi_P-\vphi_P)\,f\bigr|\,d(\wt\mu+\LL_{H}) &\leq
\|\chi_P-\vphi_P\|_{L^1(\wt\mu+ \LL_{H})} \,\|f\|_\infty \\
& \lesssim (\mu+\LL_{H})\bigl(Q_0\setminus (1-\kappa)Q_0\bigr)
\,\ell(3AQ_0).
\end{align*}
By the thin boundary condition on $Q_0$, 
$$\mu\bigl(Q_0\setminus (1-\kappa)Q_0\bigr)\lesssim\kappa\,\mu(Q_0)=\kappa\,\ell(Q_0)^n.$$
Clearly, the same estimate holds replacing $\mu$ by $\LL_{H}$. So we deduce
$$\int \bigl|(\chi_P-\vphi_P)\,f\bigr|\,d(\wt\mu+\LL_{H})\lesssim \kappa\,\ell(Q_0)^{n+1}.$$

Taking into account that the number of cubes $P\in\M_0$ is comparable to $A^n$, we get
$$\left|\int f\,d(\wt\mu-\LL_{H})\right|\lesssim A^{n+1} \,
\biggl(\ve_0 + \theta_0^{1/(n+1)} + \kappa_0^{1/2} + \frac{\delta}{\kappa} + \kappa\biggr)
\,\ell(3AQ_0)^{n+1}.$$
Choosing $\kappa=\delta^{1/2}$, the lemma follows.
\end{proof}

\vv
From now on, to simplify notation we will denote
\begin{equation}\label{eqdeltatilde}
\wt\delta = C_3\,A^{n+1} \,
\biggl(\ve_0 + \theta_0^{1/(n+1)} + \kappa_0^{1/2} + \delta^{1/2}\biggr).
\end{equation}
So the preceding lemma ensures that $\alpha^H_{\wt\mu}(3AQ_0)\leq \wt\delta$. We assume that the parameters
$\ve_0$,  $\theta_0$, $\kappa_0$, and $\delta$ are small enough so that $\wt\delta\ll1$.
\vv

\begin{lemma}\label{lem100}
We have
\begin{equation}\label{eqdki3}
\int_{Q_0}|\RR(\chi_{AQ_0}\wt\mu)|^2\,d\wt\mu\leq C_4\left(\ve + \frac1{A^2} + A^{4n+2} \,\delta^{\frac1{4n+4}}
+\ve_0+\theta_0^{\frac1{n+1}} + \kappa_0^{\frac12} + A^{2n+2}\,\wt\delta^{\frac2{4n+5}}
\right)\,\wt\mu(Q_0).
\end{equation}
\end{lemma}

\begin{proof}
Since $\wt\mu|_{Q_0} = \mu_0|_{Q_0}$, we have
\begin{equation}\label{eqsum42}
\int_{Q_0}|\RR(\chi_{AQ_0}\wt\mu)|^2\,d\wt\mu  \leq 
2\int_{Q_0}|\RR(\chi_{AQ_0}\mu_0)|^2\,d\mu_0 +
2\int_{Q_0}|\RR(\chi_{AQ_0}(\wt\mu-\mu_0)|^2\,d\mu_0.
\end{equation}
The first integral on the right hand side has been estimated in Lemma \ref{lemfac93}.
So we only have to deal with the second one. The arguments that we will use will be similar to some of the
ones in Lemma \ref{lem1}.

First, note that, using again that 
$\wt\mu|_{Q_0} = (\mu_0)|_{Q_0}$ and that $(\mu_0)|_{Q_0^c} = \mu|_{Q_0^c}$, we have
$$\RR\bigl(\chi_{AQ_0}(\wt\mu-\mu_0)\bigr) = \RR\bigl(\chi_{AQ_0\setminus Q_0}(\wt\mu-\mu)\bigr).$$
Let
 $0<\kappa<1/10$ be some small constant to be fixed below. Let $\vphi$ be a $C^1$ function
which equals $1$ on $(1-\kappa)AQ_0\setminus (1+\kappa)Q_0$ and vanishes out of $AQ_0\setminus (1+\frac\kappa2)Q_0$, with
$\|\nabla\vphi\|_\infty\lesssim (\kappa\,\ell(Q_0))^{-1}$.
We split
\begin{equation}\label{eq100.5}
\int_{Q_0}|\RR(\chi_{AQ_0}(\wt\mu-\mu_0)|^2\,d\mu_0
\leq 2\int_{Q_0}\!|\RR\bigl((\chi_{AQ_0\setminus Q_0}-\vphi)(\wt\mu-\mu)\bigr)|^2\,d\mu_0
+ 2\int_{Q_0}\!|\RR\bigl(\vphi(\wt\mu-\mu)\bigr)|^2\,d\mu_0.
\end{equation}
Concerning the first integral on the right hand side note that
$\chi_{AQ_0\setminus Q_0}-\vphi=\psi_1 + \psi_2$, with
$$|\psi_1|\leq \chi_{AQ_0\setminus (1-\kappa)AQ_0} \quad\mbox{ and }\quad
|\psi_2|\leq \chi_{(1+\kappa)Q_0\setminus Q_0}.$$
So we have
\begin{align*}
\int_{Q_0}|\RR\bigl((\chi_{AQ_0}-\vphi)(\wt\mu-\mu)\bigr)|^2\,d\mu_0 & \lesssim
\int_{Q_0}|\RR\bigl(\psi_1(\wt\mu-\mu)\bigr)|^2\,d\wt\mu+
\int_{Q_0}|\RR\bigl(\psi_2(\wt\mu-\mu)\bigr)|^2\,d\wt\mu\\
&\leq \|\RR\bigl(\psi_1(\wt\mu-\mu)\bigr)\|_{L^\infty(\wt\mu|_{Q_0})}^2\,\mu(Q_0) \\
&\quad+ 
\|\RR\bigl(\psi_2(\wt\mu-\mu)\bigr)\|_{L^4(\wt \mu|_{Q_0})}^2\,\mu(Q_0)^{1/2}.
\end{align*}
Since $\dist(\supp\psi_1,Q_0)\approx A\ell(Q_0)$, we get
$$\|\RR\bigl(\psi_1(\wt\mu-\mu)\bigr)\|_{L^\infty(\wt\mu|_{Q_0})}\lesssim \frac1{(A\ell(Q_0))^n} \,\|\psi_1\|_{L^1(\wt\mu+\mu)}
\leq \frac1{(A\ell(Q_0))^n} \,(\wt\mu+\mu)(AQ_0\setminus (1-\kappa)AQ_0).$$
Recall that in \rf{eq29} it has been shown that
$$\mu(AQ_0\setminus (1-\kappa)AQ_0)  \lesssim \left(\frac{A^{n}}\kappa\,\delta\, + \kappa\,A^{n}\right)\,\ell(Q_0)^n.
$$
To prove this we used the fact that $\alpha^H_\mu(3AQ_0)\leq\delta$, or more precisely, that
$\dist_{3AQ_0}(\mu,\LL_{H})\leq \delta$.
The same inequality holds replacing $\mu$ by $\wt\mu$ and $\delta$ by $\wt\delta$, as shown in Lemma
\ref{lemalfa}. So by arguments analogous to the ones in \rf{eq29} it follows that
$$\wt\mu(AQ_0\setminus (1-\kappa)AQ_0)  \lesssim \left(\frac{A^{n}}\kappa\,\wt\delta\, + \kappa\,A^{n}\right)\,\ell(Q_0)^n.
$$
So we deduce that
$$\|\RR\bigl(\psi_1(\wt\mu-\mu)\bigr)\|_{L^\infty(\wt\mu|_{Q_0})}\lesssim \frac{\delta+\wt\delta}\kappa + \kappa\lesssim \frac{\wt\delta}\kappa + \kappa,$$
taking into account that $\delta\leq\wt\delta$ for the last inequality.

Next we will estimate $\|\RR\bigl(\psi_2(\wt\mu-\mu)\bigr)\|_{L^4(\wt \mu|_{Q_0})}$. By the triangle inequality, we have
$$\|\RR\bigl(\psi_2(\wt\mu-\mu)\bigr)\|_{L^4(\wt\mu|_{Q_0})}\leq \|\RR_\mu\psi_2\|_{L^4(\mu|_{Q_0})}+\|
\RR_{\wt\mu} \psi_2\|_{L^4(\wt\mu|_{Q_0})}.$$
Recall that $\RR_\mu$ is bounded in $L^2(\mu|_{2Q_0})$, and so in $L^4(\mu|_{2Q_0})$, and that $\supp\psi_2\subset (1+\kappa Q_0)\setminus Q_0\subset 2Q_0$. Hence, using also the thin boundary 
property of $Q_0$, we obtain
$$\|\RR_\mu\psi_2\|_{L^4(\mu|_{Q_0})}^4\lesssim \|\psi_2\|_{L^4(\mu|_{2Q_0})}^4 \lesssim \mu((1+\kappa Q_0)\setminus Q_0) \lesssim \kappa\,\mu(Q_0).$$
We can apply the same argument to estimate $\|
\RR_{\wt\mu} \psi_2\|_{L^4(\wt\mu|_{Q_0})}$. This is due to the fact that $\RR_{\wt\mu}$ is bounded
in $L^2(\wt\mu|_{2Q_0})$. This is an easy consequence of the fact that, given two measures $\mu_1$ and
$\mu_2$ with growth of order $n$ such that, for $i=1,2$,  
$\RR_{\mu_i}
$ is bounded in $L^2(\mu_i)$, then $\RR_{\mu_1+\mu_2}$ is bounded in $L^2(\mu_1+\mu_2)$. For the proof, see Proposition 2.25 of \cite{Tolsa-llibre}, for example. So applying this result to a finite number of
translated copies of $\mu|_{Q_0}$, we deduce that $\RR_{\wt\mu}$ is bounded in $L^2(\wt\mu|_{2Q_0})$ and
so in $L^4(\wt\mu|_{2Q_0})$.
Thus we also have
$$\|\RR_{\wt\mu}\psi_2\|_{L^4(\wt\mu|_{Q_0})}^4 \lesssim \kappa\,\wt\mu(Q_0)\leq \kappa\,\mu(Q_0).$$
Gathering the estimates above, it turns out that the first integral on the right side of \rf{eq100.5}
satisfies the following:
\begin{equation}\label{eqpart1}
\int_{Q_0}|\RR(\chi_{AQ_0}(\wt\mu-\mu_0)|^2\,d\mu_0\lesssim
\left(\frac{\wt\delta}\kappa + \kappa\right)^2\mu(Q_0) + \kappa^{1/2} \,\mu(Q_0)\lesssim 
\left(\frac{\wt\delta}\kappa + \kappa^{1/4}\right)^2\mu(Q_0).
\end{equation}

It remains to estimate the second integral on the right hand side of \rf{eq100.5}.
To this end for any $x\in Q_0$ we set
\begin{align*}
|\RR\bigl(\vphi(\wt\mu-\mu_0)(x)\bigr)|& = \left|\int K(x-y)\,\vphi(y)\,d(\wt\mu-\mu)(y)\right|\\
& \leq \left|\int K(x-y)\,\vphi(y)\,d(\wt\mu-\LL_{H})(y)\right|
\\
&\quad +
\left|\int K(x-y)\,\vphi(y)\,d(\mu-\LL_{H})(y)\right|\\
& \leq \|\nabla(K(x-\cdot)\,\vphi)\|_\infty \,
\bigl[d_{3AQ_0}(\wt\mu,\LL_{H}) +
d_{3AQ_0}(\mu,\LL_{H})\bigr],
\end{align*}
where in the first identity we used the fact that $\mu_0$ coincides with $\mu$ on the support of $\vphi$.
Taking into account that $\dist(x,\supp\vphi)\gtrsim\kappa\ell(Q_0)$, we obtain
\begin{align*}
\|\nabla(K(x-\cdot)\,\vphi)\|_\infty & \leq \|\nabla K(x-\cdot)\|_{\infty,\supp\vphi} + 
\|K(x-\cdot)\|_{\infty,\supp\vphi}\,\|\nabla\vphi\|_\infty
\lesssim \frac1{(\kappa\,\ell(Q_0))^{n+1}}.
\end{align*}
By Lemma \ref{lemalfa}, $\dist_{3AQ_0}(\wt\mu,\LL_{H})\leq 
\wt\delta
\,\ell(3AQ_0)^{n+1}$ and, by the assumption (e) in the Main Lemma, $\dist_{3AQ_0}(\mu,\LL_{H})\leq 
\delta
\,\ell(3AQ_0)^{n+1}$. Therefore,
$$|\RR\bigl(\vphi(\wt\mu-\mu_0)(x)\bigr)|\lesssim \frac1{(\kappa\,\ell(Q_0))^{n+1}} (\delta+\wt\delta)
\,\ell(3AQ_0)^{n+1}\lesssim \frac{A^{n+1}}{\kappa^{n+1}}\,\wt\delta.$$
So the last integral on the right hand side of \rf{eq100.5} satisfies
\begin{equation}\label{eqpart2}
\int_{Q_0}|\RR(\chi_{AQ_0}(\wt\mu-\mu_0)|^2\,d\wt\mu\lesssim \left(\frac{A^{n+1}}{\kappa^{n+1}}\,\wt\delta\right)^2\mu(Q_0).
\end{equation}

From \rf{eq100.5}, \rf{eqpart1} and \rf{eqpart2} we deduce that
\begin{align*}
\int_{Q_0}|\RR(\chi_{AQ_0}(\wt\mu-\mu_0)|^2\,d\mu_0 &
\lesssim 
\biggl(\frac{\wt\delta}\kappa + \kappa^{1/4} + \frac{A^{n+1}}{\kappa^{n+1}}\,\wt\delta\biggr)^2\mu(Q_0) \\
&\lesssim \biggl(\kappa^{1/4} + \frac{A^{n+1}}{\kappa^{n+1}}\,\wt\delta\biggr)^2\mu(Q_0).
\end{align*}
Choosing $\kappa=\wt\delta^{\frac4{4n+5}}$, the right hand side above equals
$$\biggl(\wt\delta^{\frac1{4n+5}}+ A^{n+1}\,\wt\delta^{\frac1{4n+5}}\biggr)^2\mu(Q_0)
\lesssim A^{2n+2}\,\wt\delta^{\frac2{4n+5}}\,\mu(Q_0).$$
 Together with \rf{eqsum42} and Lemma \ref{lem100}, this yields \rf{eqdki3}.
\end{proof}
\vv

To simplify notation we will write
\begin{equation}\label{epsilontilde}
\wt\ve =C_4 \left(\ve + \frac1{A^2} +  A^{4n+2} \,\delta^{\frac1{4n+4}}
+\ve_0+\theta_0^{\frac1{n+1}} + \kappa_0^{\frac12} + A^{2n+2}\,\wt\delta^{\frac2{4n+5}}
\right),
\end{equation}
so that the preceding lemma guarantees that
$$\int_{Q_0}|\RR(\chi_{AQ_0}\wt\mu)|^2\,d\wt\mu\leq \wt\ve\,\wt\mu(Q_0).$$
\vv

We will also need the following auxiliary result below.

\begin{lemma}\label{lemaux23}
For all $Q\in\sss_0\setminus\bad$, we have
$$ \int_{1.1B_Q\setminus Q}\int_Q
\frac1{|x-y|^n}\,d\wt\mu(x)\,d\wt\mu(y)\lesssim 
\theta_0^{\frac1{2(n+1)^2}}
\,\wt\mu(Q). $$
\end{lemma}

\begin{proof}
Since for any $Q\in \sss_0\setminus\bad$ the ball $1.1B_Q$  is contained in $Q_0$,
we have that $\wt\mu$ coincides with $\mu_0$ in the above domain of integration.

Let $0<\kappa<1$ be some small constant to be fixed below. Then we split
\begin{align}\label{eqii89}
\int_{1.1B_Q\setminus Q}\int_Q
\frac1{|x-y|^n}\,d\wt\mu(y)\,d\wt\mu(x) & =
\int_{x\in1.1B_Q\setminus Q}\int_{y\in Q:|x-y|>\kappa \ell(Q)}
\frac1{|x-y|^n}\,d\wt\mu(y)\,d\wt\mu(x)\\
&\quad +\int_{x\in1.1B_Q\setminus Q}\int_{y\in Q:|x-y|\leq\kappa \ell(Q)}
\frac1{|x-y|^n}\,d\wt\mu(y)\,d\wt\mu(x).\nonumber
\end{align}
First we deal with the first integral on the right hand side:
\begin{align*}
\int_{x\in1.1B_Q\setminus Q}\int_{y\in Q:|x-y|>\kappa \ell(Q)}
\frac1{|x-y|^n}\,d\wt\mu(y)\,d\wt\mu(x)& \leq \frac1{\kappa^n\ell(Q)^n}\,\wt\mu(1.1B_Q)\,\wt\mu(Q)\\
&
\lesssim \frac1{\kappa^n}\,\Theta_{\wt\mu}(1.1B_Q)\,\wt\mu(Q)\lesssim \frac{\theta_0^{\frac1{n+1}}}{\kappa^n}\,\mu(Q),
\end{align*}
by Lemma \ref{lempoisson}.

Let us turn our attention to the last integral in \rf{eqii89}. To estimate this
we take into account the fact that given $x\in 1.1B_Q\setminus Q$, if $y\in Q$, then 
$|x-y|\geq \dist(x, Q)$. Then by the growth of order $n$ of $\mu|_{Q_0}$ and standard estimates, we get
$$\int_{y\in Q:|x-y|\leq\kappa \ell(Q)}
\frac1{|x-y|^n}\,d\wt\mu(y) \lesssim \log\biggl(2 + \frac{\kappa\,\ell(Q)}{\dist(x, Q)}\biggr)
\quad\mbox{ for all $x\in 1.1B_Q\setminus Q$.}
$$
For each $j\geq0$, denote
$$U_j = \bigl\{x\in 1.1B_Q\setminus Q:\dist(x,Q)\leq 2^{-j}\,\kappa\,\ell(Q)\bigr\}.$$
By the small boundary property of $Q$ and the fact that $Q$ is doubling,
$$\mu(U_j)\lesssim (2^{-j}\,\kappa)^{1/2}\,\mu(3.5B_Q)\lesssim (2^{-j}\,\kappa)^{1/2}\,\mu(Q).$$
Then we obtain
\begin{align*}
\int_{x\in1.1B_Q\setminus Q}\int_{y\in Q:|x-y|\leq\kappa \ell(Q)}
\frac1{|x-y|^n}\,d\wt\mu(y)\,d\wt\mu(x)& \leq 
\sum_{j\geq0} \int_{U_j\setminus U_{j+1}}
\log\biggl(2 + \frac{\kappa\,\ell(Q)}{\dist(x, Q)}\biggr)\,d\mu(x)\\
& \lesssim \sum_{j\geq0} \log\biggl(2 + \frac{\kappa\,\ell(Q)}{2^{-j-1}\kappa\ell(Q)}\biggr)
\mu(U_j)\\
& \lesssim \sum_{j\geq0} (j+1)
(2^{-j}\,\kappa)^{1/2}\,\mu(Q)\\
&\lesssim \kappa^{1/2}\,\mu(Q).
\end{align*}

Thus we have
$$ \int_{1.1B_Q\setminus Q}\int_Q
\frac1{|x-y|^n}\,d\wt\mu(x)\,d\wt\mu(y)\lesssim
\bigl(\theta_0^{1/(n+1)}
\,\kappa^{-n} + \kappa^{1/2}\bigr)\,\mu(Q).$$
Choosing $\kappa=\theta_0^{\frac1{(n+1)^2}}$, the lemma follows.
\end{proof}

\vv

It is easy to check that 
\begin{equation}\label{eqgrowth589}
\wt\mu(B(x,r))\lesssim r^n\quad\mbox{ for all $x\in\R^{n+1}$ and all $r>0$.}
\end{equation}
This follows from the analogous estimate for $\mu|_{Q_0}$ and the periodicity of $\wt\mu$,
and is left for the reader.
On the other hand, in general, we cannot guarantee that the estimates for the coefficients $P_\mu(2B_Q)$ in Lemma \ref{lempoisson}
also hold with $\mu$ replaced by $\wt\mu$. However, we have following substitute. 

\begin{lemma}\label{lempmu0}
The function
$$p _{\wt\mu}(x) = \sum_{Q\in \sss_0\setminus\bad}\chi_Q\,  P_{\wt\mu}(2B_Q)$$
satisfies
$$\int_{Q_0}p_{\wt\mu}^2\,\,d\wt\mu \lesssim \theta_0^{\frac2{3(n+1)}}\,\wt\mu(Q_0).$$
\end{lemma}

\begin{proof}
Once more, let $0<\kappa<1$ be some small constant to be fixed below. We split
\begin{equation}\label{eqpois589}
\int_{Q_0}p_{\wt\mu}^2\,d\wt\mu =\int_{x\in Q_0:\dist(x,\partial Q_0)\leq \kappa\,\ell(Q_0)}p_{\wt\mu}(x)^2\,d\wt\mu(x) 
+ \int_{x\in Q_0:\dist(x,\partial Q_0)> \kappa\,\ell(Q_0)}p_{\wt\mu}(x)^2\,d\wt\mu(x)
.\end{equation}
For the first integral on the right hand side we just take into account that $p_{\wt\mu}(x)\lesssim 1$ by \rf{eqgrowth589}, and thus
\begin{align*}
\int_{x\in Q_0:\dist(x,\partial Q_0)\leq \kappa\,\ell(Q_0)}p_{\wt\mu}(x)^2\,d\wt\mu(x)
&\lesssim \mu\bigl(\bigl\{x\in Q_0:\dist(x,\partial Q_0)\leq \kappa\,\ell(Q_0)\bigr\}\bigr) \\
&\lesssim \kappa\,\mu(Q_0)\approx \kappa
\,\wt\mu(Q_0).
\end{align*}

Let us deal with the the last integral on the right hand side of \rf{eqpois589}. Consider $x\in Q\in\sss_0$ such that
$\dist(x,\partial Q_0)> \kappa\,\ell(Q_0)$. We assume that $\kappa\gg t = \theta_0^{\frac1{(n+1)}}$. Since
$\ell(Q)\leq t\,\ell(Q_0)$, 
$$\dist(x,\partial Q_0)\approx \dist(2B_Q,\partial Q_0)\gtrsim \kappa\,\ell(Q_0).$$ 
Then,  we can write
$$p_{\wt\mu}(x) \lesssim P_\mu(2B_Q) + \sum_{j\geq1:2^jB_Q\cap\partial Q_0\neq\varnothing} 2^{-j}\,\Theta_{\wt\mu}(2^jB_Q)
\lesssim \theta_0^{\frac1{n+1}} + \sum_{j\geq1:2^jB_Q\cap\partial Q_0\neq\varnothing} 2^{-j},
$$
by Lemma \ref{lempoisson}. For the last sum we have 
$$\sum_{j\geq1:2^jB_Q\cap\partial Q_0\neq\varnothing} 2^{-j}\approx \frac{\ell(Q)}{\dist(x,\partial Q_0)} 
\lesssim \frac{t\,\ell(Q_0)}{\kappa\,\ell(Q_0)} = \frac{\theta_0^{\frac1{n+1}}}{\kappa},$$
and so we obtain 
$$p_{\wt\mu}(x) 
\lesssim \theta_0^{\frac1{n+1}} + \frac{\theta_0^{\frac1{n+1}}}{\kappa} \approx \frac{\theta_0^{\frac1{n+1}}}{\kappa}.
$$
Therefore,
$$\int_{x\in Q_0:\dist(x,\partial Q_0)\geq \kappa\,\ell(Q_0)}p_{\wt\mu}(x)^2\,d\wt\mu(x)\lesssim \frac{\theta_0^{\frac2{n+1}}}{\kappa^2}\,\wt\mu(Q_0).$$

Gathering the estimates above, we obtain
$$\int_{Q_0}p_{\wt\mu}^2\,d\wt\mu \lesssim \biggl(\kappa + \frac{\theta_0^{\frac2{n+1}}}{\kappa^2}\,\biggr)\,\wt\mu(Q_0).$$
Choosing $\kappa=\theta_0^{\frac2{3(n+1)}}$, the lemma follows.
\end{proof}

\vvv
% ***************************************************************************

\section{The approximating measure $\eta$}\label{sec8}

In this section we replace the measure $\wt\mu$ by a better behaved measure $\eta$. This new measure
is also periodic, and further it is absolutely continuous with respect to Lebesgue measure.
Roughly speaking, the fact that $\wt\mu$ is supported in a union of low density cells ensures that 
$\|\chi_{Q_0}\RR(\chi_{AQ_0}\eta)\|_{L^2(\eta)}$ is very similar to $\|\chi_{Q_0}\RR(\chi_{AQ_0}\wt\mu)\|_{L^2(\wt\mu)}$ and thus is very small. We will show that $\RR\eta$ is a well defined periodic function, because
of the periodicity of $\eta$. We will prove $\|\chi_{Q_0}\RR\eta\|_{L^2(\eta)}$ is also very small, besides
other related technical estimates which will be useful later in Section \ref{sec9}.

The main advantage of $\eta$ over $\wt\mu$ is that $\RR\eta$ is a continuous function, while
$\RR\wt\mu$ may fail to be continuous. The continuity of $\RR\eta$ will allow us to apply
a key maximum principle in the variational argument in Section \ref{sec9}. 

First we consider the measure
$$\eta_0 = \sum_{Q\in\sss_0\setminus\bad} \mu_0(Q)\,\frac{\HH^{n+1}|_{\tfrac14B  (Q)}}{\HH^{n+1}\bigl(\tfrac14B  (Q)\bigr)}.$$
So, in a sense, $\eta_0$ can be considered as an approximation of $\mu_0|_{Q_0}$ which is absolutely continuous with respect to $\HH^{n+1}$. 
Further, since the family $\sss_0$ is finite, the density of $\eta$ with respect to $\HH^{n+1}$ is
bounded.

Recall
that, by Remark \ref{remfac1}, the balls $\frac12B(Q)$, $Q\in\DD_{\sigma}$, are pairwise disjoint.
So the balls $\tfrac14B  (Q)$ in the sum above satisfy
$$\dist(\tfrac14B  (Q),\tfrac14B  (Q'))\geq \frac14\,\bigl[r(B(Q))+ r(B(Q'))\bigr]\quad
\mbox{ if $Q\neq Q'$}.$$

Now we  define the following periodic version of $\eta_0$.
Let $\M$ be the lattice of cubes from $\R^{n+1}$ introduced in Section \ref{secmutilde}.
Recall that, for $P\in\M$, $z_P$ stands for the center of $P$ and $T_P$ is the translation $T_P(x)=x+z_P$.
We define
$$\eta = \sum_{P\in\M} (T_P)_\#\eta_0.$$
In this way, $\eta$ can be considered as a kind of approximation of $\wt\mu$.

The following result should be compared to Lemma \ref{lem100}.

\begin{lemma}\label{lem200}
We have
$$\int_{Q_0}|\RR(\chi_{AQ_0}\eta)|^2\,d\eta\lesssim \ve'\,\eta(Q_0),$$
where
$\ve'= \wt\ve + A^n\,\kappa_0^{-2n-2}\,\theta_0^{\frac1{(n+1)^2}}$.
\end{lemma}

\begin{proof}
To simplify notation, we denote $\SSS=\sss_0\setminus \bad$.
We consider the function
$$f = \sum_{Q\in \SSS} m_{\wt\mu,Q}(\RR(\chi_{AQ_0}\wt\mu))\,\chi_Q.$$
It is clear that 
\begin{equation}\label{eqsk592}
\|f\|_{L^2(\wt\mu)}^2\leq \|\RR(\chi_{AQ_0}\wt\mu)\|_{L^2(\wt\mu|_{Q_0})}^2\leq \wt\ve\,
 \wt\mu(Q_0)=  \wt\ve\,\eta(Q_0).
 \end{equation}
 
For all $x\in\tfrac14B  (Q)$, $Q\in\SSS$, we write
\begin{align}\label{eqt123*}
\bigl|\RR(\chi_{AQ_0}\eta)(x)\bigr| & \leq \bigl|\RR(\chi_{\tfrac14B  (Q)}\eta)(x)\bigr| 
+  \bigl|\RR(\chi_{AQ_0\setminus \tfrac14B  (Q)}\eta)(x) - \RR(\chi_{AQ_0\setminus Q}\wt\mu))(x)\bigr|\\
&\quad + \bigl|\RR(\chi_{AQ_0\setminus Q}\wt\mu)(x)- m_{\wt\mu,Q}(\RR(\chi_{AQ_0}\wt\mu))\bigr|
+\bigl|m_{\wt\mu,Q}(\RR(\chi_{AQ_0}\wt\mu))\bigr|\nonumber\\
& =: T_1 + T_2 + T_3 + \bigl|m_{\wt\mu,Q}(\RR(\chi_{AQ_0}\wt\mu))\bigr|.\nonumber
\end{align}

Using that $$\eta|_{\tfrac14B  (Q)} = \wt\mu(Q)\,\frac{\HH^{n+1}|_{\tfrac14B  (Q)}}{\HH^{n+1}\bigl(\tfrac14B  (Q)\bigr)},$$
it follows easily that
$$T_1=\bigl|\RR(\chi_{\tfrac14B  (Q)}\eta)(x)\bigr|\lesssim \frac{\wt\mu(Q)}{r(B(Q))^n}\lesssim \theta_0^{1/(n+1)}.$$

Next we will deal with the term $T_3$ in \rf{eqt123*}.
To this end, for $x\in\tfrac14B  (Q)$  we set
\begin{align}\label{eqsd539}
\bigl|\RR(&\chi_{AQ_0\setminus Q}\wt\mu)(x) - m_{\wt\mu,Q}(\RR(\chi_{AQ_0}\wt\mu))\bigr|\\
& \leq \bigl|\RR(\chi_{1.1B_Q\setminus Q}\wt\mu)(x)\bigr| +
\bigl|\RR(\chi_{AQ_0\setminus 1.1B_Q}\wt\mu)(x) - m_{\wt\mu,Q}(\RR(\chi_{AQ_0\setminus 1.1B_Q}\wt\mu))\bigr| \nonumber\\
&\quad+ \bigl|m_{\wt\mu,Q}(\RR(\chi_{1.1B_Q\setminus Q}\wt\mu))\bigr|,\nonumber
\end{align}
taking into account that $m_{\wt\mu,Q}(\RR(\chi_Q\wt\mu))=0$, by the antisymmetry of the Riesz kernel. 
The first term on the right hand side satisfies
$$\bigl|\RR(\chi_{1.1B_Q\setminus Q}\wt\mu)(x)\bigr|\leq \int_{1.1B_Q\setminus Q}\frac1{|x-y|^n}\,d\wt\mu(y)
\lesssim \frac{\wt\mu(1.1B_Q)}{r(B(Q))^n}\lesssim \theta_0^{1/(n+1)},$$
recalling that $x\in \frac14B(Q)$ and that $\Theta_\mu(1.1B_Q)\lesssim \theta_0^{1/(n+1)}$ for the last estimate.

Now we turn our attention to the second term on the right hand side of \rf{eqsd539}. For $x'\in Q\in\SSS$, 
we have
\begin{align*}
\bigl|\RR(\chi_{AQ_0\setminus 1.1B_Q}\wt\mu)(x) -\RR(\chi_{AQ_0\setminus 1.1B_Q}\wt\mu)(x')\bigr|
&\leq \int_{AQ_0\setminus 1.1B_Q} \bigl|K(x-y)-K(x'-y)\bigr|\,d\wt\mu(y)\\
& \lesssim P_{\wt\mu}(2B_Q).
\end{align*}
taking into account that the distance  both from $x$ and $x'$ to $(1.1B_Q)^c$ is larger than $c\,r(B_Q)$. 
Averaging on $x'\in Q$ with respect to $\wt\mu$ we get
$$\bigl|\RR(\chi_{AQ_0\setminus 1.1B_Q}\wt\mu)(x) - m_{\wt\mu,Q}(\RR(\chi_{AQ_0\setminus 1.1B_Q}\wt\mu))\bigr|
\lesssim P_{\wt\mu}(2B_Q).$$

To estimate the last term in \rf{eqsd539} we just apply Lemma \ref{lemaux23}:
$$ \bigl|m_{\wt\mu,Q}(\RR(\chi_{1.1B_Q\setminus Q}\wt\mu))\bigr|\leq \frac1{\wt\mu(Q)} \int_{1.1B_Q\setminus Q}\int_Q
\frac1{|x-y|^n}\,d\wt\mu(x)\,d\wt\mu(y)\lesssim \theta_0^{\frac1{2(n+1)^2}}.$$
Thus, we obtain
\begin{equation}\label{eq1934}
T_3=
\bigl|\RR(\chi_{AQ_0\setminus Q}\wt\mu)(x) - m_{\wt\mu,Q}(\RR(\chi_{AQ_0}\wt\mu))\bigr|\lesssim \theta_0^{\frac1{n+1}} + 
 P_{\wt\mu}(2B_Q) + \theta_0^{\frac1{2(n+1)^2}}\lesssim \theta_0^{\frac1{2(n+1)^2}} + P_{\wt\mu}(2B_Q).
\end{equation}

To deal with the term $T_2$ in \rf{eqt123*} we need to introduce some additional notation. We set
$$J=\bigcup_{P\in\M} \bigl\{T_P(R'):R'\in \sss_0\setminus \bad\bigr\}.$$
For $R\in J$ such that $R=T_P(R')$, $R'\in \sss_0\setminus\bad$, we set $B(R)=T_P(B(R'))$ and
$B_{R} = T_P(B_{R'})$. Also, we denote by $J_A$ the family of cells $R\in J$ which are contained in $AQ_0$.
In this way, we have
$$\chi_{AQ_0}\wt\mu=\sum_{R\in J_A} \wt\mu|_R,\quad\mbox{ and }\quad
\chi_{AQ_0}\eta=\sum_{R\in J_A} \wt\mu(R)\,\frac{\HH^{n+1}|_{\tfrac14B  (R)}}{\HH^{n+1}\bigl(\tfrac14B  (R)\bigr)}$$
(recall that we assume $A$ to be a big odd natural number).
Note that the cells $R\in J$ are pairwise disjoint. Further, by the definition of the family $\bad$, if $R\in J$ is contained in some cube $T_P(Q_0)$, then
the ball $1.1B_R$ is also contained in $T_P(Q_0)$. Together with the doubling property of the cells from 
$\sss_0$ this guarantees that, for all $R\in J$,
$$\wt\mu(1.1B_R)\lesssim C_0\,\wt\mu(R).$$
Now for $x\in\tfrac14B  (Q)$ we write
\begin{align}\label{eqt2***}
T_2 & = \bigl|\RR(\chi_{AQ_0\setminus \tfrac14B  (Q)}\eta)(x) - \RR(\chi_{AQ_0\setminus Q}\wt\mu))(x)\bigr| \\
& \leq \sum_{R\in J_A:R\neq Q} \left|\int K(x-y)\,d(\eta_{\frac14B  (R)} - \wt\mu|_R)\right|\nonumber\\
& \leq
\sum_{R\in J_A:R\neq Q} \int |K(x-y) - K(x-z_R)|\,d(\eta_{\frac14B  (R)} + \wt\mu|_R),\nonumber
\end{align}
using that $\eta(\tfrac14B  (R)) = \wt\mu(R)$ for the last inequality.

We claim that, for $x\in \tfrac14B  (Q)$ and $y\in\tfrac14B  (R)\cup \supp(\wt\mu|_R)$,
\begin{equation}\label{eqdqr1}
|K(x-y) - K(x-z_R)|\lesssim \frac{\ell(R)}{\kappa_0^{n+1}\,D(Q,R)^{n+1}},
\end{equation}
where 
$$D(Q,R)=\ell(Q) + \ell(R) + \dist(Q,R).$$
To show \rf{eqdqr1} note first that
\begin{equation}\label{eq19050}
x\in \tfrac14B  (Q),\;\;z_R\in \tfrac14B  (R)\quad\Rightarrow \quad \quad|x-z_R|\gtrsim D(Q,R),
\end{equation}
since $\frac12B(Q)\cap \frac12B(R)=\varnothing$. Analogously, because of the same reason,
\begin{equation}\label{eq1905}
x\in \tfrac14B  (Q),\;\;y\in \tfrac14B  (R)\quad\Rightarrow \quad|x-y|\gtrsim D(Q,R).
\end{equation}
 Also,
\begin{equation}\label{eq1906}
x\in \frac14B(Q),\;\;y\in \supp(\wt\mu|_R)\quad\Rightarrow \quad|x-y|\gtrsim \kappa_0 D(Q,R),
\end{equation}
To prove this, note that
\begin{equation}\label{eq1903}
y\in \supp(\wt\mu|_R) = I_{\kappa_0}(R)\subset R,
\end{equation}
which implies that $y\not\in B(Q)$ and thus $|x-y|\geq \frac34 r(B(Q))\approx\ell(Q)$. In the case
$r(B(Q))\geq 2\kappa_0\ell(R)$
 this implies that
$$|x-y|\gtrsim \ell(Q) + \kappa_0\ell(R).$$
Otherwise, from \rf{eq1903}, since $z_Q\in\supp\wt\mu$ and $y\in I_{\kappa_0}(R)$, by the definition of $I_{\kappa_0}(R)$,
$$|z_Q-y|\geq \kappa_0\ell(R),$$
and then, as $|z_Q-x|\leq \frac14r(B(Q))\leq \frac12\kappa_0\ell(R)$, we infer that
$$|x-y|\geq |z_Q-y| - |z_Q-x|\geq \frac{\kappa_0}2\ell(R).$$
So in any case we have $|x-y|\gtrsim \kappa_0(\ell(Q) + \ell(R))$. It is easy to deduce
\rf{eq1906} from this estimate. We leave the details for the reader.

From \rf{eq19050}, \rf{eq1905}, and \rf{eq1906}, and the fact that $K(\cdot)$ is a standard
Calder\'on-Zygmund kernel, we get \rf{eqdqr1}.
Plugging this estimate into \rf{eqt2***}, we obtain
$$T_2\lesssim \frac1{\kappa_0^{n+1}}
\sum_{R\in J_A} \frac{\ell(R)\,\wt\mu(R)}{D(Q,R)^{n+1}}.$$

So from \rf{eqt123*} and the estimates for the terms $T_1$, $T_2$ and $T_3$, we infer that
for all $x\in\tfrac14B  (Q)$ with $Q\in\SSS$
\begin{equation}\label{eqnos58}
\bigl|\RR(\chi_{AQ_0}\eta)(x)\bigr|  \lesssim  \bigl|m_{\wt\mu,Q}(\RR(\chi_{AQ_0}\wt\mu))\bigr| +
\theta_0^{\frac1{2(n+1)^2}} + P_{\wt\mu}(2B_Q)+
\frac1{\kappa_0^{n+1}}
\sum_{R\in J_A} \frac{\ell(R)\,\wt\mu(R)}{D(Q,R)^{n+1}}.
\end{equation}
Denote 
$$ \wt p _{\wt\mu}(x) = \sum_{Q\in J}\chi_{\frac14B(Q)}\,  P_{\wt\mu}(2B_Q) \quad\mbox{ and }\quad
\wt g(x) = \sum_{Q\in\SSS}
\sum_{R\in J_A} \frac{\ell(R)}{D(Q,R)^{n+1}} \,\wt\mu(R)\,\chi_{\frac14B(Q)}(x).$$ 
Squaring and integrating \rf{eqnos58} with respect to $\eta$ on $Q_0$, we get
\begin{align}\label{eqff593}
\bigl\|\RR(\chi_{AQ_0}\eta)\bigr\|_{L^2(\eta|_{Q_0})}^2  & \lesssim  
\sum_{Q\in\SSS} \bigl|m_{\wt\mu,Q}(\RR(\chi_{AQ_0}\wt\mu))\bigr|^2\,\eta(\tfrac14B(Q)) \\ &\quad +
\theta_0^{\frac1{(n+1)^2}}\,\eta(Q_0) + \|\wt p_{\wt\mu}\|_{L^2(\eta|_{Q_0})}^2+ \frac1{\kappa_0^{2n+2}}\|\wt g\|_{L^2(\eta|_{Q_0})}^2,\nonumber
\end{align}
Note that, since $\eta(\tfrac14B(Q))=\wt\mu(Q)$, the first sum on the right hand side of \rf{eqff593} 
equals $\|f\|_{L^2(\wt\mu)}^2$, which does not exceed $\wt\ve\,\eta(Q_0)$, by \rf{eqsk592}.
By an analogous argument we 
deduce that $\|\wt p_{\wt\mu}\|_{L^2(\eta|_{Q_0})}^2 = \| p_{\wt\mu}\|_{L^2(\wt\mu|_{Q_0})}^2$ 
and $\|\wt g\|_{L^2(\eta|_{Q_0})}^2 = \| g\|_{L^2(\wt\mu|_{Q_0})}^2$, where
$$
 p _{\wt\mu}(x) = \sum_{Q\in J}\chi_Q\,  P_{\wt\mu}(2B_Q) \quad\mbox{ and }\quad g(x) = \sum_{Q\in\SSS}
\sum_{R\in J_A} \frac{\ell(R)}{D(Q,R)^{n+1}} \,\wt\mu(R)\,\chi_Q(x).
$$

We will estimate $\|g\|_{L^2(\wt\mu|_{Q_0})}$ by duality: for any non-negative function $h\in L^2(\wt\mu|_{Q_0})$,
we set
\begin{equation}\label{eqgh57}
\int g\,h\,d\wt\mu = 
\sum_{Q\in\SSS}
\sum_{R\in J_A} \frac{\ell(R)}{D(Q,R)^{n+1}} \,\wt\mu(R)\,\int_Q h\,d\wt\mu
= \sum_{R\in J_A} \wt\mu(R) \sum_{Q\in\SSS}
 \frac{\ell(R)}{D(Q,R)^{n+1}} \,\int_Q h\,d\wt\mu.
 \end{equation}
%Recall now that $\wt\mu|_{Q_0}=\mu_0$. 
 For each $z\in R\in J_A$ we have 
\begin{align*}
\sum_{Q\in\SSS} \frac{\ell(R)}{D(Q,R)^{n+1}} \,\int_Q h\,d\wt\mu &\lesssim
\int \frac{\ell(R)\,h(y)}{\bigl(\ell(R)+|z-y|\bigr)^{n+1}}\,d\wt\mu(y)\\
& = \int_{|z-y|\leq \ell(R)}\cdots + \sum_{j\geq1}\int_{2^{j-1}\ell(R)<|z-y|\leq 2^{j-1}\ell(R)}\cdots \\
& \lesssim 
\sum_{j\geq0}\; \avint_{B(z,2^{j}\ell(R))}h\,d\wt\mu \;\,\frac{2^{-j}\,\wt\mu(B(z,2^j\ell(R)))}{\bigl(2^j\ell(R)
\bigr)^n} \\
&\lesssim M_{\wt\mu}h(z)\,P_{\wt\mu}\bigl(B(z,\ell(R))\bigr),
\end{align*}
where $M_{\wt\mu}$ stands for the centered maximal Hardy-Littlewood operator with respect to $\wt\mu$. Then, by \rf{eqgh57},
\begin{align*}
\int g\,h\,d\wt\mu &\lesssim
 \sum_{R\in J_A} \inf_{z\in R} \big[M_{\wt\mu}h(z)\,P_{\wt\mu}\bigl(B(z,\ell(R))\bigr)\bigl]\,
 \wt\mu(R) \leq \int_{AQ_0}M_{\wt\mu}h\,\,p_{\wt\mu} \,d\wt\mu
 \\ 
 & \lesssim \|M_{\wt\mu}h\|_{L^2(\wt\mu)}\,\| p_{\wt\mu} \|_{L^2(\wt\mu|_{AQ_0})}
 \lesssim \|h\|_{L^2(\wt\mu)}\,\| p_{\wt\mu} \|_{L^2(\wt\mu|_{AQ_0})}
 .
 \end{align*}
Thus, by Lemma \ref{lempmu0} and recalling that $\wt\mu$ is $\M$-periodic,
$$
\|g\|_{L^2(\wt\mu|_{Q_0})}^2\lesssim\|p_{\wt\mu} \|_{L^2(\wt\mu|_{AQ_0})}^2 = A^n\,\|p_{\wt\mu} \|_{L^2(\wt\mu|_{Q_0})}^2 \lesssim A^n\,\theta_0^{\frac2{3(n+1)}}\,\wt\mu(Q_0).$$
Plugging this estimate into \rf{eqff593} and recalling that $\|g\|_{L^2(\eta|_{Q_0})}
=\|g\|_{L^2(\wt\mu)}$, we obtain
$$\bigl\|\RR(\chi_{AQ_0}\eta)\bigr\|_{L^2(\eta|_{Q_0})}^2  \lesssim 
\biggl(\wt\ve +\theta_0^{\frac1{(n+1)^2}}+ \frac{A^n}{\kappa_0^{2n+2}}\,\theta_0^{\frac2{3(n+1)}}\biggr)
\,\eta(Q_0) \lesssim \biggl(\wt\ve + \frac{A^n}{\kappa_0^{2n+2}}\,\theta_0^{\frac1{(n+1)^2}}\biggr)
\,\eta(Q_0),$$
as wished.
\end{proof}

\vv

Note that the Riesz kernel is locally integrable with respect to $\eta$ (recall that the number of cells
from $\sss_0$ is finite). So for any bounded function $f$ with compact support the integral
$\int K(x-y)\,f(y)\,d\eta(y)$ is absolutely convergent for all $x\in\R^{n+1}$. 

Now we wish to extend the definition of $\RR_\eta f(x)$ to $\M$-periodic functions $f\in L^\infty(\eta)$ in a pointwise way
(not only in a $BMO$ sense, say). We consider a non-negative radial $C^1$ function $\phi$ supported on $B(0,2)$ which equals $1$ on 
$B(0,1)$, and we set $\phi_r(x)=\phi\left(\frac {x}{r}\right)$ for $r>0$. We denote $\wt K_r(x-y) = K(x-y)\phi_r(x-y)$ and we define
$$
\wt \RR_{\eta,r} f(x) = \wt \RR_{r} (f\eta)(x) = \int \wt K_r(x-y)\,f(y)\,d\eta(y),
$$
and
\begin{equation}\label{eq1001}
\pv\RR_\eta f(x) = {\rm pv}\RR (f\eta)(x) = \lim_{r\to\infty} \wt\RR_{\eta,r}f(x),
\end{equation}
whenever the limit exists. Let us remark that one may also define the principal value in a more typical way by
\begin{equation}\label{eq1002}
\lim_{r\to\infty} \int_{|x-y|<r} K(x-y)\,f(y)\,d\eta(y).
\end{equation}
However, because of the smoothness of the kernels $\wt K_r$, 
 the definition \rf{eq1001} has the advantage that some of the estimates below involving the truncated operators $\RR_r$ 
 are easier than the analogous ones with the kernels $\wt K_r(x)$ replaced by $K(x)\,\chi_{|x|< r}$.
 Nevertheless, one can show
that both definitions \rf{eq1001} and \rf{eq1002} coincide, at least for the $\M$-periodic functions $f\in L^\infty(\eta)$
(we will not prove this fact because this will be not needed below).

\vv

\begin{lemma}\label{lemtyp}
Let $f\in L^\infty(\eta)$ be $\M$-periodic, that is, $f(x+z_P)=f(x)$ for all $x\in\R^{n+1}$ and all $P\in\M$. Then:
\begin{itemize}
\item $\pv\RR_\eta f(x)$ exists for {\em all} $x\in\R^{n+1}$ and $\wt\RR_{\eta,r}f\to \pv\RR_\eta f$ as $r\to\infty$ uniformly in compact subsets of $\R^{n+1}$. The convergence is also uniform on $\supp\eta$.
Further, given any compact set $F\subset\R^{n+1}$, there is $r_0=r_0(F)>0$ such that for $s> r\geq r_0$,
\begin{equation}\label{eqdif94}
\bigr\|\wt\RR_s(f\eta) - \wt\RR_r(f\eta)\bigl\|_{\infty,F}\,\lesssim  \frac{c_{F}}r\,\|f\|_\infty,
\end{equation}
where $c_{F}$ is some constant depending on $F$.

\item The function $\pv\RR_\eta f$ is $\M$-periodic and continuous in $\R^{n+1}$, and harmonic in $\R^{n+1}\setminus \supp (f\eta)$.
\end{itemize}
\end{lemma}

The arguments to prove the lemma are standard. However, for the reader's convenience we will show the details.

\begin{proof}
By the $\M$-periodicity of the measure $\nu:=f\,\eta$, it is immediate that the functions 
$\wt\RR_r(f\eta)$, $r>0$, are $\M$-periodic too. On the other hand, 
using that
$\eta$ is absolutely continuous with respect
to Lebesgue measure with a uniformly bounded density,
it is straightforward to check that each $\wt\RR_r(f\eta)$ is also continuous and bounded in $\R^{n+1}$. Then, except 
harmonicity in $\R^{n+1}\setminus \supp (f\eta)$, 
all the statements in the lemma follow if we show that the family of functions $\{\wt\RR_r(f\eta)\}_{r>0}$
satisfies \rf{eqdif94} for any compact subset  $F\subset\R^{n+1}$. 
Indeed this clearly implies  the uniform
convergence on compact subsets and on $\supp\eta$, since $\supp\eta$ is $\M$-periodic, 

Let $s>r\geq r_0$ and, denote $\wt K_{r,s}(x-y)=\wt K_{s}(x-y) - \wt K_{r}(x-y)$. 
Notice that $\wt K_{r,s}$ is a standard Calder\'on-Zygmund kernel (with constants independent of
$r$ and $s$).
We write
$$\nu=\sum_{P\in\M} (T_P)_\#(\chi_{Q_0}\nu),$$
so that
$$\wt\RR_s(f\eta)(x) - \wt\RR_r(f\eta)(x) = \int\wt K_{r,s}(x-y)\,d\Bigl(\sum_{P\in\M}(T_P)_\#(\chi_{Q_0}\nu)\Bigr)(y).$$
Since the support of $\wt K_{r,s}(x-y)$ is compact, the last sum only has a finite number of non-zero
terms, and so we can change the order of summation and integration:
\begin{align}\label{eqak55}
\wt\RR_s(f\eta)(x) - \wt\RR_r(f\eta)(x) & = \sum_{P\in\M}\int\wt K_{r,s}(x-y)\,d\bigl[(T_P)_\#(\chi_{Q_0}\nu)\bigr](y)
\\
& = \sum_{P\in\M}\int_{Q_0}\wt K_{r,s}(x-y-z_P)\,d\nu(y).\nonumber
\end{align}
By the antisymmetry of the kernel $\wt K_{r,s}$, from the last equation we derive
$$\wt\RR_s(f\eta)(x) - \wt\RR_r(f\eta)(x) = -\sum_{P\in\M}\int_{Q_0}\wt K_{r,s}(z_P-(x-y))\,d\nu(y).$$
Also, by the definition of $\M$, it is clear that $P\in\M$ if and only if $-P\in\M$. So replacing
$z_P$ by $-z_P$ does not change the last sum in \rf{eqak55}. Hence we have
$$\wt\RR_s(f\eta)(x) - \wt\RR_r(f\eta)(x) = 
\sum_{P\in\M}\int_{Q_0}\wt K_{r,s}(z_P+(x-y))\,d\nu(y).$$
Averaging the last two equations we get
\begin{equation}\label{eqsk96}
\wt\RR_s(f\eta)(x) - \wt\RR_r(f\eta)(x) = \frac12
\sum_{P\in\M} \int_{Q_0}\bigl[\wt K_{r,s}(z_P+(x-y)) - \wt K_{r,s}(z_P-(x-y))\bigr]\,d\nu(y).
\end{equation}

Note that if $x$ belongs to a compact set $F\subset\R^{n+1}$ and $y\in Q_0$, then both 
$(x-y)$ and $-(x-y)$ lie in some compact set $\wt F$.
Observe also that $\wt K_{r,s}$ vanishes in $B(0,r)$.
So if we assume $r_0\geq 2\,\diam(\wt F)$, say, then both 
$\wt K_{r,s}(z_P+(x-y))$ and $\wt K_{r,s}(z_P-(x-y))$ vanish unless $|z_P|\geq r$.
For such $x,y$ we have $|x-y|\leq \diam(\wt F)\leq \frac 12\,r\leq\frac 12\,|z_P|$ and so
$$|z_P+(x-y)|\approx |z_P+(x-y)|\approx |z_P|\geq r.$$
Then we obtain
$$\bigl|\wt K_{r,s}(z_P+(x-y)) - \wt K_{r,s}(z_P-(x-y))\bigr|\lesssim \frac{|x-y|}{|z_P|^{n+1}}
\lesssim \frac{\diam(\wt F)}{|z_P|^{n+1}}.$$
Plugging this estimate into \rf{eqsk96} we obtain
$$
\bigl|\wt\RR_s(f\eta)(x) - \wt\RR_r(f\eta)(x)\bigr|\lesssim
\sum_{P\in\M:|z_P|\geq r} \frac{\diam(\wt F)}{|z_P|^{n+1}}\,|\nu|(Q_0)
\leq \sum_{P\in\M:|z_P|\geq r} \frac{\diam(\wt F)}{|z_P|^{n+1}}\,\ell(P)^n\,\|f\|_\infty.$$
It is easy to check that 
$$\sum_{P\in\M:|z_P|\geq r} \frac{\ell(P)^n}{|z_P|^{n+1}}\lesssim \frac1r.$$
So we deduce 
$$\bigr\|\wt\RR_s(f\eta) - \wt\RR_r(f\eta)\bigl\|_{\infty,F}\lesssim  \frac{\diam(\wt F)}r\,\|f\|_\infty\to 0\quad \mbox{ as $r\to\infty$},$$
as wished.
\vv

It remains to prove that $\pv\RR_\nu f$ is harmonic in $\R^{n+1}\setminus \supp (f\eta)$.
Consider a closed ball $\overline B(0,r_1)$ and  $x\in \overline B(0,r_1)$. Then we have
$$\RR(f\phi_r\eta)(x) - \wt\RR_r(f\eta)(x) = \int K(x-y)\,\bigl(\phi_r(y)-\phi_r(x-y)\bigr)\,f(y)\,d\eta(y).$$
We write
$$|\phi_r(y)-\phi_r(x-y)|\lesssim \|\nabla\phi_r\|_\infty\,|x|\lesssim \frac{|x|}r.$$
For $r\geq 4\,r_1$, it is easy to check that
$\phi_r(y)-\phi_r(x-y)=0$ unless $|x-y|\approx|y|\approx r$. Thus
\begin{align*}
\bigl|\RR(f\phi_r\eta)(x) - \wt\RR_r(f\eta)(x)\bigr| &\lesssim 
\int_{\begin{subarray}{l}|y|\leq Cr\\
C^{-1}r\leq|x-y|\leq Cr \end{subarray}}
\frac1{|x-y|^n}\,\frac{|x|}r\,|f(y)|\,d\eta(y)\\
&\lesssim \frac{|x|}{r^{n+1}}\,\|f\|_\infty\,\eta(B(0,Cr))\lesssim 
\frac{r_1}{r}\,\|f\|_\infty.
\end{align*}
That is to say,
$$\bigl\|\RR(f\phi_r\eta) - \wt\RR_r(f\eta)\bigr\|_{\infty,\overline B(0,r_1)}\lesssim 
\frac{r_1}{r}\,\|f\|_\infty\to 0\quad \mbox{ as $r\to\infty$.}$$
Since $\wt\RR_r(f\eta)$ converges uniformly to $\pv\RR_\eta f$ in $F$ as $r\to\infty$, it follows
that $\RR(f\phi_r\eta)$ also converges uniformly to $\pv\RR_\eta f$ in $\overline B(0,r_1)$. 

Note now that $\RR(f\phi_r\eta)$ is harmonic out of $\supp(f\eta)$ because $f\phi_r\eta$ has compact
support, and so by its local uniform convergence to $\pv\RR_\eta f$, we deduce that $\pv\RR_\eta f$ is harmonic out of $\supp(f\eta)$ too.
\end{proof}
\vv

From now on, to simplify notation we will denote $\pv\RR_\eta f$ just by $\RR_\eta f$.
\vv
\begin{lemma}\label{lem28}
Let $L_\M^\infty(\eta)$ denote the Banach space of the $\M$-periodic functions which belong to 
$L^\infty(\eta)$ equipped with the norm $\|\cdot\|_{L^\infty(\eta)}$. The map
$\RR_\eta:L_\M^\infty(\eta)\to L_\M^\infty(\eta)$ is bounded. Further, for all $f\in L_\M^\infty(\eta)$
and $r>0$ big enough we have
\begin{equation}\label{eqa901}
\|\RR(f\eta)-\wt\RR_r(f\eta)\|_{L^\infty(\eta)} \lesssim \frac{\|f\|_{L^\infty(\eta)}}r.
\end{equation}
\end{lemma}

We remark that the bound on the norm of $\RR_\eta$ from  $L_\M^\infty(\eta)$ to 
$L_\M^\infty(\eta)$ depends strongly on the construction of $\eta$. This is finite due to the fact that the number of cells from $\sss_0$ is finite, but it may explode as  this number grows. The precise value of the norm will not play any role in the estimates below, we just need to know that this is finite.

\begin{proof}
Since $f$ is $\M$-periodic, from \rf{eqdif94} we infer that for $s> r\geq r_0=r_0(Q_0)$,
$$
\bigr\|\wt\RR_s(f\eta) - \wt\RR_r(f\eta)\bigl\|_{\infty,F}\lesssim  \frac{c_{F}}r\,\|f\|_\infty,
$$
Letting $s\to\infty$, $\wt\RR_s(f\eta)$ converges uniformly to $\RR\nu$ and so 
we get \rf{eqa901}.

To prove the boundedness of  $\RR_\eta:L_\M^\infty(\eta)\to L_\M^\infty(\eta)$, note first that
 $\wt K_{r_0}$ is compactly supported and $\eta$ is absolutely continuous with respect
to Lebesgue measure on a compact set with a uniformly bounded density. Hence we deduce that $\wt\RR_{\eta,r_0}:L_\M^\infty(\eta)\to L_\M^\infty(\eta)$ is bounded, which together with \rf{eqa901} applied to $\wt\RR_{r_0}$ implies that 
$\RR_\eta:L_\M^\infty(\eta)\to L_\M^\infty(\eta)$ is bounded.
\end{proof}

\vv

From now on, given $x\in\R^{n+1}$, we denote 
$$x_H =(x_1,\cdots,x_n),$$ 
so that $x=(x_H,x_{n+1})$. Also, we write
$$\RR^H = (\RR_1,\ldots,\RR_n),$$
where $\RR_j$ stands for the $j$-th component of $\RR$,
so that $\RR=(\RR^H,\RR_{n+1})$.

For simplicity, in the arguments below we will assume that the function $\phi$ defined slightly above \rf{eq1001} is of the form $\phi(x)=\wt\phi(|x|^2)$, for some even $C^1$ function $\wt\phi$ which 
equals $1$ on $[0,1]$ and vanishes out of $[0,2^{1/2}]$.

\vv

\begin{lemma}\label{lem83}
Let $f\in L^1_{loc}(\eta)$ be $\M$-periodic. Then,
\begin{itemize}
\item[(a)] Let $\wt A\geq5$ be some odd natural number. For all $x\in 2Q_0$,
$$\bigl|\RR(\chi_{(\wt AQ_0)^c}f\eta)(x)\bigr|\lesssim  \frac1{\wt A\,\ell(Q_0)^n}\,\int_{Q_0} |f|\,d\eta.$$

\item[(b)] For all $x\in\R^{n+1}$ such that $\dist(x,H)\geq \ell(Q_0)$,
\begin{equation}\label{eqb11}
\bigl|\RR(f\eta)(x)\bigr|\lesssim  \frac1{\ell(Q_0)^n}\,\int_{Q_0} |f|\,d\eta
\end{equation}
and
\begin{equation}\label{eqb12}
\bigl|\RR^H(f\eta)(x)\bigr|\lesssim  \frac{1}{\dist(x,H)\,\ell(Q_0)^{n-1}}\,\int_{Q_0} |f|\,d\eta
\end{equation}
\end{itemize}
\end{lemma}

\begin{proof}We denote $\nu = f\eta$.
The arguments to prove the estimate in (a) are quite similar to the ones used  in the proof of 
Lemma \ref{lemtyp}. Since we are assuming that $\wt A$ is some odd number, there is a subset
$\M_{\wt A}\subset\M$ such that 
$$\chi_{({\wt A}Q_0)^c}\nu = \sum_{P\in\M_{\wt A}} (T_P)_\#(\chi_{Q_0}\nu).$$
Further the cubes $P\in\M_{\wt A}$ satisfy $|z_P|\gtrsim {\wt A}\ell(Q_0)$.
So for any $x\in Q_0$ and all $r>0$ we have
$$\wt\RR_r(\chi_{({\wt A}Q_0)^c}\nu)(x) = 
\int\wt K_r(x-y)\,d\Bigl(\sum_{P\in\M_{\wt A}}(T_P)_\#(\chi_{Q_0}\nu)\Bigr)(y).$$
Since the support of $\wt K_{r}(x-y)$ is compact, the last sum only has a finite number of non-zero
terms, and so we can change the order of summation and integration, and thus
\begin{align}\label{eqak556}
\wt\RR_r(\chi_{({\wt A}Q_0)^c}\nu)(x) & = \sum_{P\in\M_{\wt A}}\int\wt K_{r}(x-y)\,d\bigl[(T_P)_\#(\chi_{Q_0}\nu)\bigr](y)
\\
& = \sum_{P\in\M_{\wt A}}\int_{Q_0}\wt K_{r}(x-y-z_P)\,d\nu(y).\nonumber
\end{align}
By the antisymmetry of the kernel $\wt K_{r}$, from the last equation we get
$$ \wt\RR_r(\chi_{({\wt A}Q_0)^c}\nu)(x) = -\sum_{P\in\M_{\wt A}}\int_{Q_0}\wt K_{r}(z_P-(x-y))\,d\nu(y).$$
 Also, by the definition of $\M_{\wt A}$, it follows that $P\in\M_{\wt A}$ if and only if $-P\in\M_{\wt A}$. So replacing
$z_P$ by $-z_P$ does not change the last sum in \rf{eqak556}, and then we have
$$\wt\RR_r(\chi_{({\wt A}Q_0)^c}\nu)(x) = 
\sum_{P\in\M_{\wt A}}\int_{Q_0}\wt K_{r}(z_P+(x-y))\,d\nu(y).$$
Averaging the last two equations we get
\begin{equation}\label{eqsk96'}
\wt\RR_r(\chi_{({\wt A}Q_0)^c}\nu)(x) = \frac12
\sum_{P\in\M_{\wt A}} \int_{Q_0}\bigl[\wt K_{r}(z_P+(x-y)) - \wt K_{r}(z_P-(x-y))\bigr]\,d\nu(y).
\end{equation}

Note now that $x\in2Q_0$, $y\in Q_0$ and, recalling that $|z_P|\gtrsim {\wt A}\ell(Q_0)$ for $P\in \M_{\wt A}$, 
 we have
$$|z_P+(x-y)| \approx |z_P-(x-y)|\approx|z_P|.$$
Thus,
$$\bigl|\wt K_{r}(z_P+(x-y)) - \wt K_{r}(z_P-(x-y))\bigr| \lesssim \frac{|x-y|}{|z_P|^{n+1}}\lesssim
\frac{\ell(Q_0)}{|z_P|^{n+1}}.$$
Then, from this estimate and \rf{eqsk96'} we deduce that
$$
\bigl|\wt\RR_r(\chi_{({\wt A}Q_0)^c}\nu)(x)\bigr|\lesssim
\sum_{P\in\M:|z_P|\geq C^{-1}{\wt A}\ell(Q_0)} \frac{\ell(Q_0)}{|z_P|^{n+1}}\,|\nu|(Q_0) \lesssim 
\frac{|\nu|(Q_0)}{{\wt A}\,\ell(Q_0)^n}.$$
as wished.

\vv
To prove the first estimate in (b), let $x\in\R^{n+1}$ be such that $\dist(x,H)\geq \ell(Q_0)$. Since 
$\RR\nu$ is $\M$-periodic, we may assume that $x_H\in Q_0\cap H$. 
 As in \rf{eqsk96'}, for any $r>0$ we have
\begin{equation}\label{eqsj54}
\wt\RR_r\nu(x) = \frac12
\sum_{P\in\M} \int_{Q_0}\bigl[\wt K_{r}(z_P+(x-y)) - \wt K_{r}(z_P-(x-y))\bigr]\,d\nu(y).
\end{equation}
We claim that for $x$ as above and $y\in Q_0$,
\begin{equation}\label{claim64}
\bigl|\wt K_{r}(z_P+(x-y)) - \wt K_{r}(z_P-(x-y))\bigr|\lesssim \frac{\dist(x,H)}{
\bigl(\dist(x,H) + |z_P|\bigr)^{n+1}}.
\end{equation}
Indeed, if $|z_P|\geq 2|x-y|$, then $\dist(x,H) + |x-z_P|\approx |x-y| +|z_P|\approx |z_P|$, and thus
$$\bigl|\wt K_{r}(z_P+(x-y)) - \wt K_{r}(z_P-(x-y))\bigr|\lesssim 
\frac{|x-y|}{|z_P|^{n+1}}\approx \frac{\dist(x,H)}{|z_P|^{n+1}}\approx \frac{\dist(x,H)}{
\bigl(\dist(x,H) + |z_P|\bigr)^{n+1}},
$$
which shows that \rf{claim64} holds in this case.

On the other hand, if $|z_P|< 2|x-y|$, then
$$\bigl|\wt K_{r}(z_P+(x-y)) - \wt K_{r}(z_P-(x-y))\bigr|\lesssim 
\frac 1{\bigl|(z_P - y)+ x|^{n}} + \frac 1{\bigl|(z_P+y)-x|^{n}}.$$ 
It is immediate to check that $\dist(x,y-z_P)\approx\dist(x,z_P+y)\approx\dist(x,H)\approx |x-y|$,
and so we deduce that
\begin{align*}
\bigl|\wt K_{r}(z_P+(x-y)) - \wt K_{r}(z_P-(x-y))\bigr| &\lesssim \frac1{|x-y|^n}.
 \end{align*}
Further from the condition $|z_P|< 2|x-y|$ we infer that 
$$|x-y|\approx |x-y| + |z_P|
\approx \dist(x,H)+|z_P|,$$
and thus
\begin{align*}
\bigl|\wt K_{r}(z_P+(x-y)) - \wt K_{r}(z_P-(x-y))\bigr| &\lesssim 
\frac{|x-y|}{\bigl(|x-y| + |z_P|\bigr)^{n+1}}\approx
\frac{\dist(x,H)}{
\bigl(\dist(x,H) + |z_P|\bigr)^{n+1}},
\end{align*}
which completes the proof of \rf{claim64}.

From \rf{eqsj54} and \rf{claim64} we deduce
\begin{align*}
\bigl|\wt\RR_r\nu(x)\bigr| & \lesssim 
\sum_{P\in\M} \int_{Q_0}\frac{\dist(x,H)}{
\bigl(\dist(x,H) + |z_P|\bigr)^{n+1}}\,d|\nu|(y)\\
&=|\nu|(Q_0)\,
 \frac{\dist(x,H)}{\ell(P)^n} \sum_{P\in\M} \frac{\ell(P)^n}
{
\bigl(\dist(x,H) + |z_P|\bigr)^{n+1}}
\end{align*}
It is straightforward to check that
$$\sum_{P\in\M} \frac{\ell(P)^n}
{
\bigl(\dist(x,H) + |z_P|\bigr)^{n+1}}\lesssim \frac1{\dist(x,H)},$$
and thus \rf{eqb11} follows.

\vv
We turn now our attention to the last estimate in (b). 
Again let $x\in\R^{n+1}$ be such that $\dist(x,H)\geq \ell(Q_0)$ and $x_H\in Q_0\cap H$, so that
the identity \rf{eqsj54} is still valid.
We claim that for $y\in Q_0$ and $r$ big enough,
\begin{equation}\label{claim64'}
\bigl|\wt K_{r}^H(z_P+(x-y)) - \wt K_{r}^H(z_P-(x-y))\bigr|\lesssim \frac{\ell(Q_0)}{
\bigl(\dist(x,H) + |z_P|\bigr)^{n+1}},
\end{equation}
where $\wt K_{r}^H$ is the kernel of $\wt\RR_r^H$.
To prove this observe that
$$\wt K_{r}^H(z) = z_H\,\psi_r(|z|^2),\quad \mbox{ 
with\;
$\psi_r(t) = \dfrac{\wt\phi_r(t)}{t^{\frac{n+1}2}}$. }$$
Then we have
\begin{align*}
\bigl|&\wt K_{r}^H(z_P+(x-y)) - \wt K_{r}^H(z_P-(x-y))\bigr| \\
&= \bigl| ((z_{P,H}+(x_H-y_H))\,\psi_r\bigl(|z_P+(x-y)|^2\bigr) - (z_{P,H}-(x_H-y_H))\,\psi_r\bigl(|z_P-(x-y)|^2\bigr)\bigr|\\
& \leq 2\,\bigl|x_H-y_H\bigr|\,\psi_r\bigl(|z_P+(x-y)|^2\bigr) \\
&\quad
+ \bigl|z_{P,H}-(x_H-y_H)\bigr|\,\,\bigl|\psi_r\bigl(|z_P-(x-y)|^2\bigr)
-\psi_r\bigl(|z_P+(x-y)|^2\bigr)\bigr|\\
& =: T_1+ T_2.
\end{align*}

To deal with $T_1$ we write
$$T_1\leq \frac{2\,\bigl|x_H-y_H\bigr|}{|z_P+ (x-y)|^{n+1}}.$$
Note then that $\bigl|x_H-y_H\bigr|\leq\ell(Q_0)$, while $|x-y|\approx\dist(x,H)$. Further, it
is easy to check that 
\begin{equation}\label{equa98}
|z_P+ (x-y)|\approx |z_P- (x-y)|\approx|z_P|+\dist(x,H),
\end{equation}
which implies that
$$T_1\lesssim \frac{\ell(Q_0)}{
\bigl(\dist(x,H) + |z_P|\bigr)^{n+1}}.$$

Now we will estimate $T_2$. To this end we intend to apply the mean value theorem. It is easy to
check that for all $r,t>0$,
$$|\psi_r'(t)|\lesssim \frac1{t^{\frac{n+3}2}},$$
and then, by \rf{equa98},
$$\bigl|\psi_r\bigl(|z_P-(x-y)|^2\bigr)
-\psi_r\bigl(|z_P+(x-y)|^2\bigr)\bigr|\lesssim \frac{\bigl||z_P-(x-y)|^2 - |z_P+(x-y)|^2\bigr|}
{\bigl(\dist(x,H)+|z_P|\bigr)^{n+3}}.$$
Now we have
\begin{align*}
\bigl||z_P-(x-y)|^2 - |z_P+(x-y)|^2\bigr| & = 
\bigl|\bigl[(z_{P,H}-(x_H-y_H))^2 + (x_{n+1}-y_{n+1})^2\bigr] \\
& \quad- \bigl[(z_{P,H}+(x_H-y_H))^2 
+ (x_{n+1}-y_{n+1})^2\bigr]\bigr|\\ 
& = 4\,
\bigl|z_{P,H}\cdot(x_H-y_H)\bigr| \leq 4\,|z_P|\,\ell(Q_0).
\end{align*}
Thus we infer that
$$T_2\lesssim \frac{\bigl|z_{P,H}-(x_H-y_H)\bigr|\,|z_P|\,\ell(Q_0)}
{\bigl(\dist(x,H)+|z_P|\bigr)^{n+3}}\lesssim \frac{\ell(Q_0)}
{\bigl(\dist(x,H)+|z_P|\bigr)^{n+1}}.
$$
Together with the estimate above for $T_1$ this yields \rf{claim64'}.

From \rf{eqsj54} and \rf{claim64'} we obtain
$$\bigl|
\wt\RR_r\nu(x)\bigr| \lesssim |\nu|(Q_0)
\sum_{P\in\M} \frac{\ell(Q_0)}
{\bigl(\dist(x,H)+|z_P|\bigr)^{n+1}}.
$$
It is easy to check that
$$\sum_{P\in\M} \frac{\ell(Q_0)^{n+1}}
{\bigl(\dist(x,H)+|z_P|\bigr)^{n+1}}\lesssim \frac{\ell(Q_0)}{\dist(x,H)},$$
and then \rf{eqb12} follows.
\end{proof}

\vv
\begin{lemma}\label{lem8.5}
We have
$$\int_{Q_0}|\RR\eta|^2\,d\eta\lesssim \left(\ve'+ \frac1{A^{2}}\right)\,\eta(Q_0).$$
\end{lemma}

\begin{proof}
By Lemma \ref{lem200} it is enough to show that
$$\int_{Q_0}|\RR(\chi_{(AQ_0)^c}\eta)|^2\,d\eta\lesssim \frac1{A^2}\,\eta(Q_0),$$
which is an immediate consequence
of Lemma \ref{lem83} (a).
\end{proof} 

\vv

\begin{rem}\label{remparameters}
By taking $A$ big enough and $\delta,\ve$ small enough in the assumptions of the Main Lemma~\ref{mainlemma}, and then choosing appropriately the parameters $\ve_0,\kappa_0,\theta_0$, it follows that
\begin{equation}\label{eqpetita1}
\int_{Q_0}|\RR\eta|^2\,d\eta\ll \eta(Q_0).
\end{equation}
Indeed, the preceding lemma asserts that 
$$\int_{Q_0}|\RR\eta|^2\,d\eta\lesssim \left(\ve'+ \frac1{A^{2}}\right)\,\eta(Q_0),$$
with $\ve'$ given in Lemma \ref{lem200} by
$$\ve'= \wt\ve + A^n\,\kappa_0^{-2n-2}\,\theta_0^{\frac1{(n+1)^2}},$$
where $\wt\ve$ is defined in \rf{epsilontilde} by
$$\wt\ve =C_4 \left(\ve + \frac1{A^2} +  A^{4n+2} \,\delta^{\frac1{4n+4}}+\ve_0+\theta_0^{\frac1{n+1}} + \kappa_0^{\frac12} + A^{2n+2}\,\wt\delta^{\frac2{4n+5}}
\right),$$
and $\wt\delta$ in \rf{eqdeltatilde} by
$$\wt\delta = C_3\,A^{n+1} \,
\biggl(\ve_0 + \theta_0^{1/(n+1)} + \kappa_0^{1/2} + \delta^{1/2}\biggr).
$$
Hence if we take first $A$ big enough and then $\ve_0,\kappa_0,\delta,\theta_0$ small enough (depending on $A$), so that moreover
$\theta_0\ll \kappa_0$ (to ensure that $A^n\kappa_0^{-2n-2}\,\theta_0^{\frac1{(n+1)^2}}\ll1$), then \rf{eqpetita1} follows.
\end{rem}

\vvv
% ***************************************************************************

\section{Proof of the Key Lemma by contradiction}\label{sec9}

Recall that we are trying to prove the Key Lemma \ref{keylemma} by contradiction, and so we are assuming that
\begin{equation}\label{eqcvo992}
\mu\Biggl(\bigcup_{Q\in \LD} Q\Biggr)> (1-\ve_0)\,\mu(Q_0).
\end{equation}
 This assumption has allowed us 
to prove that $\|\chi_{Q_0}\RR\eta\|_{L^2(\eta)}^2\ll \eta(Q_0)$ in Lemma \ref{lem8.5}. In this section,  by means of a variational argument, we will show that $\|\chi_{Q_0}\RR\eta\|_{L^2(\eta)}^2\gtrsim \eta(Q_0)$. This will give
%$\|\chi_{Q_0}\RR\eta\|_{L^2(\eta)}$ from below by means of a variational argument and 
us the desired contradiction.

\vv

\subsection{A variational argument and an almost everywhere inequality}

\begin{lemma}
\label{Le:DesigualdadAERiNu}
Suppose that, for some $0<\lambda\leq1$, the inequality
\begin{equation*}
\int_{Q_0} |\RR\eta|^2 d\eta \leq \lambda \,\eta(Q_0)
\end{equation*}
holds. Then, there is a %$\M$-periodic 
function $b\in L^{\infty}(\eta)$ such that
\begin{enumerate}[(i)]
\item\label{le:3-1} $0\leq b\leq 2$,
\item\label{le:3-2} $b$ is $\M$-periodic,
\item\label{le:3-3} $\displaystyle{\int_{Q_0} b\,d\eta = \eta(Q_0})$,
\end{enumerate}
and such that the measure $\nu = b\eta$ satisfies
\begin{equation}\label{eqa331}
\int_{Q_0} |\RR\nu|^2 d\nu \leq \lambda \,\nu(Q_0)
\end{equation}
and
\begin{equation}
\label{eq:DesCTPNu}
|\RR\nu(x)|^2 + 2\RR^*((\RR\nu)\nu)(x) \leq 6\lambda \quad\mbox{ for $\nu$-a.e. $x\in\R^{n+1}$.}
\end{equation}

\end{lemma}

\begin{proof}
In order to find such a function $b$, we consider the following class of admissible functions
\begin{equation}
\mathcal{A} = \Bigl\{ a\in L^{\infty}(\eta): \,a\geq0,\,\text{ $a$ is $\M$-periodic, and\, $\textstyle\int_{Q_0}a\,d\eta=\eta(Q_0)$} \Bigr\}
\end{equation}
and we define a functional $J$ on $\mathcal{A}$ by
\begin{equation}
J(a) = \lambda \|a\|_{L^\infty(\eta)}\,\eta(Q_0) + \int_{Q_0} |\RR(a\eta)|^2 a\, d\eta.
\end{equation}

Observe that $1\in \mathcal{A}$ and 
$$
 J(1) = \lambda \,\eta(Q_0) + \int_{Q_0} |\RR\eta|^2\, d\eta \leq 2\lambda \,\eta(Q_0),
$$
Thus 
$$\inf_{a\in\AZ} J(a)\leq2\lambda \,\eta(Q_0).$$
Since $J(a)\geq \lambda \|a\|_{L^\infty(\eta)}\,\eta(Q_0)$,
it is clear that
$$\inf_{a\in\AZ} J(a) = \inf_{a\in\AZ:\|a\|_{L^\infty(\eta)}\leq 2} J(a).$$
We claim that $J$ attains a global minimum on $\mathcal{A}$, i.e., there is a function $b\in \mathcal{A}$ such that $J(b)\leq J(a)$ for all $a\in \mathcal{A}$. 
Indeed, by the Banach-Alaoglu theorem there exists a sequence $\{a_k\}_k\subset \AZ$, with 
$J(a_k)\to \inf_{a\in\AZ} J(a)$, $\|a_k\|_{L^\infty(\eta)}\leq 2$, 
so that $a_k$ converges weakly * in $L^\infty(\eta)$ to some function
$b\in \AZ$. 
It is clear that $b$ satisfies \rf{le:3-1}, \rf{le:3-2} and \rf{le:3-3}.
Also, since $y\mapsto \frac{x-y}{|x-y|^{n+1}}$ belongs to $L^1_{loc}(\eta)$
(recall that $\eta$ has bounded density with respect to Lebesgue measure), it follows that
for all $x\in Q_0$ and all $r>0$, $\wt\RR_r(a_k\eta)(x)\to\wt\RR_r(b\eta)(x)$ as
$k\to\infty$. To see that for all $x\in Q_0$,
$\RR(a_k\eta)(x)\to\RR(b\eta)(x)$,  we write for any $k,r>0$ big enough
\begin{align*}
\bigl|\RR(a_k\eta)(x) - \RR(b\eta)(x)\bigr| & \leq 
\bigl|\RR(a_k\eta)(x) - \wt\RR_r(a_k\eta)(x)\bigr| +
\bigl|\wt\RR_r(a_k\eta)(x) - \wt\RR_r(b\eta)(x)\bigr| \\
&\quad +
\bigl|\wt\RR_r(b\eta)(x) - \RR(b\eta)(x)\bigr| \\
& \leq C\,\frac{\|a_k\|_{L^\infty(\eta)} + \|b\|_{L^\infty(\eta)}}r +
\bigl|\wt\RR_r(a_k\eta)(x) - \wt\RR_r(b\eta)(x)\bigr|,
\end{align*}
appliying \rf{eqa901} for the last inequality. Taking limits in $k$ in both sides we obtain
$$\limsup_{k\to\infty} \bigl|\RR(a_k\eta)(x) - \RR(b\eta)(x)\bigr|
\leq \frac{C}r + \limsup_{k\to\infty}
\bigl|\wt\RR_r(a_k\eta)(x) - \wt\RR_r(b\eta)(x)\bigr| = \frac{C}r.$$
Since this holds for all $r>0$ big enough, we infer that $\RR(a_k\eta)(x)\to\RR(b\eta)(x)$
as $k\to\infty$,
as wished. Taking into account that
$$|\RR(a_k\eta)(x)|\leq \frac Cr + |\wt\RR_r(a_k\eta)(x)| \leq \frac Cr + 
2\int_{|x-y|\leq r}\frac1{|x-y|^n}\,d\eta(y) \leq C(r)$$
 for all $r>0$ big enough,
by the dominated convergence theorem we infer that
$$\int_{Q_0} |\RR(a_k\eta)|^2 d\eta \to \int_{Q_0} |\RR(b\eta)|^2 d\eta\quad\mbox{ as\;
$k\to\infty$.}$$  
Using also that $\|b\|_{L^\infty(\eta)}\leq \limsup_k \|a_k\|_{L^\infty(\eta)}$, it follows 
that $J(b)\leq \limsup_k J(a_k)$, which proves the claim that $J(\cdot)$ attains a minimum at $b$.

\vv

The estimate \rf{eqa331} for $\nu=b\,\eta$ follows from the fact that $J(b)\leq J(1)$, because 
the property (iii) implies that $\|b\|_{L^\infty(\eta)}\geq 1$.

\vv
To prove \rf{eq:DesCTPNu} we perform a blow-up argument taking advantage of the fact that $b$ is a minimizer for $J$. Let $B$ be any ball contained in $Q_0$ and centered on $\supp\nu\cap Q_0$. Let
\begin{equation}
P_{\M}(B) = \bigcup_{R\in\M} (B+z_R)
\end{equation}
be the ``periodic extension'' of $B$ with respect to $\M$. Now, for every $0\leq t<1$, define
\begin{equation}
b_t = (1-t\chi_{P_{\M}(B)})b + t\,\frac{\nu(B)}{\nu(Q_0)}\,b.
\end{equation}
It is clear that $b_t\in\mathcal{A}$ for all $0\leq t<1$ and $b_0=b$. Therefore,
\begin{equation}
\begin{aligned}
J(b) & \leq J(b_t) = \lambda \|b_t\|_{\infty} \eta(Q_0) + \int_{Q_0} |\RR(b_t\eta)|^2 b_t\,d\eta 
\\
	&\leq \lambda \left(1+t\frac{\nu(B)}{\nu(Q_0)}\right)\|b\|_{\infty}\eta(Q_0) + \int_{Q_0} |\RR(b_t\eta)|^2 b_t\, d\eta := h(t).
\end{aligned}
\end{equation}
\par\medskip
Since $h(0)=J(b)$, we have that $h(0)\leq h(t)$ for $0\leq t<1$ and, thus $h'(0+)\geq 0$
(assuming that $h'(0+)$ exists). Notice that
\begin{equation*}
\frac{db_t}{dt}\Big|_{t=0}  = -\chi_{P_{\M}(B)}b + \frac{\nu(B)}{\nu(Q_0)}\,b,
\end{equation*} 
Therefore,
\begin{align*}
0 \leq h'(0+) &= \lambda\frac{\nu(B)}{\nu(Q_0)}\,\|b\|_{\infty}\eta(Q_0) + \frac{d}{dt}\Big|_{t=0}\int_{Q_0} |\RR(b_t\eta)|^2 b_t d\eta \\
	&= \lambda\frac{\nu(B)}{\nu(Q_0)}\|b\|_{\infty}\eta(Q_0) + 2\int_{Q_0}  \RR \left(\frac{db_t}{dt}\Big|_{t=0}\eta\right)\cdot \RR\nu \, b\,d\eta + \int_{Q_0} |\RR\nu|^2\, \frac{db_t}{dt}\Big|_{t=0}d\eta\\
	&=  \lambda\frac{\nu(B)}{\nu(Q_0)}\,\|b\|_{\infty}\eta(Q_0) + 2\int_{Q_0} \RR\left(\left(-\chi_{P_{\M}(B)}b + \frac{\nu(B)}{\nu(Q_0)}b\right)\eta\right)\cdot \RR\nu \, b\,d\eta \\
	& \qquad+ \int_{Q_0} |\RR\nu|^2 \left(-\chi_{P_{\M}(B)}b + \frac{\nu(B)}{\nu(Q_0)}b\right)d\eta\\
	&=  \lambda\frac{\nu(B)}{\nu(Q_0)}\|b\|_{\infty}\eta(Q_0) - 2\int_{Q_0}  \RR(\chi_{P_{\M}(B)}\nu)\cdot\RR\nu\,d\nu + 2\frac{\nu(B)}{\nu(Q_0)}\int_{Q_0} |\RR\nu|^2\,d\nu\\
	& \qquad -\int_{B}|\RR\nu|^2 \,d\nu + \frac{\nu(B)}{\nu(Q_0)}\int_{Q_0} |\RR\nu|^2\,d\nu,
\end{align*}
where we used that $P_{\M}(B)\cap Q_0=B$ in the last identity.
The fact that the derivatives above commute with the integral sign and with the
operator $\RR$ is guaranteed by the fact that $b_t$
is an affine function of $t$ and then one can expand the integrand
$|\RR(b_t\eta)|^2 b_t$ and obtain a polynomial expression on $t$.
Using also that
$\lambda\leq1$ and that $J(b)\leq2\lambda\,\nu(Q_0)$, we get
\begin{align}\label{eqfi83}
\int_B |\RR\nu|^2 \,d\nu + 2\int_{Q_0}  \RR(\chi_{P_{\M}(B)}\nu)\cdot\RR\nu\,d\nu & \leq \frac{\nu(B)}{\nu(Q_0)}\left[\lambda \|b\|_{\infty}\eta(Q_0) + 3\int_{Q_0} |\RR\nu|^2\,d\nu \right]\\
&\leq 3\,J(b)\,\nu(B)\leq 6\lambda\,\nu(B).\nonumber
\end{align}
We claim now that
\begin{equation}\label{eqafi84}
\int_{Q_0}  \RR(\chi_{P_{\M}(B)}\nu)\cdot\RR\nu\,d\nu = \int_B \RR^*((\RR\nu)\nu)\,d\nu.
\end{equation}
Assuming this for the moment, from  \rf{eqfi83} and \rf{eqafi84}, dividing by $\nu(B)$, we obtain 
\begin{equation*}
\frac{1}{\nu(B)}\int_{B} |\RR\nu|^2d\nu + \frac{2}{\nu(B)}\int_B \RR^*((\RR\nu)\nu)\,d\nu \leq 6\lambda,
\end{equation*}
and so, letting $\nu(B)\rightarrow 0$ and applying Lebesgue's differentiation theorem, we obtain
\begin{equation*}
|\RR\nu(x)|^2 + 2\RR^*((\RR\nu)\nu)(x) \leq  6\,\lambda\quad \mbox{ for $\nu$-a.e. $x\in\R^{n+1}$,}
\end{equation*}
 as desired.

It remains to prove the claim \rf{eqafi84}. By the uniform convergence of $\wt\RR_r(\chi_{P_{\M}(B)}\nu)$
and $\wt\RR_r\nu$ to $\RR(\chi_{P_{\M}(B)}\nu)$ and $\RR\nu$, respectively, we have
\begin{equation}\label{eqlim498}
\int_{Q_0}  \RR(\chi_{P_{\M}(B)}\nu)\cdot\RR\nu\,d\nu = \lim_{r\to\infty}
\int_{Q_0}  \wt\RR_r(\chi_{P_{\M}(B)}\nu)\cdot\wt\RR_r\nu\,d\nu.
\end{equation}
Since $\wt K_r(x-\cdot)$ has compact support, for all $x\in Q_0$,
\begin{align*}
\wt\RR_r(\chi_{P_{\M}(B)}\nu)(x) = \int_{P_{\M}(B)} \wt K_r(x-y)\,d\nu(y)
& = \sum_{P\in \M}  \int \wt K_r(x-y)\,d((T_P)_\#(\chi_B\nu)(y).
\end{align*}
For the last identity we have used the fact that the sum above runs only over a finite number of $P\in\M$ because there is only a finite number
of non-zero terms (in fact, we may assume these $P\in\M$ to be independent of 
$x\in Q_0$). Thus we have
$$\wt\RR_r(\chi_{P_{\M}(B)}\nu)(x) = \sum_{P\in \M}  \int_B \wt K_r(x-y-z_P)\,d\nu(y)
= \sum_{P\in\M} \wt\RR_r(\chi_B\nu)(x-z_P),$$
and so
\begin{align*}
\int_{Q_0}  \wt\RR_r(\chi_{P_{\M}(B)}\nu)(x)\cdot\wt\RR_r\nu\,d\nu(x) &=
\sum_{P\in \M} \int_{Q_0} \wt\RR_r(\chi_B\nu)(x-z_P) 
\cdot\wt\RR_r\nu(x)\,d\nu(x)\\
& = \sum_{P\in \M} \int_{Q_0-z_P} \wt\RR_r(\chi_B\nu)(x) 
\cdot\wt\RR_r\nu(x+z_P)\,d((T_P)^{-1}_{\,\#}\nu)(x)
\end{align*}
Since $\wt\RR_r\nu$ is $\M$-periodic, $\wt\RR_r\nu(x+z_P)=\wt\RR_r\nu(x)$ and 
$(T_P)^{-1}_{\,\#}\nu = \nu$, and then by Fubini, 
\begin{align}\label{eqalf42}
\int_{Q_0}  \wt\RR_r(\chi_{P_{\M}(B)}\nu)(x)\cdot\wt\RR_r\nu(x)\,d\nu(x)
& = \sum_{P\in \M} \int_{Q_0-z_P} \wt\RR_r(\chi_B\nu)(x) 
\cdot\wt\RR_r\nu(x)\,d\nu(x)\\
 & =\int \wt\RR_r(\chi_B\nu)(x) 
\cdot\wt\RR_r\nu(x)\,d\nu(x)\nonumber \\
& = \int_B \wt\RR_r^*((\wt\RR_r\nu)\nu)(y)\,d\nu(y).\nonumber
\end{align}
Since $\wt\RR_r\nu$ converges uniformly to $\RR\nu$ as $r\to\infty$ and $\wt\RR_r^*$ tends
to $\RR^*$ in operator norm in $L_\M^\infty(\eta)\to L_\M^\infty(\eta)$, we deduce that
$$\lim_{r\to\infty}\int_B \wt\RR_r^*((\wt\RR_r\nu)\nu)(y)\,d\nu(y) =
\int_B \RR^*((\RR\nu)\nu)\,d\nu.$$
Together with \rf{eqlim498} and \rf{eqalf42} this yields \rf{eqafi84}.
\end{proof}

\subsection{A maximum principle.}

\begin{lemma}
\label{Le:PrMax}
Assume that 
$
\ds\int_{Q_0} |\RR\eta|^2\,d\eta \leq \lambda \eta(Q_0)
$
 for some $0<\lambda\leq1$,
and let $b$ and $\nu$ be as in Lemma \ref{Le:DesigualdadAERiNu}. Let $K_S>10$ be some (big) constant and let $S$ be the horizontal strip
\begin{equation*}
S = \bigl\{x\in\R^{n+1}\colon |x_{n+1}|\leq K_S\ell(Q_0)\bigr\}.
\end{equation*}
Also, set
\begin{equation*}
f(x) = c_S\,x_{n+1}e_{n+1} = c_S(0,\dots,0,x_{n+1}), \quad\mbox{ with\; $c_S = \displaystyle\int
\dfrac{1}{\bigl(|y_H|^2+(K_S\ell(Q_0))^2\bigr)^{\frac{n+1}2}}\,d\nu(y).$}
\end{equation*}
Then, we have
\begin{equation}\label{eqd*1}
|\RR\nu(x)-f(x)|^2 + 4\RR^*((\RR\nu)\nu)(x) \lesssim \lambda^{1/2} + \frac1{K_S^2} \quad \mbox{ for all
$x\in S$.}
\end{equation}
Further, 
\begin{equation}\label{eqd*2}
c_S \lesssim \frac1{K_S\ell(Q_0)}.
\end{equation}
\end{lemma}

\begin{proof}
The inequality \rf{eqd*2} is very easy. Indeed, we just have to use that $\nu(B(x,r))\lesssim r^n$
for all $x\in\R^{n+1}$ and $r>0$, and use standard estimates which we leave for the reader.

To prove \rf{eqd*1}, we denote
\begin{equation*}
F(x) = |\RR\nu(x)-f(x)|^2 + 4\RR^*((\RR\nu)\nu)(x).
\end{equation*}
It is clear that $F$ is subharmonic in $\R^{n+1}\setminus\supp(\nu)$ and continuous in the whole space $\R^{n+1}$, by Lemma \ref{lemtyp}. So if we show that the estimate in \rf{eqd*1} holds
for all $x\in\supp\nu \cup \partial S$, then this will be also satisfied in the whole $S$. Indeed, since $F$ is $\M$-periodic and continuous in $S$, the maximum of $F$ in $S$ is attained, and since $F$ is subharmonic in $\interior{S}\setminus \supp\nu$, this must be attained in 
$\supp\nu \cup \partial S$.

\vv
First we check that the inequality in \rf{eqd*1} holds for all $x\in\supp\nu$. To this end, recall
that by Lemma \ref{Le:DesigualdadAERiNu}
\begin{equation*}
|\RR\nu(x)|^2 + 2\RR^*((\RR\nu)\nu)(x) \leq 6\lambda\quad\mbox{$\nu$-almost everywhere in $\supp(\nu)$,}
\end{equation*}
and this inequality extends to the whole $\supp(\nu)$ by continuity. Therefore we have, for all $x\in\supp(\nu)$,
\begin{align*}
F(x) &= |\RR\nu(x)-f(x)|^2 + 4\RR^*((\RR\nu)\nu)(x) \leq 2|\RR\nu(x)|^2 + 2|f(x)|^2 + 4\RR^*((\RR\nu)\nu)(x) \\
	&\leq 12\lambda + 2|f(x)|^2 \leq 12\lambda + \bigl(c_S\ell(Q_0)\bigr)^2\lesssim \lambda + \frac1{K_S^2},
\end{align*}
where we took into account that $|x_{n+1}|\leq\frac12\ell(Q_0)$ for $x\in\supp\nu$ and we used \rf{eqd*2}.
\vv

Our next objective consists in getting an upper bound for $F$ in $\partial S$. 
By applying Lemma \ref{lem83} to the function $\RR\nu$ (which is $\M$-periodic), with $\RR^*$ instead of $\RR$
(since $\RR$ is antisymmetric we are allowed to do this) we deduce that, for  all $x\in\partial S$,
$$\bigl|\RR^*((\RR\nu)\nu)(x)\bigr|\lesssim \frac1{\ell(Q_0)^n}\int_{Q_0}|\RR\nu|\,d\nu
\lesssim \frac1{\ell(Q_0)^n}\left(\int_{Q_0}|\RR\nu|^2\,d\nu\right)^{1/2}\,\nu(Q_0)^{1/2}
\lesssim\lambda^{1/2}.
$$

It suffices to show now that $|\RR\nu(x)-f(x)|\lesssim\frac1{K_S}$ for all $x\in\partial S$.
 We write $\RR\nu(x) = (\RR^H\nu(x),\RR_{n+1}\nu(x))$.
From  \rf{eqb12} we infer that
$$\bigl|\RR^H \nu(x)\bigr|\lesssim  \frac{1}{K_S\,\ell(Q_0)^{n}}\,\nu(Q_0)\lesssim \frac1{K_S}.$$

Hence it just remains to prove that
\begin{equation}\label{eqak423}
\bigl|\RR_{n+1} \nu(x)\,e_{n+1}-f(x)\bigr| 
\lesssim  \frac{1}{K_S} \quad \mbox{ for all $x\in\partial S$.}
\end{equation}
%Denote by $\partial_+ S$ the intersection of $\partial S$ with the upper hyperplane, 
 %(so that
%$x_{n+1}=K\ell(Q_0)$ for $x\in\partial_+ S$), and set $\partial_-S = \partial S\setminus \partial_+S$. 
To prove this estimate we can assume without loss of generality that $x_{n+1} ={K_S}\ell(Q_0)$
and that $x_H\in Q_0\cap H$, by the $\M$-periodicity of $\RR_{n+1} \nu$.
Since
$f(x) = c_S\,{K_S}\,\ell(Q_0)\,e_{n+1}$ for this point $x$, \rf{eqak423} is equivalent to
\begin{equation}\label{eqak424}
\bigl|\RR_{n+1} \nu(x)-c_S\,{K_S}\,\ell(Q_0)\bigr|
\lesssim  \frac{1}{K_S}.
\end{equation}
Note first that
$$\RR_{n+1} \nu(x) = \lim_{r\to0}\int\phi_r(x-y)\,\frac{x_{n+1} - y_{n+1}}{|x-y|^{n+1}}\,d\nu(y)
=\int\frac{x_{n+1} - y_{n+1}}{|x-y|^{n+1}}\,d\nu(y)
,$$
by an easy application of the dominated convergence theorem (using that 
$|x_{n+1}-y_{n+1}|\le \dist(x,H) + \ell(Q_0)$).  Consider the point $x_0=(0,{K_S}\ell(Q_0))$. Since for all
$y\in\supp\nu$,
$$|x-x_0|+|y-y_H|\leq\ell(Q_0)\leq \frac12 |x-y|,$$
and since the $(n+1)$-th component of $K(\cdot)$, which we denote by $K_{n+1}(\cdot)$, is a standard Calder\'on-Zygmund kernel,
$$\bigl|K_{n+1}(x-y)- K_{n+1}(x_0-y_H)\bigr| \lesssim \frac{|x-x_0|+|y-y_H|}{|x-y|^{n+1}}
\lesssim \frac{\ell(Q_0)}{|x-y|^{n+1}}.$$
Therefore, integrating with respect to $\nu$, we derive
$$\bigl|\RR_{n+1} \nu(x)- c_S\,{K_S}\,\ell(Q_0)\bigr| =\left|\int \bigl(K_{n+1}(x-y)- K_{n+1}(x_0-y_H)
\bigr)\,d\nu(y)\right|\lesssim \int\!\!\frac{\ell(Q_0)}{|x-y|^{n+1}}\,d\nu(y).$$
Since $\dist(x,\supp\nu)\gtrsim {K_S}\,\ell(Q_0)$ and $\nu$ is a measure with growth of order $n$, by standard estimates it follows that 
$$\int\frac{\ell(Q_0)}{|x-y|^{n+1}}\,d\nu(y)\lesssim \frac1{K_S},$$
which proves \rf{eqak424} and finishes the proof of the lemma.
\end{proof}

\vv

The next result is an immediate consequence of Lemma \ref{Le:PrMax}.

\begin{lemma}
\label{Le:PrMaxxx}
Assume that, for some $0<\lambda\leq1$, the inequality
\begin{equation*}
\int_{Q_0} |\RR\eta|^2\,d\eta \leq \lambda \eta(Q_0)
\end{equation*}
is satisfied, and let $b$ and $\nu$ be as in Lemma \ref{Le:DesigualdadAERiNu}. 
Then, we have
\begin{equation}\label{eqd*11}
|\RR\nu(x)|^2 + 4\RR^*((\RR\nu)\nu)(x) \lesssim \lambda^{1/2} \quad \mbox{ for all
$x\in \R^{n+1}$.}
\end{equation}
\end{lemma}

\begin{proof}
This follows by letting ${K_S}\to\infty$ in the inequality \rf{eqd*1}, taking into account that
$c_S\to0$, by \rf{eqd*2}.
\end{proof}
\vv

\subsection{The contradiction}

The estimate in the next lemma is the one that will allow us to contradict the assumption \rf{eqcvo992} and to complete the proof of the Key Lemma \ref{keylemma}.

\begin{lemma} \label{lemcontra22}
There is some constant $c_3>0$ depending only\footnote{In fact, keeping track of the dependencies, one can check that $c_3$ depends only on $n$ and $C_0$, and not on $C_1$. However, this is not necessary for the proof of the Key Lemma.} on $n,C_0,C_1$ such that
\begin{equation}\label{eqcontra220}
\int_{Q_0} |\RR\eta|^2d\eta \geq c_3 \,\eta(Q_0)
\end{equation}
\end{lemma}

\begin{proof}
Suppose that, for some $0<\lambda\leq1$, we have
\begin{equation}\label{eqcontra22}
\int_{Q_0} |\RR\eta|^2d\eta \leq \lambda \,\eta(Q_0)
\end{equation}
 We intend to show that then 
$\lambda\geq c_3$ for some constant $c_3$ depending only on $n$ and $C_0$. 

Let $b$ and $\nu$ be as in Lemma \ref{Le:DesigualdadAERiNu}.
By Lemma \ref{Le:PrMaxxx}, we have
\begin{equation}
\label{eq:129}
|\RR\nu(x)|^2 + 4\RR^*((\RR\nu)\nu)(x)\lesssim \lambda^{1/2}\quad \mbox{ everywhere in $\R^{n+1}$.}
\end{equation}
 Now pick a smooth function $\varphi$ with $\chi_{Q_0}\leq \varphi \leq \chi_{2Q_0}$ and $\|\nabla\varphi\|_{\infty}\lesssim \frac{1}{\ell(Q_0)}$. Set $\psi = C_5 \nabla \varphi$, so that $\RR^*(\psi \HH^{n+1})=\varphi$. Then, we have
\begin{align*}
\eta(Q_0)=\nu(Q_0) &\leq \int \varphi\, d\nu = \int \RR^*(\psi \HH^{n+1})\,d\nu \\
		    & = \int \RR\nu \,\psi\, d\HH^{n+1} \leq \left( \int |\RR\nu|^2 |\psi|\, d\HH^{n+1} \right)^{1/2} \left(\int |\psi|\,d\HH^{n+1}\right)^{1/2}.
\end{align*}
\par\medskip
First of all, observe that 
\begin{equation*}
\|\psi\|_\infty \lesssim\frac1{\ell(Q_0)}\quad \mbox{ and }\quad\int |\psi|\,d\HH^{n+1} \lesssim \ell(Q_0)^n
\end{equation*}
and so
\begin{equation}
\label{eq:132}
\eta(Q_0) \lesssim \left( \int |\RR\nu|^2 |\psi| d\HH^{n+1} \right)^{1/2} \ell(Q_0)^{n/2}.
\end{equation}
Furthermore,  by \rf{eq:129} we have
\begin{align}\label{eqkkll}
\int |\RR\nu|^2 &|\psi| \,d\HH^{n+1} \le C\,\lambda^{1/2}\int |\psi| d\HH^{n+1} + 4\left| \int \RR^*((\RR\nu)\nu)|\psi|\,d\HH^{n+1}\right|\\
& \lesssim \lambda^{1/2}\ell(Q_0)^n + \left| \int \RR^*(\chi_{(5Q_0)^c}(\RR\nu)\nu)|\psi|\,d\HH^{n+1}\right|
+\left| \int \RR^*\bigl(\chi_{5Q_0}(\RR\nu)\nu\bigr)|\psi|\,d\HH^{n+1}\right|.\nonumber
\end{align}
To estimate the first integral on the right hand side we apply Lemma \ref{lem83} (a) with
$\wt A=5$ and $f=\RR\nu\, b$ (where $b$ is such that $b\eta=\nu$), and then we deduce that
for all $x\in 2Q_0$,
\begin{align*}
\bigl|\RR^*\bigl(\chi_{(5Q_0)^c}(\RR\nu)\nu\bigr)(x) &\lesssim \frac1{\ell(Q_0)^n}\,\int_{Q_0} |\RR\nu\, b|\,d\eta \\
&= \frac1{\ell(Q_0)^n}\,\int_{Q_0} |\RR\nu|\,d\nu\lesssim \left(\avint_{Q_0} |\RR\nu|^2\,d\nu
\right)^{1/2}\lesssim \lambda^{1/2}.
\end{align*}
Thus, recalling that $\psi$ is supported in $2Q_0$,
$$\left| \int \RR^*(\chi_{(5Q_0)^c}(\RR\nu)\nu)|\psi|\,d\HH^{n+1}\right|
\lesssim \lambda^{1/2}\,\|\psi\|_1\lesssim \lambda^{1/2}\,\nu(Q_0).$$
Concerning the last integral on the right hand side of \rf{eqkkll}, we have
\begin{align*}
\left| \int \RR^*\bigl(\chi_{5Q_0}(\RR\nu)\nu\bigr)|\psi|\,d\HH^{n+1}\right|
& = \left| \int_{5Q_0}\RR\nu\cdot\RR(|\psi|\,d\HH^{n+1})\,d\nu\right| \\
&\leq \left( \int_{5Q_0}|\RR\nu|^2\,d\nu\right)^{1/2} 
\left( \int_{5Q_0}|\RR(|\psi|\,d\HH^{n+1})|^2\,d\nu\right)^{1/2}. 
\end{align*}
The first integral on the right hand side does not exceed $c\lambda\,\nu(Q_0)$
(by \rf{eqa331} and the periodicity of $\RR\nu$). For the second one,
using that $|\psi|\lesssim \frac1{\ell(Q_0)}\chi_{2Q_0}$, it follows easily that
$\|\RR(|\psi|\,\HH^{n+1})\|_\infty\lesssim 1$. So we get
$$\left| \int \RR^*\bigl(\chi_{5Q_0}(\RR\nu)\nu\bigr)|\psi|\,d\HH^{n+1}\right|\lesssim
\lambda^{1/2}\nu(Q_0).$$
So from \rf{eqkkll} and the last estimates we deduce that
$$\int |\RR\nu|^2 |\psi| d\HH^{n+1}\lesssim \lambda^{1/2}\nu(Q_0).$$
Thus, by \rf{eq:132},
$$\nu(Q_0)\lesssim \lambda^{1/4}\nu(Q_0).$$
That is, $\lambda\gtrsim 1.$
\end{proof}

\vvv
Now recall that we are trying to prove the Key Lemma \ref{keylemma} by contradiction. We assumed 
 that
\begin{equation}\label{eqassu990}
\mu\Biggl(\bigcup_{Q\in \LD} Q\Biggr)> (1-\ve_0)\,\mu(Q_0).
\end{equation}
Using this hypothesis, we showed in Lemma \ref{lem8.5} and Remark \ref{remparameters} that
$\int_{Q_0}|\RR\eta|^2\,d\eta\ll \eta(Q_0)$ if $A$ is big enough and are  $\delta,\ve,\kappa_0,\theta_0$ small enough and chosen suitably, under the assumption that $\ve_0$ is small enough too.
This contradicts the conclusion of Lemma \ref{lemcontra22}. Hence, \rf{eqassu990} cannot hold and thus we are done.

\vvv
% ***************************************************************************

\section{Construction of the AD-regular measure $\zeta$ and the uniformly rectifiable set $\Gamma$
in the Main Lemma} \label{sec10}

Denote
\begin{equation}\label{eqff09}
F = Q_0\cap\supp\mu\setminus \bigcup_{Q\in \LD} Q.
\end{equation}
Recall that the Key Lemma \ref{keylemma} tells us that
$$\mu(F)\geq \ve_0\,\mu(Q_0).$$
It is easy to check that $0<\Theta_*^n(x,\mu)\leq \Theta^{n,*}(x,\mu)< \infty$ for $\mu$-a.e. $x\in F$. Since $\RR_\mu$ is bounded on $L^2(\mu|_{F})$ it follows that $\mu|_F$ is $n$-rectifiable, by the
Nazarov-Tolsa-Volberg theorem \cite{NToV-pubmat}. However, to get a big piece of a set contained in a uniformly $n$-rectifiable
set $\Gamma$ like the one required in the Main Lemma and in Theorem \ref{teo0} we have to argue more carefully.
 To this end, first we will construct an auxiliary AD-regular measure $\zeta$ such that $\zeta(F)\gtrsim\mu(F)$,
and then we will apply the Nazarov-Tolsa-Volberg theorem \cite{NToV} for AD-regular measures.
\vv

Next we are going to construct the aforementioned auxiliary measure $\zeta$. The arguments 
for this construction can be considered as a quantitative version of the ones from \cite{NToV-pubmat}, which
rely on a covering theorem of Pajot (see \cite{Pajot}).

Recall the notation $\sigma = \mu|_{Q_0}$. Consider the maximal dyadic operator
$$\M_{\DD_\sigma} f(x) = \sup_{Q\in\DD_{\sigma}:x\in Q} \frac1{\sigma(Q)} \int_Q |f|\,d\sigma,$$
where $\DD_{\sigma}$ is the David-Mattila lattice associated $\sigma$.
Let $F$ be as in \rf{eqff09} and set
$$\wt F= \Bigl\{x\in F:\M_{\DD_\sigma} (\chi_{F^c})(x) \leq 1-\frac{\ve_0}2\Bigr\}.$$
We wish to show that 
\begin{equation}\label{eqsss39}
\sigma(\wt F)\geq \frac12\,\sigma(F).
\end{equation}
 To this end, note that
$$F\setminus \wt F= \Bigl\{x\in F:\M_{\DD_\sigma} (\chi_{F^c})(x) > 1-\frac{\ve_0}2\Bigr\}$$
and consider a collection of maximal (and thus disjoint) cells $\{Q_i\}_{i\in J}\subset\DD_\sigma$ such that $\sigma(Q_i\setminus F)
>(1-\frac{\ve_0}2)\sigma(Q_i)$. Observe that
$$F\setminus\wt F = \bigcup_{i\in J} Q_i\cap F.$$
 Clearly, the cells $Q_i$ satisfy $\sigma(Q_i\cap F)\leq\frac{\ve_0}2\sigma(Q_i)$ and so we have
$$\sigma(F\setminus \wt F) \leq \sum_{i\in J} \sigma(Q_i\cap F) \leq \sum_{i\in J}\frac{\ve_0}2 \,\sigma(Q_i)\leq \frac{\ve_0}2\,\sigma(Q_0)
\leq \frac12\,\sigma(F),$$
which proves \rf{eqsss39}.

For each $i\in J$ we consider the family $\AZ_i$ of maximal doubling cells from $\DD_\sigma^{db}$ which cover $Q_i$, and we define
$$\AZ=\bigcup_{i\in J}\AZ_i.$$
Finally, we denote by $\AZ_0$ the subfamily of the cells $P\in\AZ$ such that $\sigma(P\cap F)>0$.
Now, for each $Q\in \AZ_0$ we consider an $n$-dimensional sphere $S(Q)$ concentric with $B(Q)$ and
with radius $\frac14 r(B(Q))$. We define
$$\zeta = \sigma|_{\wt F} + \sum_{Q\in \AZ_0} \HH^n|_{S(Q)}.$$

\vv

\begin{rem}\label{remqi}
If $P\in\AZ_0$ and $P\subset Q_i$ for some $i\in J$, then 
$$\ell(P)\approx_{\theta_0,C_0}\ell(Q_i).$$
Indeed, since $P$ is a maximal doubling cell contained in $Q_i$, by Lemma
\ref{lemcad23} and the fact that $3.5B_P\subset 100 B(P)$,
$$\Theta_\sigma(3.5B_P)\lesssim
\Theta_\sigma(100B(P))\lesssim A_0^{-9n(J(P)-J(Q_i))}\,\Theta_\sigma(100B(Q_i))
\lesssim_{C_0} A_0^{-9n(J(P)-J(Q_i))}.$$
Since $\sigma(P\cap F)>0$, it turns out that $P$ is not contained
in any cell from $\LD$, and so $\Theta_\sigma(3.5B_P)>\theta_0$.
So we have 
$$\theta_0\lesssim_{C_0} A_0^{-9n(J(P)-J(Q_i))},$$
which implies that $|J(P)-J(Q_i)|\lesssim_{\theta_0,C_0} 1$.\vv

A very similar argument shows that if $P\in\DD_\sigma$ satisfies $P\cap F\neq\varnothing$ (and so it is not contained in any cell from $\LD$), then there exists some $Q\in\DD_\sigma^{db}$ which
contains $Q$ and such that
$$\ell(P)\approx_{\theta_0,C_0}\ell(Q).$$
The details are left for the reader.

From the two statements above, if follows that {\em for any cell $P\in\DD_\sigma$ which is not strictly contained in any cell from $\AZ_0$ there exists some cell $\wh P\in\DD_\sigma^{db}$ which is not contained in any cell $Q_i$, $i\in J$, so that $P\subset\wh P$ and $\ell(P)\approx_{\theta_0,C_0}\ell(Q).$}
\end{rem}

\vv

\begin{lemma}\label{lemAD}
The measure $\zeta$ is AD regular, with the AD-regularity constant depending on $C_0$, $\theta_0$,
and $\ve_0$.
\end{lemma}

\begin{proof}
First we will show the upper AD-regularity of $\zeta$. That is, we will prove that $\zeta(B(x,r))\leq C(C_0,\theta_0)\,r^n$ for all
$x,r$. By the upper AD-regularity of $\sigma$, it is enough to show that the measure
$$\nu = \sum_{Q\in \AZ_0} \HH^n|_{S(Q)}$$
is also upper AD-regular. So we have to prove that 
\begin{equation}\label{eqgrnu33}
\nu(B(x,r))\leq C(C_0,\theta_0)\,r^n\quad \mbox{ for all $x\in\bigcup_{Q\in\AZ_0}S(Q)$ and all $r>0$.}
\end{equation}
Take $x\in S(Q)$, for some $Q\in\AZ_0$. Clearly the estimate above holds if the only sphere $S(P)$, $P\in\AZ_0$, that intersects $B(x,r)$ is just $S(Q)$ itself.
So assume that $B(x,r)$ intersects a sphere $S(P)$, $P\in\AZ_0$, with $P\neq Q$.
Recall that $\frac12 B(Q)\cap \frac12B(P)=\varnothing$, by 
Remark \ref{remfac1}, and thus for some constant $C_6$,
$P\subset B(x,C_6r)$. Hence,
$$\nu(B(x,r))\leq \sum_{P\in\AZ_0:P\subset B(x,C_6r)} \nu(\tfrac14S(P))\lesssim \sum_{P\in\AZ_0:P\subset B(x,C_6r)} \ell(P)^n.$$
Note now that by the definition of $\AZ_0$, $\sigma(F\cap P)>0$, which implies that
$P\not\in\LD$ and that $P$ is not contained in any cell from $\LD$, and thus taking also into account that $P\in\DD^{db}$,
\begin{equation}\label{eqas568}
\sigma(P)\gtrsim \sigma(3.5B_P)\gtrsim \theta_0\,\ell(P)^n.
\end{equation}
Together with the upper AD-regularity of $\sigma$ this yields
$$
\nu(B(x,r))\lesssim \frac1{\theta_0}\sum_{P\in\AZ_0:P\subset B(x,C_6r)} \sigma(P)\lesssim 
\frac1{\theta_0}\,\sigma(B(x,C_6r)) \lesssim_{C_0,\theta_0}r^n,
$$
which concludes the proof of \rf{eqgrnu33}.

It remains now to show the lower AD-regularity of $\zeta$. 
First we will prove that 
\begin{equation}\label{eqdk742}
\zeta(2B_Q)\gtrsim_{\theta_0,\ve_0,C_0} \ell(Q)^n \quad\mbox{ if $Q\in\DD_\sigma^{db}$ is not contained in any cell $Q_i$, $i\in J$.}
\end{equation}
Indeed, note that by the definition of the cells $Q_i$, $i\in J$,
$$\sigma(Q\setminus F)\leq  \Bigl(1-\frac{\ve_0}2\Bigr)\,\sigma(Q),$$
or equivalently,
$$\sigma(Q\cap F)\geq \frac{\ve_0}2\,\sigma(Q).
$$
Since $Q$ is not contained in any cell from $\LD$ (by the definitions of $F$ and $\AZ_0$) and is doubling, 
\begin{equation}\label{eqdk743}
\sigma(Q\cap F)\gtrsim_{\ve_0} \sigma(3.5B_Q)\gtrsim_{\theta_0,\ve_0}\ell(Q)^n.
\end{equation}
On the other hand, by the construction of $\zeta$,
$$\sigma(Q\cap F) = \sigma(Q\cap\wt F) + \sum_{P\in\AZ_0:P\subset Q}\sigma(P\cap F) \lesssim_{C_0}
\zeta(Q\cap\wt F) + \sum_{P\in\AZ_0:P\subset Q}\HH^n(S(P)).$$
We may assume that all the cells $P\subset Q$ satisfy $S(P)\subset 2B_Q$, just by choosing the constant $A_0$ in the construction of the lattice $\DD_\sigma$ big enough. Then we get
$$\sigma(Q\cap F)  \lesssim_{C_0}
\zeta(Q\cap\wt F) + \sum_{P\in\AZ_0:S(P)\subset 2B_Q}\zeta(S(P))  \lesssim_{C_0}
\zeta(2B_Q)
.$$
Together with \rf{eqdk743}, this gives \rf{eqdk742}.

To prove the lower AD regularity of $\zeta$, note that by Remark \ref{remqi} there is some constant $C'(C_0,\theta_0)$ such that 
if $x\in S(Q)$, $Q
\in\AZ_0$, and $ C'(C_0,\theta_0)\,\ell(Q)< r\leq \diam(Q_0)$, then there exists $P\in\DD_\sigma^{db}$ not contained in any cell $Q_i$, $i\in J$, such that $2B_P\subset B(x,r)$, with $\ell(P)\approx_{\theta_0,C_0} r$. The same holds for $0 < r\leq \diam(Q_0)$ if $x\in\wt F$. 
From \rf{eqdk742}  we deduce that
$$\zeta(B(x,r))\geq \zeta(2B_P) \gtrsim_{\theta_0,\ve_0,C_0} \ell(P)^n \approx_{\theta_0,\ve_0,C_0} r^n.$$
In the case that $ r\leq C'(C_0,\theta_0)\,\ell(Q)$ for  $x\in S(Q)$, $Q
\in\AZ_0$, the lower AD-regularity of $\HH^n|_{S(Q)}$ gives the required lower estimate
for $\zeta(B(x,r))$.
\end{proof}

\vv

\begin{lemma}\label{lemrieszeta}
The Riesz transform $\RR_\zeta$ is bounded in $L^2(\zeta)$, with a bound on the norm depending on
$C_0$, $C_1$, $\theta_0$, and $\ve_0$.
\end{lemma}

To prove this result we will follow very closely the arguments in the last part of the proof of the Main Lemma 2.1
of \cite{NToV-pubmat}. For completeness, we will show all the details.

\vv
For technical reasons, it will  be convenient to work with an $\ve$-regularized
version $\wh \RR_{\nu,\ve}$ of the  Riesz transform $\RR_{\nu}$. For a measure $\nu$ with growth of order $n$, we set
$$
\wh \RR_{\nu,\ve}f(x) = \int \frac{x-y}{\max(|x-y|,\ve)^{n+1}}\,f(y)\,d\nu(y).
$$
It is easy to check that
$$|\wh \RR_{\nu,\ve}f(x) -  \RR_{\nu,\ve}f(x)|\leq c\,M_n f(x)\qquad\mbox{for all $x\in\R^{n+1}$},$$
where $c$ is independent of $\ve$ and
$M_n$ is the maximal operator defined by
$$M_n f(x)=\sup_{r>0} \frac1{r^n}\int_{B(x,r)}|f|\,d\nu.$$
Since $M_n$ is bounded in $L^2(\nu)$ (because $\nu$ has growth of order $n$), it turns out that
$\RR_\nu$ is bounded in $L^2(\nu)$ if and only if the operators $\wh \RR_{\nu,\ve}$ are bounded in $L^2(\nu)$
uniformly on $\ve>0$. The advantage of $\wh \RR_{\nu,\ve}$ over $\RR_{\nu,\ve}$ is that the kernel
$$\wh K_\ve(x) = \frac{x}{\max(|x|,\ve)^{n+1}}$$
is continuous and
satisfies the smoothness condition
$$|\nabla \wh K_\ve(x)|\leq \frac{c}{|x|^{n+1}},\quad |x|\neq\ve$$
(with $c$ independent of $\ve$),
which implies that $\wh K_\ve(x-y)$
is a standard Calder\'on-Zygmund kernel (with constants independent of $\ve$), unlike the kernel of $\RR_{\nu,\ve}$.

\vv
\begin{proof}[Proof of Lemma \ref{lemrieszeta}]
To shorten notation, in the arguments below we will allow all the implicit constants in the relations $\lesssim$ and $\approx$ to depend on $C_0,C_1,\theta_0,\ve_0$.

Denote
$$\nu = \sum_{Q\in \AZ_0} \HH^n|_{S(Q)},$$
so that $\zeta = \sigma|_{\wt F} +\nu$.
Since $\RR_\sigma$ is bounded in $L^2(\sigma)$, it is enough to show that $\RR_\nu$ is bounded 
in $L^2(\nu)$. Indeed,  the boundedness of both operators implies the boundedness of $\RR_{\sigma+\nu}$ in $L^2(\sigma+\nu)$ (see Proposition 2.25 of \cite{Tolsa-llibre}, for example).

As in \rf{eqik00}, given $\kappa>0$, for each $Q\in\AZ_0$, we consider the set
$$I_{\kappa}(Q) = \{x\in Q:\dist(x,\supp\sigma\setminus Q)\geq \kappa\ell(Q)\}.$$
By the small boundary condition of $Q$, the fact that $Q$ is doubling, 
and that $\sigma(Q)\gtrsim \theta_0\,\ell(Q)^n$ (as shown in \rf{eqas568}), we deduce
there exists some $\kappa>0$ small enough such that
\begin{equation}\label{eqdfy782}
\sigma(I_\kappa(Q))\geq \frac12\,\sigma(Q)\gtrsim \theta_0\,\ell(Q)^n.
\end{equation}

We consider the measure
$$\wt\sigma = \sum_{Q\in \AZ_0} c_Q\,\sigma|_{I_\kappa(Q)},$$
with $c_Q=\HH^n(S(Q))/\sigma(I_\kappa(Q))$. By \rf{eqdfy782} it follows that the constants $c_Q$, $Q\in \AZ_0$, have a uniform bound depending on $\theta_0$, and thus $\RR_{\wt\sigma}$ is bounded in $L^2(\wt\sigma)$ (with a norm possibly depending on $\theta_0$). Further,
$\nu(S(Q)) = \wt\sigma(Q)$ for each $Q\in \AZ_0$. 

It is clear that, in a sense, $\wt\sigma$ can be considered as an approximation of $\nu$ (and
 conversely). To prove the boundedness of $\RR_\nu$ in $L^2(\nu)$, 
 we will prove that $\wh \RR_{\nu,\ve}$ is bounded in $L^2(\nu)$ uniformly on $\ve>0$ by comparing 
 it to $\wh \RR_{\wt\sigma,\ve}$.
First we need to introduce some local and non local operators: given $z\in \bigcup_{Q\in \AZ_0} S(Q)$, we denote by
$S(z)$ the sphere $S(Q), Q\in \AZ_0,$ that contains $z$. Then we write, for $z\in\bigcup_{Q\in \AZ_0} S(Q)$,
$$\RR^{loc}_{\nu,\ve} f(z) = \wh \RR_{\nu,\ve}(f\chi_{S(z)})(z),
 \qquad \RR^{nl}_{\nu,\ve} f(z) = \wh \RR_{\nu,\ve}(f\chi_{\R^{n+1} \setminus S(z)})(z).$$
We define analogously $\RR^{loc}_{\wt\sigma,\ve} f$ and $\RR^{nl}_{\wt\sigma,\ve} f$:
given $z\in \bigcup_{Q\in \AZ_0} Q$, we denote by
$Q(z)$ the cell $Q\in \AZ_0,$ that contains $z$. Then for $z\in\bigcup_{Q\in \AZ_0} Q$, we set
$$\RR^{loc}_{\wt\sigma,\ve} f(z) = \wh \RR_{\wt\sigma,\ve}(f\chi_{Q(z)})(z),
 \qquad \RR^{nl}_{\wt\sigma,\ve} f(z) = \wh \RR_{\wt\sigma,\ve}(f\chi_{\R^{n+1} \setminus Q(z)})(z).$$

It is straightforward to check that $\RR^{loc}_{\nu,\ve}$ is bounded in $L^2(\nu)$, and
that $\RR^{loc}_{\wt\sigma,\ve}$ is bounded in $L^2(\wt\sigma)$, both uniformly
on $\ve$ (in other words, 
$\RR^{loc}_{\nu}$ is bounded in $L^2(\nu)$ and
 $\RR^{loc}_{\wt\sigma}$ is bounded in $L^2(\wt\sigma)$). 
Indeed,
$$\|\RR^{loc}_{\nu,\ve}f\|_{L^2(\nu)}^2 = \sum_{Q\in \AZ_0}
\|\chi_{S(Q)}\wh \RR_{\nu,\ve}(f\chi_{S(Q)})\|_{L^2(\nu)}^2 \lesssim\sum_{Q\in\AZ_0}
\|f\chi_{S(Q)}\|_{L^2(\nu)}^2 = \|f\|_{L^2(\nu)}^2,$$
by the boundedness of the Riesz transforms on $S(Q)$. Using the boundedness
of $\RR_\sigma$ in $L^2(\sigma)$, it follows analogously that $\RR^{loc}_{\wt\sigma,\ve}$ is bounded in 
$L^2(\wt\sigma)$.

\vspace{2mm}
\noindent
{\bf Boundedness of $\RR^{nl}_{\nu}$ in $L^2(\nu)$. }
We must show that 
$\RR^{nl}_{\nu}$ is bounded in $L^2(\nu)$. To this end, we will compare $\RR^{nl}_{\nu}$ to $\RR^{nl}_{{\wt\sigma}}$.
 Observe first that, since 
$\RR^{nl}_{{\wt\sigma},\ve} = \wh \RR_{{\wt\sigma},\ve} - \RR^{loc}_{{\wt\sigma},\ve}$, and both
$\wh \RR_{{\wt\sigma},\ve}$ and $\RR^{loc}_{{\wt\sigma},\ve}$ are bounded in $L^2({\wt\sigma})$, it follows that 
$\RR^{nl}_{{\wt\sigma},\ve}$ is bounded in $L^2({\wt\sigma})$ (all uniformly on $\ve>0$).

Note also that for two different cells $P,Q\in\AZ_0$, we have
\begin{equation}\label{eqdkl93001}
\dist(S(P),S(Q))\approx \dist(I_\kappa(P),I_\kappa(Q))\approx \dist(S(P),I_\kappa(Q))\approx D(P,Q),
\end{equation}
where $D(P,Q)=\ell(P) + \ell(Q) + \dist(P,Q)$ and the implicit constants may depend on $\kappa$.  
The arguments to prove this are exactly the same as the ones for \rf{eq19050}, \rf{eq1905} and  \rf{eq1906}, and so we omit them.
In particular, \rf{eqdkl93001} implies that $\bigl(S(P)\cup I_\kappa(P)\bigr) \cap \bigl(S(Q)\cup I_\kappa(Q)\bigr) = \varnothing$,
and thus for every $z\in\R^{n+1}$ there is at most one cell $Q\in \AZ_0$ such that $z\in S(Q)\cup I_\kappa(Q)$, which we denote by $Q(z)$.
Hence we can extend 
 $\RR^{nl}_{\nu,\ve}$ and $\RR^{nl}_{\wt\sigma,\ve}$ to $L^2(\wt\sigma + \nu)$ by setting 
$$\RR^{nl}_{\nu,\ve} f(z) = \wh \RR_{\nu,\ve}(f\chi_{\R^{n+1} \setminus S(Q(z))})(z), \qquad
\RR^{nl}_{\wt\sigma,\ve} f(z) = \wh \RR_{\wt\sigma,\ve}(f\chi_{\R^{n+1} \setminus Q(z)})(z).$$

We will prove below that, for all $f\in L^2({\wt\sigma})$ and $g\in L^2(\nu)$ satisfying
\begin{equation}\label{eqcond33}
\int_{I_\kappa(P)} f\,d{\wt\sigma} = \int_{S(P)} g\,d\nu\quad \mbox{ for all $P\in\AZ_0$,}
\end{equation}
we have
\begin{equation}\label{eqkey1}
I(f,g) := \int |\RR^{nl}_{{\wt\sigma},\ve} f - \RR^{nl}_{\nu,\ve} g|^2\,d({\wt\sigma} + \nu) \lesssim
\|f\|_{L^2({\wt\sigma})}^2 + \|g\|_{L^2(\nu)}^2,
\end{equation} 
uniformly on $\ve$.
Let us see how the boundedness of $\RR^{nl}_{\nu}$ in $L^2(\nu)$ follows from this
 estimate.
As a preliminary step, we show that $\RR_\nu^{nl}:L^2(\nu)\to L^2({\wt\sigma})$ is bounded. To this
end, given $g\in L^2(\nu)$, we consider a function $f\in L^2({\wt\sigma})$ satisfying \rf{eqcond33}
that is constant on each ball $B_j$. It is straightforward to check that
$$
\|f\|_{L^2({\wt\sigma})}\leq \|g\|_{L^2(\nu)}.
$$
Then from the $L^2({\wt\sigma})$ boundedness of $\RR_{\wt\sigma}^{nl}$ and \rf{eqkey1}, we obtain
$$
\|\RR_{\nu,\ve}^{nl}g\|_{L^2({\wt\sigma})}\leq \|\RR_{{\wt\sigma},\ve}^{nl}f\|_{L^2({\wt\sigma})} + I(f,g)^{1/2}\lesssim
 \|f\|_{L^2({\wt\sigma})} + \|g\|_{L^2(\nu)}\lesssim\,\|g\|_{L^2(\nu)},
$$
which proves that $\RR_\nu^{nl}:L^2(\nu)\to L^2({\wt\sigma})$ is bounded. 

It is straightforward to check that the adjoint of $(\RR_{\nu,\ve}^{nl})_j:L^2(\nu)\to L^2({\wt\sigma})$ (where $(\RR_{\nu,\ve}^{nl})_j$ stands for the $j$-th component of $(\RR_{\nu,\ve}^{nl})_j$) equals $-(\RR_{\wt\sigma,\ve}^{nl})_j:L^2(\wt\sigma)\to L^2({\nu})$. So by duality we deduce that
$\RR_{\wt\sigma}^{nl}:L^2(\wt\sigma)\to L^2({\nu})$ is also bounded.

To prove now the $L^2(\nu)$ boundedness of $\RR_\nu^{nl}$, we consider an arbitrary function
$g\in L^2(\nu)$, and we construct $f\in L^2({\wt\sigma})$ satisfying \rf{eqcond33} which is constant
in each ball $P$.  Again, we have $\|f\|_{L^2({\wt\sigma})}\leq \|g\|_{L^2(\nu)}.$
Using the boundedness of $\RR_{\wt\sigma}^{nl}:L^2({\wt\sigma})\to L^2(\nu)$ together with \rf{eqkey1}, we obtain
$$\|\RR_{\nu,\ve}^{nl}g\|_{L^2(\nu)}\leq \|\RR_{{\wt\sigma},\ve}^{nl}f\|_{L^2(\nu)} + I(f,g)^{1/2}\lesssim 
 \|f\|_{L^2({\wt\sigma})} + \|g\|_{L^2(\nu)}\lesssim\|g\|_{L^2(\nu)},$$
as wished. 

It remains to prove that \rf{eqkey1} holds for $f\in L^2({\wt\sigma})$ and $g\in L^2(\nu)$ satisfying
\rf{eqcond33}.
For $z\in\bigcup_{P\in\AZ_0} P$, we have
$$|\RR^{nl}_{{\wt\sigma},\ve} f(z) - \RR^{nl}_{\nu,\ve} g(z)| \leq \sum_{P\in\AZ_0:P\neq Q(z)} \left|\int \wh K_\ve(z-y)
(f(y)\,d{\wt\sigma|_{I_\kappa(P)}}(y)-g(y)\,d\nu|_{S(P)}(y))\right|,$$
where $\wh K_\ve(z)$ is the kernel of the $\ve$-regularized  $n$-Riesz transform.
By standard estimates, using \rf{eqcond33} and \rf{eqdkl93001},
and the smoothness of $\wh K_\ve$, it follows that
\begin{align*}
\biggl|\int \wh K_\ve(z-y)&
(f(y)\,d{\wt\sigma|_{I_\kappa(P)}}(y)-g(y)\,d\nu|_{S(P)}(y))\biggr| \\ &= \left|\int_{P} (\wh K_\ve(z-y)- K_\ve(z-x))
(f(y)\,d{\wt\sigma|_{I_\kappa(P)}}(y)-g(y)\,d\nu|_{S(P)}(y))\right|\\
&\lesssim\int\frac{|x-y|}{|z-y|^{n+1}}(|f(y)|\,d{\wt\sigma|_{I_\kappa(P)}}(y)+|g(y)|\,d\nu|_{S(P)}(y))\\
&\approx \frac{\ell(P)}{D(Q(z),P)^{n+1}}\,\int(|f(y)|\,d{\wt\sigma|_{I_\kappa(P)}}(y)+|g(y)|\,d\nu|_{S(P)}(y)).
\end{align*}
Recall that $Q(z)$ stands for the cell $Q, Q\in\AZ_0,$ such that $z\in S(Q)\cup I_\kappa(Q)$.

We consider the operators
$$
T_{\wt\sigma}(f)(z) = \sum_{P\in\AZ_0: P\neq Q(z)}\frac{\ell(P)}{D(Q(z),P)^{n+1}}\,\int
f\,d{\wt\sigma|_{I_\kappa(P)}}\,,
$$
and $T_\nu$, which is defined in the same way with ${\wt\sigma_{I_\kappa(P)}}$ replaced by $\nu|_{S(P)}$.
Observe that
\begin{align*}
I(f,g)& \leq c\,\|T_{\wt\sigma}(|f|) + T_\nu(|g|)\|_{L^2({\wt\sigma}+\nu)}^2 \\
&\leq 
2c\,\|T_{\wt\sigma}(|f|)\|_{L^2({\wt\sigma}+\nu)}^2 + 2c\,\|T_\nu(|g|)\|_{L^2({\wt\sigma}+\nu)}^2 \\
&= 
4c\,\|T_{\wt\sigma}(|f|)\|_{L^2({\wt\sigma})}^2 + 4c\,\|T_\nu(|g|)\|_{L^2(\nu)}^2,
\end{align*}
where, for the last equality, we took into account that both $T_{\wt\sigma}(|f|)$ and $T_\nu(|g|)$ are constant on 
$I_\kappa(P)\cup S(P)$ and that ${\wt\sigma}(I_\kappa(P))=\nu(S(P))$ for all $P\in\AZ_0$.

To complete the proof of \rf{eqkey1} it is enough to show that $T_{\wt\sigma}$ is bounded in $L^2({\wt\sigma})$ and 
$T_\nu$ in $L^2(\nu)$. We only deal with $T_{\wt\sigma}$, since the arguments for $T_{\nu}$ are
analogous. We argue by duality again. So we consider non-negative functions $f,h\in L^2({\wt\sigma})$ and we write
\begin{align*}
\int T_{\wt\sigma}(f)\,h\,d{\wt\sigma} &= \int \left(\sum_{P\in\AZ_0:P\neq Q(z)}\frac{\ell(P)}{D(P,Q(z))^{n+1}}\,\int_{P}
f\,d{\wt\sigma}\right)\,h(z)\,d{\wt\sigma}(z)\\
&\lesssim
\sum_{P\in\AZ_0}\ell(P) \int_{P}f\,d{\wt\sigma}  \int_{\R^{n+1}\setminus P}\frac1{\bigl(\dist(z,P) + \ell(P))^{n+1}}\,
\,h(z)\,d{\wt\sigma}(z).
\end{align*}
From the growth of order $n$ of ${\wt\sigma}$, it follows easily that
$$
\int_{\R^{n+1} \setminus P}\frac1{(\dist(z,P) + \ell(P))^{n+1}}\,
\,h(z)\,d{\wt\sigma}(z)\lesssim \frac{1}{\ell(P)}\,M_{\wt\sigma} h(y)\quad \mbox{ for all $y\in P$,}
$$
 where $M_{\wt\sigma}$ stands for the (centered) maximal Hardy-Littlewood operator (with respect to ${\wt\sigma}$).
 Then we deduce that
$$\int T_{\wt\sigma}(f)\,h\,d{\wt\sigma}\lesssim \sum_{P\in\AZ_0} \int_{P}f(y)\,M_{\wt\sigma} h(y)\,d{\wt\sigma}(y)\lesssim
\|f\|_{L^2({\wt\sigma})} \|h\|_{L^2({\wt\sigma})},$$
by the $L^2({\wt\sigma})$ boundedness of $M_{\wt\sigma}$. Thus $T_{\wt\sigma}$ is bounded in $L^2({\wt\sigma})$.
\end{proof}

\vv

\begin{proof}[\bf Proof of the Main Lemma \ref{mainlemma}]
By Lemmas \ref{lemAD}, \ref{lemrieszeta}, and the Nazarov-Tolsa-Volberg theorem of 
\cite{NToV},  $\zeta$ is a uniformly $n$-rectifiable measure. So it only remains to note that
the set $\Gamma:=\supp\zeta$ satisfies the required properties from the Main Lemma: it is uniformly $n$-rectifiable and, by \rf{eqsss39} and recalling that $\wt F\subset \Gamma$ and $\sigma = \mu|_{Q_0}$, we have
$\mu(\Gamma) \geq \mu(\wt F) =\sigma(\wt F) \geq \frac{\ve_0}2\,\mu(Q_0).$
\end{proof}

%\enlargethispage{1cm}

\vvv

\end{document}